\documentclass[a4paper,11pt,oneside  ]{article}	

\usepackage{amsmath}

\usepackage{amssymb}
\usepackage{enumerate}
\usepackage{esint}
\usepackage{amsmath}
\usepackage{amsthm}
\usepackage{amssymb}
\usepackage{authblk}
\usepackage{enumerate}
\usepackage{fullpage}
\usepackage{esint}

\theoremstyle{plain}
\newtheorem{theorem}{Theorem}[section]
\newtheorem{corollary}[theorem]{Corollary}
\newtheorem{lemma}[theorem]{Lemma}
\newtheorem{proposition}[theorem]{Proposition}
\newtheorem{remark}{Remark}

\theoremstyle{definition}

\newtheorem{assumption}[theorem]{Assumption}

\theoremstyle{remark}

\newcommand{\N}{\mathbb{N}}
\newcommand{\R}{\mathbb{R}}

\newcommand{\Q}{\mathbb{Q}}
\newcommand{\Z}{\mathbb{Z}}
\newcommand\calA{\mathcal{A}}
\newcommand\calT{\mathcal T}

\newcommand\calNN{\mathcal{NN}}
\newcommand\calI{\mathcal I}
\newcommand\calE{\mathcal E}
\newcommand\calL{\mathcal L}
\newcommand\calF{\mathcal F}
\newcommand\e{\varepsilon}

\newcommand\dist{\operatorname{dist}}

\newcommand{\bb}{{\operatorname{b}}}
\newcommand{\ee}{{\operatorname{e}}}

\newcommand{\wto}{\rightharpoonup}

\newcommand\ho{{\operatorname{hom}}}

\newcommand\loc{{\operatorname{loc}}}
 
\newcommand{\step}[1]{\medskip\noindent\textit{Step #1. }}

\title{Stochastic homogenization of nonconvex discrete energies with degenerate growth}
\date{\today}

\author[1]{Stefan Neukamm\thanks{stefan.neukamm@tu-dresden.de}}
\author[1]{Mathias Sch\"affner\thanks{mathias.schaeffner@tu-dresden.de}}
\author[2]{Anja Schl\"omerkemper\thanks{anja.schloemerkemper@mathematik.uni-wuerzburg.de}}
\affil[1]{Department of Mathematics, Technische Universit\"at Dresden}
\affil[2]{Institute of Mathematics, University of W\"urzburg}

\date{January 24, 2017}
\begin{document}

\maketitle

\begin{abstract}
 Recently, there has been considerable effort to understand periodic and stochastic homogenization of elliptic equations and integral functionals with degenerate growth, as well as related questions on the effective behavior of conductance models in degenerate, random environments. In the present paper we prove stochastic homogenization results for nonconvex energy functionals with degenerate growth under moment conditions. In particular, we study the continuum limit of discrete, nonconvex energy functionals defined  on crystal lattices in dimensions $d\geq 2$. We consider energy functionals with random (stationary and ergodic) pair interactions; thus our problem corresponds to a stochastic homogenization problem. In the non-degenerate case, when the interactions satisfy a uniform $p$-growth condition, the homogenization problem is well-understood. In this paper, we are interested in a degenerate situation, when the interactions neither satisfy a uniform growth condition from  above nor from below. We consider interaction potentials that obey a $p$\mbox{-}growth condition with a random growth weight $\lambda$. We show that if $\lambda$ satisfies the moment condition $\mathbb E[\lambda^\alpha+\lambda^{-\beta}]<\infty$ for suitable values of $\alpha$ and $\beta$, then the discrete energy $\Gamma$\mbox{-}converges to an integral functional with a non-degenerate energy density. In the scalar case, it suffices to assume that $\alpha\geq 1$ and $\beta\geq\frac{1}{p-1}$ (which ensures the non-degeneracy of the homogenized energy density). In the general, vectorial case, we additionally require that $\alpha>1$ and $\frac{1}{\alpha}+\frac{1}{\beta}\leq \frac{p}{d}$. 
\end{abstract}

\tableofcontents

\section{Introduction}

Stochastic homogenization of uniformly elliptic systems and integral functionals with uniform growth is classic and well understood. In contrast, homogenization of systems and (vectorial) integral functionals with degenerate growth is more recent and the level of understanding is still incomplete, see discussion below. Gaining a deep understanding of the degenerate case is of mathematical interest in its own and moreover important for problems and applications in different fields: For example, discrete (convex \& quadratic) integral functionals with degenerate growth are intensively studied in stochastic analysis, namely, as Dirichlet forms associated with random conductance models (see \cite{Biskup} for a survey). In that context homogenization corresponds to the validity of an invariance principle for a random walk in a random environment.  Another application is atomistic modelling of materials, where homogenization of discrete integral functionals  corresponds to an ansatz-free derivation of a continuum model via a discrete-to-continuum limit; in particular, models for rubber may lead to integral functionals with degenerate growth properties, e.g. see \cite{Gloria,KuhnGrun}. 

In this contribution we study stochastic homogenization of discrete energy functionals with nonconvex random pair interactions of  degenerate $p$-growth and finite range of dependence. For simplicity, in the introduction we restrict to the lattice graph $\calL=\Z^d$ with edge set $\calE=\{\ee=[z,z+e_i]\ |\ z\in\Z^d, i=1,\dots,d\}$ and $e_1,\ldots,e_d$ being the canonical basis in $\R^d$. Later we consider complex hyper-cubic lattices and interactions of finite range. For $0<\e\ll1$ consider the energy functional
\begin{equation}\label{minproblem}
 H_\e(\omega;u,A):=\e^d\!\!\!\!\sum_{\ee\in\e\calE\cap A}\!\!\! V\left(\omega;\tfrac{\ee}\e,\nabla u(\ee)\right),
\end{equation}
where $A$ denotes a domain in $\R^d$, $u:\e\calL\to \R^n$ is a state variable, $\nabla u(\ee)$ denotes a discrete directional derivative, cf. \eqref{def:findif}, $V(\omega;\cdot,\cdot): \calE \times \R^n \to [0,\infty)$ denotes a random  interaction potential, and $\omega$ stands for a configuration sampled from a stationary and ergodic law. 

In \cite{ACG11}, a general $\Gamma$-convergence result for functionals of the form \eqref{minproblem} is proven under the assumption that the interaction potential $V(\omega;\cdot,\cdot)$ is non-degenerate, i.e., satisfies a uniform $p$-growth condition (see also \cite{AC04,BS13} for the non-random case). The key point and novelty of the present paper is to consider random potentials with degenerate growth: We suppose that
$$\lambda(\omega;\ee)(\tfrac1c|r|^p-c) \leq V(\omega;\ee,r)\leq c(\lambda(\omega;\ee)|r|^p+1),$$
where $\lambda(\omega;\cdot):\calE\to[0,\infty)$ denotes a (stationary and ergodic) random weight function that satisfies certain moment conditions, cf.\  Assumptions~\ref{H} and \ref{A:2T}. Roughly speaking, we assume that for exponents $\alpha,\beta$ satisfying
\begin{equation}\label{intro:ass0}
 \alpha\geq 1,\quad\beta\geq\frac1{p-1} 
\end{equation}
we have
\begin{equation*}
  \lim_{\e\downarrow0}\frac{\e^d}{|Q|}\sum_{\ee\in\e\calE\cap Q}\left(\lambda(\omega;\tfrac{\ee}\e)^\alpha+\lambda(\omega;\tfrac{\ee}\e)^{-\beta}\right)<\infty,
\end{equation*}
where $Q$ denotes an arbitrary cube in $\R^d$. Note that this allows for interactions that are degenerate in the sense that we might have $\inf_{\ee\in\calE}\lambda(\omega;\ee)=0$ and $\sup_{\ee\in\calE}\lambda(\omega;\ee)=\infty$ with positive probability.
\medskip

Our main result proves that $H_\e(\omega;\cdot,A)$ $\Gamma$-converges (for almost every $\omega$) to a deterministic, homogeneous integral functional, whose limiting energy density (thanks to \eqref{intro:ass0}) is non-degenerate and satisfies a standard $p$-growth condition, cf.~Theorem~\ref{T:1} and Corollary~\ref{C:1}.  For the convergence result in the scalar case (i.e.\ $n=1$), it suffices to combine the moment condition \eqref{intro:ass0} with an additional mild ``convexity at infinity'' assumption on the interaction potential $V(\omega;\cdot,\cdot)$, cf.~Assumption~\ref{A:2T}~(b). In the vectorial case (i.e.\ $n>1$), we need additional moment conditions, namely   
\begin{equation}\label{intro:ass2}
\alpha>1,\quad\frac1\alpha+\frac1\beta\leq\frac{p}d.
\end{equation}
On a technical level, the main difference between the scalar and the vectorial case shows up in a gluing construction (see Lemmas~\ref{L:glue}~and~\ref{L:glue:convex}), which we use e.g. to modify boundary values in the construction of recovery sequences. In the scalar case, the gluing construction relies on a truncation argument, which is a variant of a construction in \cite{CS79}, where periodic homogenization is studied. In the vectorial case the truncation argument is not available. Instead, we impose relation \eqref{intro:ass2} from which we deduce a Rellich-type compact embedding, see Lemma~\ref{L:interpolation2}. In the critical case $\frac{1}{\alpha}+\frac{1}{\beta}=\frac{p}{d}$ the Rellich-type argument is subtle. Its proof combines a weighted Poincar\'e inequality, a two-scale argument and input from ergodic theory.
\smallskip

The study of homogenization problems of differential equations (see e.g. the classical monograph \cite{Bensoussan} and the references therein) and integral functionals has a long history. Periodic (continuum) homogenization of convex and nonconvex integral functionals was studied in the seminal works \cite{Braides86, Marcellini77, Mueller87}. In \cite{DMM86, MM94} these results were extended to the stationary and ergodic, random case by appealing to the subadditive ergodic theorem in \cite{AK81}. The interest in discrete-to-continuum $\Gamma$-limits for energies with nonconvex interactions is more recent, see e.g. \cite{AC04, ACG11,BC,Braides-Gelli-2000, BS13, KLR,SSZ11, SS15b}. 

Homogenization of divergence form elliptic equations and of integral functionals with degenerate growth in the above sense have basically been studied in two situations:
\begin{enumerate}[(i)]
\item The weight functions are of \textit{Muckenhoupt} class, see e.g. \cite{BT04,DASC92,DASC94,EPPW06}. These works exploit the existence of Sobolev inequalities in spaces with Muckenhoupt weights.
\item The weight functions satisfy moment conditions, see e.g. \cite{CS79} where periodic, scalar, convex integral functionals are studied under the assumption that the periodic weight function $\lambda$ satisfies $\lambda,\lambda^{-\frac{1}{p-1}}\in L_{\loc}^1(\R^d)$. This matches precisely our assumptions in the random and discrete setting, cf.\ \eqref{intro:ass0}. The analysis in \cite{CS79} (and related results for degenerate elliptic equations in divergence form, e.g. \cite{ZP08}, or for unbounded integral functionals with convex growth \cite{DG15}) relies on truncation methods.
\end{enumerate}
Recently, (discrete) elliptic equations with degenerate growth have also been extensively studied in the context of invariance principles for random walks in degenerate random environments, see e.g.\ \cite{ADS15,BM15,CD15}, or on percolation clusters, see e.g.\ \cite{BB07,MP07,SS04}; see also \cite{Biskup} for an overview. In particular, in \cite{ADS15,CD15} a quenched invariance principle is obtained under the moment condition $\frac1\alpha+\frac1\beta<\frac2d$, cf.~\eqref{intro:ass2}  with $p=2$. The latter is used to establish a weighted Sobolev inequality, which is needed to implement a  Moser iteration. The same moment condition and the topic of regularity is also addressed in \cite{Bella-Otto}, which discusses a Liouville property for elliptic systems with degenerate, stationary and ergodic coefficients. A similar problem is addressed in \cite{LNO16}, where regularity results for the corrector in stochastic homogenization are proven in a percolation like situation. In the very recent work \cite{AD16}, quantitative homogenization and large scale regularity results have been established for (discrete) elliptic equations on the supercritical (Bernoulli bond) percolation cluster. 

In a wider context, homogenization of non-convex vectorial problems with degenerate growth is studied in \cite{BG95}, where soft and stiff inclusion are considered in a periodic setting, see also the monograph~\cite{JKO94} for more on this and other 'non-standard' (periodic and stochastic) homogenization problems. 
\smallskip

The results of the present paper are novel in various aspects.
\begin{itemize}
 \item To our knowledge it is the first stochastic homogenization result for nonconvex vectorial (discrete) energy functionals with degenerate growth from below and above. We are not aware of any analogous result in the continuum case; we will extend the proofs of this paper to the continuum case in an upcoming work. 
 \item As mentioned above, in the convex/monotone operator case stochastic homogenization with degenerate growth is considered under the assumption that the weight functions are of Muckenhoupt class for every realization, cf.~\cite{EPPW06}. Notice that the moment assumptions considered in our paper are less restrictive in the scalar case while they cannot directly be compared in the vectorial setting.
 \item We would like to emphasize that similar to \cite{ADS15,ADS16,CD15} and \cite{Bella-Otto}, we use the relation $\frac1\alpha+\frac1\beta\leq \frac{p}{d}$, cf.~\eqref{intro:ass2}, in order to obtain a weighted Poincar\'e inequality with a constant which can be controlled by the ergodic theorem. It turns out that in order to show the invariance principles in \cite{ADS15,ADS16,CD15} or the regularity result of \cite{Bella-Otto} the strict inequality is needed, but, as we prove, equality is sufficient to prove our homogenization result.  
\end{itemize}
\smallskip

The main focus of our paper is to understand the general stationary and ergodic case in stochastic homogenization of nonconvex discrete energies. In the scalar case our moment conditions are optimal for homogenization towards a \textit{non-degenerate} integral functional, see Remark~\ref{rem:optimal}. If we replace ergodicity by the strong assumption of independent and identically distributed weight functions $\lambda(\omega;\ee)$, then the moment condition on $\lambda(\omega;\ee)^{-1}$ can be significantly weakened by appealing to ideas from \cite{MO16} and  \cite{ADS16}. We describe this for a rather particular case in Section~\ref{sec:iid}.

\medskip

This article is organized as follows: In the next section we give the precise definition of our discrete energy. In Section~3 we state our main result, cf. Theorem~\ref{T:1} and Corollary~\ref{C:1},  and present a detailed summary of its proof. In Section~4 we present the proof of Theorem~\ref{T:1} and several auxiliary results.

\subsection*{Notation}
For convenience of the reader we list some of the notation used in the paper:
\begin{itemize}
\item $d\geq2$ dimension of the domain of $u$; $n\geq 1$ dimension of the codomain.
\item $\calL$ (and $\calE)$ stand for the set of vertices $x,y,z,\ldots$ (and edges $\bb, \ee,\ldots$),  see Section~\ref{sec:2}.
\item $\calNN$ stands for the subset of $\calE$ for which we impose moment conditions on the growth from below, see Section~\ref{sec:2}.
\item $\calE_0$ and $\calNN_0$ are defined in \eqref{def:E0} and \eqref{def:NN0}.
\item $[x_\ee,y_\ee]$ our notation for an (oriented) edge $\ee\in\calE$.
\item For $A\subset\R^d$ and $\calE'\subset\calE$ we write $\ee=[x_\ee,y_\ee]\in\calE'\cap A$ if $\ee\in\calE'$ and $x_\ee,y_\ee\in A$.
\item $Y=[0,1)^d$ denotes the fundamental region of the lattice.
\item $\calA^\e, \calA^\e_g(A), \calA_{\#}(kY)$ denote function spaces of piecewise affine functions (subordinate to $\e\calL$ or $\calL$, respectively), see \eqref{def:calA}, \eqref{def:Aepsg} and \eqref{def:Aper}.
\item $\partial_\bb^\e$ stands for a discrete directional derivative, see \eqref{def:findif2}.
\item For every $A\subset\R^d$ and $\delta>0$, we define 
  \begin{align*}
    (A)_{\delta}:=\{x\in\R^d\ |\ \dist(x,A)<\delta\},\qquad (A)_{-\delta}:=\{x\in A\ |\ \dist(x,\R^d\setminus A)>\delta\}.
  \end{align*}
  In particular, we use this notation in connection with a parameter $R\geq 1$, the \textit{range of interaction} defined in \eqref{def_of_R}.
\item For $A,B\subset\R^d$ we write $A\Subset B$ if $\bar A$ is compact and contained in the open interior of $B$.
\item We follow the convention that $\frac{\beta}{\beta+1}p=(1-\frac{1}{\beta+1})p=p$ if $\beta=\infty$.
\end{itemize}

\section{Setting of the problem}\label{sec:2}
\paragraph{The graph $(\calL,\calE)$ and the edge set $\calNN$}
Fix $d\geq2$. We consider locally finite, connected, $\Z^d$-periodic graphs with oriented edges. More precisely, we assume that:
\begin{itemize}
 \item The vertex set $\calL\subset\R^d$ has the form $\calL=\bigcup_{i=1}^k(q_i+\Z^d),$ where $k\in\N$ and $q_1,\dots,q_k\in Y:=[0,1)^d$ with $q_1=0$ and $q_i\neq q_j$ for $i\neq j$, (i.e.\ $\calL$ is a crystal lattice).
 \item The edge set $\calE\subset \{[x,y] | x,y\in \calL,x\neq y\}$ is $\Z^d$-periodic and locally finite, i.e.\ $\ee\in\calE$ implies $\ee+\Z^d\subset\calE$ and for all $x\in\calL$ the set $\{[x,y]\in \calE\}$ is finite.
\end{itemize}
For examples see the end of Section~\ref{sec:2}. Given a subset $\calE'\subset\cal E$ and two vertices $x,y$, we say $(x,y)$ are connected in $(\calL,\calE')$, if there exists a path $\ell=(x_0,\ldots,x_m)$ of finite length $m\in\N$ such that $x_0=x$, $x_m=y$ and for all $i=1,\ldots,m$ we have either $[x_{i-1},x_i]\in\calE'$ or $[x_{i},x_{i-1}]\in\cal E'$ (i.e.\ we ignore the orientation of the edges when speaking about connectedness). We say that $(\calL,\calE')$ is connected, if any two distinct vertices $x,y\in\calL$ are connected in $(\calL,\calE')$. We assume that:
\begin{itemize}
 \item $(\calL,\calE)$ is connected.
\end{itemize}
The edge set $\calE$ can be written as a direct sum of the lattice $\Z^d$ and the generating edge set
\begin{equation}\label{def:E0}
  \calE_0:=\{[x,y]\in\calE\ |\ x\in Y\}.
\end{equation}
Indeed, for every $\ee\in\calE$ there exists a unique pair $(z_\ee,\bb_\ee)\in\Z^d\times \calE_0$ such that $\ee=z_\ee+\bb_\ee$. 
\smallskip

In addition to $\calE$ we consider another, possibly smaller set of edges, which we denote by $\calNN$ and which is fixed from now on. Let us anticipate that we are going to assume moment conditions on the growth from below only for the potentials associated with edges in $\calNN$. We suppose that
\begin{itemize}
\item $\calNN\subset\calE$ is $\Z^d$-periodic and the subgraph $(\calL,\calNN)$ is connected.
\end{itemize}
A typical choice of $\calNN$ is the set of nearest-neighbour edges. Note that our assumptions allow for more general choices of $\calNN$. We set
\begin{equation}\label{def:NN0}
\calNN_0:=\calE_0\cap\calNN.
\end{equation}
\paragraph{Discrete derivative} We introduce discrete derivatives for state variables defined on the scaled lattice $\e\calL$ with scaled edge set $\e\calE$.
For a function $u:\e\calL\to\R^n$ and $\ee\in\e\calE$, we define the discrete directional derivative as
\begin{equation}\label{def:findif}
 \nabla u(\ee):=\frac{u(y_\ee)-u(x_\ee)}{|y_\ee-x_\ee|},
\end{equation} 
where $x_\ee,y_\ee\in\e\calL$ are the unique vertices with $\ee=[x_\ee,y_\ee]$. Note that $\nabla u(\ee)\in \R^n$.
For $z\in\e\mathcal L$, $\bb\in\calE_0$ and $u:\e\calL\to\R^n$ we introduce the notation
\begin{equation}\label{def:findif2}
 \partial_\bb^\e u(z):=\nabla u(z+\e \bb).
\end{equation}

For our purpose it is convenient to identify $u:\e\calL\to\R^n$ with a canonical piecewise affine interpolation. To that end, we fix a triangulation of $\calL\cap\overline Y$ and denote by $\calT$ its $\Z^d$-periodic extension. We set
\begin{equation}\label{def:calA}
 \calA^\e:=\{u\in C(\R^d,\R^n):u\mbox{ is affine on $T$ for each element $T\in\e\calT$}\}.
\end{equation}
In this paper we tacitly identify $u:\e\calL\to\R^n$ with its unique interpolation in $\calA^\e$. As usual, $\nabla u$ denotes the weak derivative of a function $u: \R^d \to \R^n$ with $\nabla u: \R^d \to \R^{n\times d}$.
 Note that we have the following elementary relation between norms of discrete gradients and the corresponding norms of the associated piecewise affine interpolation:
\begin{lemma}\label{L:sumint}
For $1\leq q<\infty$ there exists $c_0\in(0,\infty)$ such that
\begin{align}\label{est:AB}
\forall u\in\calA^1:\, \int_Y|\nabla u|^q\,dx\leq c_0\sum_{\ee\in\calNN\cap B_{\frac{R}4}(0)}|\nabla u(\ee)|^q
\end{align}
and for all bounded Lipschitz domains $A\subset \R^d$ and $\e>0$ it holds
\begin{align}\label{est:sumint}
 \forall u\in\calA^\e:\, \int_{(A)_{-\e R}}|\nabla u|^q\,dx \leq c_0\e^d\!\!\!\!\sum_{z\in\e\Z^d\cap A}\sum_{\bb\in\calNN_0}|\partial_\bb^\e u(z)|^q,
\end{align}
with $R\geq 1$ defined in \eqref{def_of_R} below.
\end{lemma}
(For the convenience of the reader we present the elementary proof in Appendix~\ref{appendix}.)
\medskip

At various places in the paper we consider functions $u\in\calA^\e$ defined on a Lipschitz domain $A\subset\R^d$ subject to \textit{Dirichlet boundary conditions}. With the latter we mean that $u\in\mathcal A^\e$ is prescribed on all vertices outside of $A$ and in a boundary layer of thickness proportional to $\e$ (i.e.\ on all vertices that ``interact'' with vertices outside of $A$). For the precise definition we introduce the \textit{range of interaction,} 
\begin{equation}
 \label{def_of_R}
    R\text{ defined as the  smallest number with the following properties:}
\end{equation}
\begin{enumerate}[(a)]
 \item the graph with vertex set $\calL\cap[0,1]^d$ and edge set $\calNN\cap B_{\frac{R}{4}}(0):=\{[x,y]\in\calNN\,:\,|x|,|y|<\frac{R}{4}\,\}$ is connected,
 \item $R\geq \max\{|y|\,:\,[x,y]\in\calE_0\}$.
\end{enumerate}
For $g\in W^{1,\infty}_\loc(\R^d,\R^n)$ and $A\subset \R^d$ we set
\begin{equation}\label{def:Aepsg}
 \mathcal A^\e_g(A):=\{\,u\in\mathcal A^\e\,:\,u=g\text{ in }\R^d\setminus(A)_{-\e R}\,\}.
\end{equation}
\begin{remark}\label{Remark:finiterange0}
Note that $R$ is a finite number that only depends on $(\calL,\calNN)$. By (a), we have $R\geq 4\sqrt d >1$. Furthermore, by \eqref{def:findif2} and (b) we have for any $z\in\Z^d$ and $u,v\in\calA^\e$:
\begin{equation}\label{rem:finiterange}
 u= v\quad\mbox{in $B_{\e R}(z)$}\quad\Rightarrow\quad \partial_\bb^\e u(z)=\partial_\bb^\e v(z)\mbox{ for all $\bb\in\calE_0$.}
\end{equation}
\end{remark}

\paragraph{Stationary and ergodic interaction potentials} We consider random interaction potentials $\{V(\omega;\ee,\cdot)\}_{\ee\in\calE}$, $V(\omega;\ee,\cdot):\R^n\to[0,\infty)$, that are \textit{statistically homogeneous and ergodic}, and \textit{continuous in their last argument}. We phrase this assumption by appealing to the language of ergodic, measure preserving dynamical systems (which is a  standard in the theory of stochastic homogenization, see e.g. the seminal paper \cite{PV79}): Let $(\Omega,\calF,\mathbb P)$ denote a probability space and $\tau=(\tau_z)_{z\in\Z^d}$ a family of measurable mappings $\tau_z:\Omega\to\Omega$ satisfying
  \begin{itemize}
  \item (group property) $\tau_0\omega=\omega$ for all $\omega\in\Omega$ and $\tau_{x+y}=\tau_x\tau_y$ for all $x,y\in\Z^d$.
  \item (stationarity) For every $z\in\Z^d$ and $B\in\calF$ it holds $\mathbb P(\tau_z B)=\mathbb{P}(B)$.
  \item (ergodicity) All $B\in\calF$ with $\tau_z B=B$ for all $z\in\Z^d$ satisfy $\mathbb P(B)\in\{0,1\}$.
  \end{itemize}
  For each $\bb\in\calE_0$ let the potential $V_\bb:\Omega\times \R^n\to[0,\infty)$, $(\omega,r)\mapsto V_{\bb}(\omega;r)$ be measurable in $\omega$ and continuous in $r$.
  We assume that the random interaction potentials $\{V(\omega;\ee,\cdot)\}_{\ee\in\calE}$ take the form
  \begin{equation*}
    V(\omega;\ee,\cdot)=V_{\bb_\ee}(\tau_{z_\ee}\omega;\cdot)\qquad\text{for }\ee=z_\ee+\bb_\ee\in\calE\text{ and all }\omega\in\Omega.
  \end{equation*}
  Note that by this construction, for any finite set of edges $\ee_1,\ldots,\ee_m\in\calE$ and all $r\in\R^n$ the joint distribution of the random variables $V(\omega;\ee_1+z,r),\ldots, V(\omega;\ee_m+z,r)$ does not depend on the shift $z\in\Z^d$. See the end of this section for examples.

\paragraph{Energy functional}
 For given $\e>0$ and $A\subset\R^d$, we define the energy of the discrete system $E_\e(\cdot;\cdot,A):\Omega\times L_\loc^1(\R^d,\R^n)\to[0,\infty]$ by
\begin{equation}\label{ene:pair}
 E_\e(\omega;u,A):=\begin{cases}
                    \displaystyle{\e^d\!\!\!\!\sum_{z\in\e\Z^d\cap A}\sum_{\bb\in \calE_0}}V_\bb(\tau_{\frac{z}\e}\omega;\partial_\bb^\e u(z))&\mbox{if $u\in\calA^\e$,}\\
                    \infty&\mbox{otherwise.}
                   \end{cases}
\end{equation}
Note that \eqref{minproblem} coincides with  $E_\e(\omega;u,A)$ for $u\in\calA^\e$ up to a small difference at the boundary: In \eqref{minproblem} we sum over all edges in $\e\calE$ that connect vertices in $A$, while in the definition of $E_\e$ we sum over all edges $\ee\in \e\calE$ of the form $\ee=z+\e\bb$ with $z\in\e\Z^d\cap A$ and $\bb\in\calE_0$. For the upcoming analysis it is more convenient to work with $E_\e$. However, as shown in Corollary~\ref{C:1} our results also hold for $H_\e$.

\paragraph{Moment conditions}
We will prove a homogenization result for the energy $E_\e$ defined above under the following growth conditions on the interaction potentials $V_\bb$: 
\begin{assumption}\label{H}
There exist $1<p<\infty$, a finite constant $c_1>0$, exponents
\begin{equation}\label{alphabeta0}
1 \leq \alpha \leq \infty,\quad\tfrac1{p-1}\leq \beta \leq \infty,
\end{equation}
and random variables $\lambda_\bb:\Omega\to[0,\infty)$ (for $\bb\in\calE_0$) such that the following properties hold:
\begin{itemize}
\item ($p$-growth condition)
\begin{align}
  \lambda_\bb(\omega) (\frac1{c_1}|r|^p-{c_1})\leq V_\bb(\omega;r)\leq {c_1}(1+\lambda_\bb(\omega) (|r|^p+1)),\label{ass:V:1}
\end{align}
\item (moment condition)
  \begin{equation}
    \begin{aligned}
      \forall\bb\in\calE_0\,:\,&
      \left\{\begin{aligned}
          \mathbb E[\lambda_\bb^\alpha]&<\infty&&\text{if }\alpha<\infty,\\
          \sup_{\Omega}\lambda_\bb&<\infty&&\text{if }\alpha=\infty,
        \end{aligned}\right.\\
      \forall\bb\in\calNN_0\,:\,&
      \left\{\begin{aligned}
          \mathbb E[(\tfrac{1}{\lambda_\bb})^{\beta}]&<\infty&&\text{if }\beta<\infty,\\
          \sup_{\Omega}\tfrac{1}{\lambda_\bb}&<\infty&&\text{if }\beta=\infty,
        \end{aligned}\right.
    \end{aligned}      
    \label{ass:X0}        
  \end{equation}
  where $\mathbb E$ denotes the expected value.
\end{itemize}      

\end{assumption}
Note that \eqref{alphabeta0} in particular implies $\frac{\beta}{\beta+1}p\geq 1$, where here and below we follow the convention that $\frac{\beta}{\beta+1}p=(1-\frac{1}{\beta+1})p=p$ if $\beta=\infty$. We do not assume any quantitative continuity of the interaction potentials (like Lipschitz continuity), see Remark~\ref{R:lipschitz} for a further discussion. Moreover, in \eqref{ass:X0} the moment conditions on the growth from below are only imposed for edges in $\calNN_0$. 
The above assumptions ensure coercivity of the energy $E_\e$ and that the homogenized energy density defined below satisfies a non-degenerate $p$-growth condition, see Lemma~\ref{L:indwhom} and Lemma~\ref{L:coercivity}. However, we need some additional assumptions on the interaction potentials for our homogenization result: either we impose stronger moment conditions than \eqref{alphabeta0} or we consider the scalar case only and impose some mild \textit{convexity at infinity} condition on the interaction potentials $V_\bb$.
\begin{assumption}\label{A:2T}
 One of the following statements are true:
 \begin{itemize}
  \item[(A)] (vectorial case) The exponents $\alpha,\beta$ in \eqref{alphabeta0} satisfy also \eqref{intro:ass2}.
 \item[(B)] (scalar case) $n=1$ and for all $\bb\in\calE_0$ there exists a function $f_\bb:\Omega\times \R\to[0,\infty)$ and constants $c_2<\infty$, $q\in(1,p)$ such that  
  \begin{equation*}
 \begin{split}
 \forall \bb\in \calE_0:\quad&V_\bb(\omega;\cdot)+f_\bb(\omega;\cdot):\R\to[0,+\infty)\mbox{ is convex},\\
 &f_\bb(\omega;r)\leq c_2(1+\lambda_\bb(\omega)(|r|^q+1)).
\end{split}
 \end{equation*}
 \end{itemize}

\end{assumption}

\paragraph{Examples}

\begin{itemize}
 \item \textit{(Lattices)}. A typical example for a graph $(\calL,\calE)$ satisfying the above assumptions is given by the hyper-cubic lattice with finite range interactions, i.e.\ $\calL=\Z^d$, $\calE=\{[x,y]\, |\, x,y\in\Z^d,\, 0<|x-y|<R'\}$ for some $R'\geq1$, and $\calNN=\{[x,x+e_i]\, |\, x\in \Z^d,i=1,\dots,d\}$ the set of nearest-neighbour edges. Notice that, by a change of variables, we may treat arbitrary Bravais lattices, see e.g.\ \cite[Example 5.1 and Example 5.2]{Braides-Gelli-2000}.
\smallskip

We also allow for more complicated lattices: For example, up to an affine transformation, the \textit{Kagome-lattice} in $\R^2$ is given by the set of vertices $\calL=\cup_{i=1}^3(q_i+\Z^2)$ where $q_1=0$, $q_2=\frac12e_1$ and $q_3=\frac12e_2$,  and $\calE_0:=\{[0,\frac12 e_1],[\frac12e_1,e_1],[0,\frac12 e_2],[\frac12e_2,e_2],\frac12[e_1,e_2],[\frac12e_2,e_2-\frac12e_1]\}$.
 
\item \textit{(Potentials)}. For a given graph $(\calL, \calE)$ we can introduce a canonical probability space as follows: Let $(\widetilde\Omega,\widetilde{\mathcal B})$ denote the open interval $\widetilde\Omega:=(0,\infty)\subset \R$  together with its Borel-$\sigma$-algebra $\widetilde {\mathcal B}:=\mathcal B((0,\infty))$. We denote by $(\Omega,{\mathcal F})=(\widetilde\Omega^\calE,\widetilde{\mathcal B}^{\otimes \calE})$ the $\calE$-fold product measure space. Then for $z\in\Z^d$ the maps $\tau_z:\Omega\to\Omega$, $(\tau_z\omega)(\ee):=\omega(\ee+z)$ form a group of measurable mappings. A simple example of a probability measure on $(\Omega,{\mathcal F})$ that turns $\tau$ into a stationary and ergodic dynamical system is the product measure ${\mathbb P}=\widetilde{\mathbb P}^{\otimes\calE}$, where $\widetilde{\mathbb P}$ is an arbitrary probability measure on $(\widetilde\Omega,\widetilde{\mathcal B})$. In that case, the coordinate projections $\{\omega\mapsto\omega(\bb)\}_{\bb\in\calE}$ are independent and identically distributed random variables.
  
A model family of random potentials $\{V(\omega;\ee,\cdot)\}_{\ee\in\calE}$ is given by
\begin{equation*}
  V(\omega;\ee,\cdot):=\omega(\ee)V_{\bb_\ee}(\cdot)=(\tau_{z_\ee}\omega)(\bb_\ee)V_{\bb_\ee}(\cdot),
\end{equation*}
where $(z_\ee,\bb_\ee)\in\Z^d\times\calE_0$ is uniquely defined by $\ee=z_\ee+\bb_{\ee}$ and where we assume that the finite set of potentials $\{V_{\bb}\}_{\bb\in\calE_0}$ satisfies $V_\bb\in C(\R^n,[0,\infty))$ and a standard $p$-growth condition for some $p>1$. Note that in this model case incidentally we have $\lambda_\bb(\omega)=\omega(\bb)$ for $\bb\in\calE_0$.

Let us mention some more specific examples for $V_\bb$ and conditions that ensure Assumptions~\ref{H} and \ref{A:2T}.
\begin{itemize}
\item Set $n=1$ and $V_\bb(r)=r^2$ for all $\bb\in\calE_0$. This corresponds to a random conductance model, cf.~\cite{Biskup}. In this case, the Assumptions~\ref{H} and \ref{A:2T}~(B) are satisfied for $p=2$ if $\mathbb E[\omega(\ee)]<\infty$ for all $\ee \in\calE_0$ and $\mathbb E[\omega(\ee)^{-1}]<\infty$ for all $\ee\in\calNN_0$.
\item  Set $n=1$ and $V_\bb(r)=(r^2-1)^2$. These potentials are not convex but satisfy Assumption~\ref{A:2T}~(B) for $p=4$ (choose $f_\bb(\omega;r):=2\omega(\bb) r^2$ for $\bb\in\calE_0$). Hence, the Assumptions~\ref{H} and \ref{A:2T}~(B) are satisfied for $p=4$ if $\mathbb E[\omega(\ee)],\mathbb E[\omega(\ee)^{-\frac13}]<\infty$ for $\ee\in\calE_0$.
\item Set $n=d=2$ and consider the lattice graph $(\Z^2,\calE)$ with the generating edge set $\calE_0=\{\pm e_1,\pm e_2,\pm (e_1+e_2),\pm (e_1-e_2)\}$. Choosing $V_\bb(r)=(|r|-1)^2$, $r\in\R^2$, the energy $E_\e(\omega;u,A)$ corresponds to the elastic energy of a deformation $u$ of the weighted lattice $\e\Z^d\cap A$. The Assumptions~\ref{H} and \ref{A:2T}~(A) are satisfied for $p=2$ if $\mathbb E[\omega(\bb)^\alpha]<\infty$ and $\mathbb E[\omega(\bb)^{-\beta}]<\infty$ for all $\bb\in\calE_0$ where $\alpha>1$, $\beta\geq1$ and $\frac1\alpha+\frac1\beta\leq 1$.
\end{itemize}
\end{itemize}
\begin{remark}
  In this work, we consider $\Z^d$-periodic graphs, but the analysis directly extends to general Bravais (multi) lattices. An interesting extension of our analysis is to consider stationary and ergodic random lattices as in \cite{ACG11}. In this context, two settings might be distinguished: If the geometry of the random lattice is regular in the sense of approximation theory (e.g. in the sense of \cite[Definition 13]{ACG11}), we expect that the analysis of this paper can be adapted without major difficulties. On the other hand, it is an interesting and open question to which extend our result is valid for random lattices with ``degenerate geometry'', which e.g. might feature regions with a low or very high density of vertices.
\end{remark}

\section{Main result}

The main theorem of this paper is a homogenization result for the energy $E_\e$, cf. \eqref{ene:pair}. 
\begin{theorem}\label{T:1}
Suppose that Assumptions~\ref{H} and \ref{A:2T} are satisfied. Then there exists a continuous energy density $W_\ho:\R^{n\times d}\to[0,\infty)$ satisfying the standard $p$-growth condition
\begin{equation*}
  \exists {c_1}'>0\,\forall F\in\R^{n\times d}\,:\qquad
 \frac1{{c_1}'}|F|^p-{c_1}'\leq W_\ho(F)\leq {c_1}'(1+|F|^p),
\end{equation*}
and there exists a set $\Omega'\subset\Omega$ with $\mathbb P(\Omega')=1$ such that for all $\omega\in\Omega'$ the following properties hold:\\
For all bounded Lipschitz domains $A\subset \R^d$ the sequence of functionals $(E_\e(\omega;\cdot,A))_\e$ $\Gamma$-converges with respect to the $L^{\frac{\beta}{\beta+1}p}(A)$-topology as $\e\downarrow0$ to the functional $E_\ho(\cdot,A)$ given by
\begin{equation*}
 E_\ho(u,A):=\begin{cases}\displaystyle{\int_A} W_\ho(\nabla u(x))\, dx&\mbox{if $u\in W^{1,p}(A,\R^n)$,}\\ \infty&\mbox{else.}\end{cases}
\end{equation*}
\end{theorem}
In Section~\ref{S:outline} below we outline the proof of Theorem~\ref{T:1}, which is split into several lemmas and propositions. In particular, Theorem~\ref{T:1} directly follows from Proposition~\ref{P:inf}, which yields the $\Gamma$-liminf inequality, and \ref{P:sup}, which asserts the existence of a recovery sequence.

\medskip

The homogenized energy density $W_{\hom}$ can be characterized by means of asymptotic formulas. 
\begin{lemma}\label{L:Whomper}
Consider the situation of Theorem~\ref{T:1}. For all $F\in\R^{n\times d}$ we have 
\begin{equation*}
 W_\ho(F)=\lim_{k\uparrow\infty}\mathbb E\left[W_\ho^{(k)}(\cdot;F)\right],
\end{equation*}
where $W_\ho^{(k)}:\Omega\times\R^{n\times d}\to[0,\infty)$, $k\in\N$, is defined by
\begin{equation}\label{Whomperkdef}
 W_\ho^{(k)}(\omega;F):=\inf\left\{\frac1{k^d}E_1(\omega;g_F+\phi,kY)\ |\ \phi\in\calA_\#(kY)\right\},
\end{equation}
with $g_F(x)=Fx$ for all $x\in\R^d$ and 
\begin{equation}\label{def:Aper}
 \calA_\#(kY)=\{\phi\in\calA^1\, |\, \mbox{$\phi$ is $k\Z^d$-periodic}\}.
\end{equation}
\end{lemma}
(For the proof see Section~\ref{S:Whomper}.)
\medskip

Theorem~\ref{T:1} comes in hand with a compactness statement (cf. Corollary~\ref{C:1} and Lemma~\ref{L:coercivity}), saying that the sequence of energies $E_\e(\omega;\cdot, A)$ (when restricted to functions with prescribed Dirichlet boundary value or functions with prescribed mean) is equicoercive in the weak topology of $W^{1,\frac{\beta}{\beta+1}p}(A)$. Hence, the convergence in Theorem~\ref{T:1} holds for any $L^q$ into which $W^{1,\frac{\beta}{\beta+1}p}$ \textit{compactly} embeds. In fact, we can further improve the convergence, by appealing to the fact that our energy controls a weighted $L^p$-norm of the gradient with a weight that is stationary and ergodic. The upcoming lemma is of particular interest in the critical case $\frac1q=\frac{\beta+1}{\beta}\frac{1}{p}-\frac1d$, i.e. when $W^{1,\frac{\beta}{\beta+1}p}(A)$ is not compactly (but only continuously) embedded into $L^q_{\loc}(A)$:
\begin{lemma}[Compactness]\label{L:compactness}
  Suppose Assumption~\ref{H} is satisfied. Fix $\omega\in\Omega_0$, which is a set of full measure that is introduced below, cf. Remark~\ref{R:Omega0}. Fix a bounded Lipschitz domain $A\subset \R^d$ and consider a sequence $(u_\e)$ satisfying
  \begin{equation}\label{ass:coer}
    \limsup_{\e\downarrow0} E_\e(\omega;u_\e,A)<\infty.
  \end{equation}
  Then $u_\e\wto u$ weakly in $L^1(A,\R^n)$ implies  $u_\e\to u$ strongly in $L^q_{\loc}(A,\R^n)$ for all
    $1\leq q< \infty$ satisfying
    \begin{equation}\label{ass:on:q}
      \left\{\begin{aligned}
          \frac{1}{q}&\geq \frac{\beta+1}{\beta}\frac{1}{p}-\frac{1}{d}&&\text{if }\beta<\infty,\\
          \frac{1}{q}&> \frac1p-\frac{1}{d}&&\text{if }\beta=\infty.
        \end{aligned}\right.
    \end{equation}
\end{lemma}
(For the proof see Section~\ref{S:comp}.)
\medskip

As it is well-known, $\Gamma$-convergence is stable under perturbations by continuous functionals and implies convergence of minima and minimizers under suitable coercivity properties. In the following we make this explicit and establish a compactness and $\Gamma$-convergence result for discrete energies subject to Dirichlet data and with additional \textit{body forces} of the form
\begin{equation*}
  F_\e(u):=\e^{d}\sum_{x\in\e\calL\cap A}f_\e(x)\cdot u(x).
\end{equation*}
We assume that the sequence $f_\e:\e\calL\cap A\to\R^n$ weakly converges to a limit $f\in L^{\frac{q}{q-1}}(A,\R^n)$ in the sense that 
\begin{equation}\label{def:w-conv}
  \left\{\begin{aligned}
      &\lim\limits_{\e\downarrow 0}\e^{d}\sum_{x\in\e\calL\cap A}f_\e(x)\cdot\varphi(x)=\int_A f(x)\cdot\varphi(x)\qquad\text{for all }\varphi\in C(\overline A,\R^n),\\
      &\limsup\limits_{\e\downarrow 0}\e^{d}\sum_{x\in\e\calL\cap A}|f_\e(x)|^{\frac{q}{q-1}}<\infty.
  \end{aligned}\right.
\end{equation}

\begin{corollary}\label{C:1}
In the situation of Theorem~\ref{T:1}, fix $\omega\in\Omega'$, a Lipschitz domain $A\subset\R^d$ and an exponent $1\leq q<\infty$ satisfying \eqref{ass:on:q}. Consider the functional
\begin{equation*}
 J_{\e}(u):=
      \begin{cases}
        H_\e(u)-F_\e(u)&\mbox{if $u\in\mathcal A^\e_g(A)$}\\
        +\infty&\mbox{otherwise},
      \end{cases}
\end{equation*}
where
\begin{equation*}
 H_{\e}(u):=\e^d\sum_{z\in\e\Z^d}\sum_{\bb\in\calE_0\atop z+\e\bb\in \e\calE\cap A} V_\bb(\tau_{\frac{z}\e}\omega;\partial_\bb^\e u(z)).
\end{equation*}

We assume weak convergence of the body forces in the sense of \eqref{def:w-conv}. Then the following properties hold:
    \begin{enumerate}[(a)]
    \item (Coercivity) Any sequence $(u_\e)$ with finite energy, i.e.\
      \begin{equation}\label{ene:bounded:bc}
        \limsup\limits_{\e\downarrow 0}J_{\e}(u_\e)<\infty,
      \end{equation}
      admits a subsequence that strongly converges in $L^q(A,\R^n)$ (and weakly converges in $W^{1,\frac{\beta}{\beta+1}p}(A,\R^n)$) to a limit $u\in g+W^{1,p}_0(A,\R^n)$.
    \item ($\Gamma$-convergence) The sequence $(J_{\e})$ $\Gamma$-converges with respect to strong convergence in $L^q(A)$ to the functional $J_{\ho}$ given by
      \begin{equation*}
        J_{\ho}(u):=\begin{cases}\displaystyle{\int_A} W_\ho(\nabla u(x))-f(x)\cdot u(x)\, dx&\mbox{if $u\in g+W_0^{1,p}(A,\R^n)$,}\\ \infty&\mbox{otherwise,}\end{cases}
      \end{equation*}
      and it holds
      \begin{equation*}
        \liminf_{\e\downarrow0}\min J_{\e}=\min J_{\ho}.
      \end{equation*}
      Moreover, every minimizing sequence $(u_\e)$ strongly converges (up to a subsequence) in $L^q(A,\R^n)$ (and weakly in $W^{1,\frac{\beta}{\beta+1}p}(A,\R^n)$) to a minimizer $u$ of $J_{\ho}$.
    \end{enumerate}
  \end{corollary}
  (For the convenience of the reader, we give a proof in Appendix~\ref{appendix}.)
  
\begin{remark}
 In Corollary~\ref{C:1}, $H_\e$ can be replaced by $E_\e(\omega;\cdot,A)$ without changing the limit, since both energies only differ close to the boundary. We phrase Corollary~\ref{C:1} in terms of $H_\e$, since this allows an easier comparison with existing results in the discrete-to-continuum literature, see \cite{AC04,ACG11}.
\end{remark}

  \begin{remark}[Homogenization of the Euler-Lagrange equation]
    If the potentials $V_{\bb}(\omega;r)$ are strictly convex and smooth in $r$, then the minimizer of $J_\e$ can be characterized as the unique solution $u_\e\in\calA^\e_g(A)$ to the Euler-Lagrange equation
    \begin{equation*}
      \sum_{z\in\e\Z^d}\sum_{\bb\in\calE_0\atop z+\e\bb\in \e\calE\cap A}V_\bb'(\tau_{\frac{z}{\e}}\omega;\partial_\bb^\e u_\e(z))\partial_\bb^\e v(z)=\sum_{x\in\e\calL\cap A}f_\e(x)\cdot v(x)\qquad\forall v\in\calA_0^\e(A).
    \end{equation*}
    Hence, Corollary~\ref{C:1} can be rephrased as a homogenization result for the discrete elliptic equation above. In the following, we consider the quadratic case, i.e.~$V_{\bb}(\omega;r)=\lambda_{\bb}(\omega)r^2$, only. In this situation, the  homogenized energy density turns out to be quadratic; i.e. it can be written in the form $W_{\hom}(F)=\frac{1}{2}F\cdot\mathbb LF$ for some symmetric, strongly elliptic fourth order tensor $\mathbb L$. By Corollary~\ref{C:1} we conclude that $u_\e\to u$ strongly in $L^q(A)$ where $u\in g+W^{1,2}_0(A)$ is the unique weak solution to 
    \begin{equation*}
      -\nabla\cdot\mathbb L\nabla u=f\qquad\text{in }A.
    \end{equation*}
Elliptic systems with degenerate random coefficients are considered in \cite{Bella-Otto} in a continuum setting. In that paper sublinearity of the (extended) corrector (and thus the homogenization result) is established under the assumption $\frac1\alpha+\frac1\beta<\frac2d$, which is stronger than $\frac1\alpha+\frac1\beta\leq \frac2d$ (our assumption of  Corollary~\ref{C:1}). As mentioned earlier, in the scalar case, the continuum and periodic analogue of the above statement is contained in \cite{CS79,ZP08}, but we are not aware of an extension to the random setting. An exception is the scalar case in dimension $d=2$: In this situation the invariance principle proven in \cite{Biskup} implies the above homogenization result.      
  \end{remark}

  \begin{remark}[Optimality]\label{rem:optimal}
    Assumption~\ref{H} is optimal in the following sense: If condition \eqref{alphabeta0}, i.e. $\alpha\geq1$ and $\beta\geq\frac1{p-1}$, is violated, then the homogenized energy density $W_\ho$ might become degenerate. To illustrate this fact we consider the integer lattice $(\Z^d,\mathbb B^d)$, where $\mathbb B^d=\{[z,z+e_i]\, |\, i\in\{1,\dots,d\},\, z\in\Z^d\}$, and consider for $u:\e\Z^d\to\R$ the quadratic energy functional
    \begin{equation}\label{example:energy}
      E_\e(\omega;u,A):=\e^d\sum_{z\in\e\Z^d\cap A}\sum_{i=1}^d\omega(\tfrac{z\cdot e_1}{\e})|\partial_{e_i}^\e u(z)|^2,
    \end{equation}
    where $\{\omega(x)\}_{z\in\Z}$ are independent and identically distributed $(0,\infty)$-valued random variables. Note that the energy describes a layered medium that is constant in any direction different from $e_1$. Since the energy is convex and quadratic, and the medium is layered, the auxilliary energy densities $W_\ho^{(k)}(\omega;F)$, see \eqref{Whomperkdef}, can be calculated explicitly (e.g. by appealing to the associated Euler-Lagrange equation, which factorizes to one-dimensional equations). Indeed, one can show that for all $\omega\in\Omega$, $k\in\N$ and $\ell\in\R$ we have 
    \begin{align*}
      W_\ho^{(k)}(\omega;\ell e_j)=
      \begin{cases}
        \ell^2\left(\frac{1}{k}\sum_{z=0}^{k-1}\frac{1}{\omega(z)}\right)^{-1}&\text{if }j=1,\\
        \ell^2\frac{1}{k}\sum_{z=0}^{k-1}\omega(z)&\text{else.}
      \end{cases}
    \end{align*}  
    If $\mathbb E[\omega(0)]+\mathbb E[\tfrac{1}{\omega(0)}]<\infty$, then Theorem~\ref{T:1} can be applied. In particular, Assumption~\ref{H} is satisfied (note that $p=2$). On the other hand, if  $\mathbb E[\tfrac{1}{\omega(0)}]=\infty$, then 
    $$\lim_{k\uparrow\infty}\mathbb E[W_\ho^{(k)}(\cdot;\ell e_1)]=0,$$
    and we deduce that $W_{\ho}(\ell e_1)=0$ for all $\ell$. Hence, $W_{\ho}$ is degenerate from below. Likewise, if $\mathbb E[\omega(0)]=\infty$, then 
    $$\lim_{k\uparrow\infty}\mathbb E[W_\ho^{(k)}(\cdot;e_2)]=\infty,$$
    and we deduce that $W_{\ho}(e_2)=\infty$, which means that $W_{\ho}$ does not satisfy the growth condition from above.
    
    To check optimality of the exponent $\beta=\frac1{p-1}$ for general $p>1$, it suffices to replace the quadratic term $|\partial_{e_i}^\e u(z)|^2$ in \eqref{example:energy} by the $p$-th power $|\partial_{e_i}^\e u(z)|^p$. For all $\omega\in\Omega$, $k\in\N$ and $\ell\in\R$, it follows 
    $$0\leq W_\ho^{(k)}(\omega;\ell e_1)\leq \ell^p \left(\frac{1}{k}\sum_{z=0}^{k-1}\omega(z)^{-\frac{1}{p-1}}\right)^{-(p-1)}.$$
    Indeed, choose $\phi(z)=\varphi_k(z\cdot e_1)$ in \eqref{Whomperkdef}, where $\varphi_k$ is the $k\Z$-periodic function satisfying $\varphi_k(x)=\ell (\frac1k\sum_{z=0}^{k-1}\omega(z)^{-\frac1{p-1}})^{-1}\sum_{z=0}^{x-1}\omega(z)^{-\frac{1}{p-1}}-\ell x$ for all $x\in\Z\cap[0,k)$. As above, we obtain that $\mathbb E[\omega^{-\frac1{p-1}}]=\infty$ implies $W_\ho(\ell e_1)=0$ for all $\ell\in\R$. 
    
  \end{remark}

\subsection{Outline of the proof of Theorem~\ref{T:1}}\label{S:outline}

A key ingredient in any result on stochastic homogenization is ergodic theory. We rely on two types of ergodic theorems. The first one is Birkhoff's individual ergodic theorem:
\begin{theorem}[Birkhoff's ergodic theorem, cf.\ {\cite[Section 6.2]{Krengel}}]\label{T:Birk}
  For all $f\in L^1(\Omega)$ there exists a set of full measure $\Omega_f\subset\Omega$ such that for any Lipschitz domain $A\subset\R^d$ with $0<|A|<\infty$ we have
  \begin{equation}\label{eq:T:Birk}
    \lim\limits_{\e\downarrow 0}\e^d\!\!\!\!\sum_{z\in A\cap\e\Z^d}f(\tau_{\frac{z}\e}\omega)=|A|\,\mathbb E[f]\qquad\text{for all }\omega\in\Omega_f.
  \end{equation}  
\end{theorem}
Another ingredient is an ergodic theorem for subadditive quantities. Similarly to the continuum case (cf.\ \cite{DMM86,MM94}), we define for all $F\in\R^{n\times d}$ and measurable $A\subset\R^d$ with $|A|<\infty$
\begin{equation*}
  m_F(\omega;A):=\inf\left\{E_1(\omega;g_F+\phi,A)\ |\ \phi\in\calA_0^1(A)\right\},
\end{equation*}
where $g_F$ denotes the linear map $g_F(x)=Fx$.
It turns out that $m_F(\omega;\cdot)$ is subadditive and we derive from a variant of the Ackoglu-Krengel subadditive ergodic theorem, cf.~\cite{ACG11,AK81}:
\begin{lemma}\label{L:indwhom}
Suppose Assumption~\ref{H} is satisfied. For every $F\in\R^{n\times d}$ there exists a set $\Omega_F\subset\Omega$ of full measure such that for all cubes $Q$ satisfying $\bar Q=[a,b]$ with $a,b\in\R^d$ and $\omega\in\Omega_F$: 
\begin{equation}\label{lim1}
  W_0(F):=\inf_{k\in\N}\frac{\mathbb E \left[m_F(\cdot;kY)\right]}{k^d}=\lim_{k\uparrow\infty}\frac{\mathbb E \left[m_F(\cdot;kY)\right]}{k^d}=\lim_{t\uparrow\infty}\frac{m_F(\omega;tQ)}{|tQ|}.
\end{equation}
\end{lemma}
The proof of this statement is rather standard. For the convenience of the reader we present the argument in Appendix~\ref{appendix}.
\begin{remark}\label{Remark:finiterange}

The small non-locality of the discrete energy (i.e.\ $E_\e(\omega;u,A)$ potentially depends on the values of $u$ on the larger set $(A)_{\e R}$ with $R$ defined in \eqref{def_of_R})  makes the construction of suitable subadditive quantities slightly more subtle than in the continuum case. Note that by definition \eqref{def:Aepsg}, we have
\begin{equation*}
  \calA_0^\e(A)=\left\{u\in\calA^\e\, :\, u=0\mbox{ in $\R^d\setminus (A)_{-\e R}$}\,\right\}.
\end{equation*}
Hence, functions in $\calA_0^1(A)$ vanish in a ``safety zone'' close to $\partial A$. This takes care of the small non-locality of the energy $E_\e$. In particular, we have the following \textit{additive} structure: For all $A_1,A_2\subset \R^d$ Lipschitz domains and $u_i\in\calA_0^\e(A_i)$ for $i=1,2$, we have
\begin{equation*}
  A_1\cap A_2=\emptyset\qquad\Rightarrow\qquad
  E_\e(\omega;u_1+u_2,A_1\cup A_2)=E_\e(\omega;u_1,A_1)+E_\e(\omega;u_2,A_2).
\end{equation*}
\end{remark}

For future reference we state a rescaled version of \eqref{lim1}, which is proven in Section~4:
\begin{corollary}\label{C:rep_Whom}
  Suppose Assumption~\ref{H} is satisfied. Let $F\in\R^{n\times d}$ and let $g:\R^d\to\R^n$ denote an affine function with $\nabla g\equiv F$. Then for all $\omega\in\Omega_F$ and all cubes $Q$ satisfying $\bar Q=[a,b]$ with $a,b\in\R^d$ we have
  \begin{equation}\label{lim1e}
    W_0(F)=\lim_{\e\downarrow0}\left(\frac{1}{|Q|}\inf\left\{E_\e(\omega;\varphi,Q)\ |\ \varphi\in\calA^\e_g(Q)\,\right\}\right).
  \end{equation}  
\end{corollary}

We are going to prove our main theorem with $W_\ho$ replaced by $W_0$ and show a posteriori that $W_0=W_{\ho}$ (cf.~proof of Lemma~\ref{L:Whomper}).
\medskip

Before we outline the proof of the main theorem, we comment on the exceptional sets in the two ergodic theorems, since this is a slightly subtle issue. 
\begin{remark}\label{R:Omega0}
Both ``good'' sets $\Omega_f$ and $\Omega_F$ in Theorem~\ref{T:Birk} and Lemma~\ref{L:indwhom} depend a priori on $f$ (resp. $F$), but not on  $A$ (resp. $Q$). In the proof of Theorem~\ref{T:1} we apply \eqref{eq:T:Birk} to a finite number of different functions, which are all related to the weight functions $\lambda_\bb$. Therefore, a posteriori  we may find a common ``good'' set $\Omega_0$ of full measure. We denote this set by $\Omega_0$. For the set $\Omega_F$ appearing in the subadditive ergodic theorem, the situation is more subtle. Clearly we can find a common ``good'' set $\Omega_1$ of full measure such that \eqref{lim1} is valid for all $F$ with rational entries. Yet, since we do not assume any Lipschitz continuity of the potentials (which is in contrast to the analysis in \cite[Theorem 1]{DMM86}, \cite[Corollary 3.3]{MM94}), we cannot directly conclude that there exists a set of full measure such that \eqref{lim1} holds for all $F\in\R^{n\times d}$. Note that the latter would directly imply continuity of $W_0$, cf.~Remark~\ref{R:lipschitz}. In our case we prove continuity of $W_0$ by a different argument (cf. Proposition~\ref{P:condwhom} below). We set $\Omega_1:=\bigcap_{F\in\Q^{n\times d}}\Omega_F\cap\Omega_0$, so that for all $\omega\in\Omega_1$
\begin{itemize}
\item  \eqref{eq:T:Birk} is valid for the finite family of functions $f$ mentioned above,
\item  \eqref{lim1} holds true for all $F\in\Q^{n\times d}$.
\end{itemize}
\end{remark}
\medskip

We outline the proof of Theorem~\ref{T:1} and refer to Section~4 for details and the proofs of the lemmas and propositions below. We first observe that thanks to the moment condition in Assumption~\ref{H} the energy density $W_0$ is non-degenerate:
\begin{lemma}\label{L:indwhom2}
Suppose Assumption~\ref{H} is satisfied. Then there exists $c_1'>0$ such that
\begin{equation}
 \frac1{{c_1}'}|F|^p-{c_1}'\leq W_0(F)\leq {c_1}'(1+|F|^p)\quad\mbox{for all $F\in\R^{n\times d}$.}\label{cond:Whom1}
\end{equation}
\end{lemma}
By exploiting the boundedness of $\mathbb E[\lambda_\bb^{-\frac{1}{p-1}}]$ for $\bb\in\calNN_0$ we obtain equicompactness of the energy:

\begin{lemma}[Compactness]\label{L:coercivity}
  Suppose Assumption~\ref{H} is satisfied. Fix $\omega\in\Omega_0$, cf. Remark~\ref{R:Omega0}, and a bounded Lipschitz domain $A\subset \R^d$. Let $(u_\e)$ denote a sequence with finite energy, i.e. satisfying \eqref{ass:coer}. Then $u_\e\wto u$ weakly in $L^1(A,\R^n)$ implies $u\in W^{1,p}(A,\R^n)$ and $u_\e\wto u$ weakly in $W_{\loc}^{1,\frac{\beta}{\beta+1} p}(A,\R^n)$.
\end{lemma}
The previous lemma combined with a two-scale argument allows to construct recovery sequences for affine limits:
\begin{lemma}[Recovery sequence---affine limit]\label{L:affine}
  Suppose Assumption~\ref{H} is satisfied. Consider an affine function $g:\R^d\to\R^n$ with $\nabla g=F$ and fix $\omega\in\Omega_F\cap \Omega_0$, cf. Remark~\ref{R:Omega0}. Then for all bounded Lipschitz domains $A\subset\R^d$ there exists a sequence $(u_\e)$ such that 
  \begin{equation*}
    \lim_{\e\downarrow0}\|u_\e-g\|_{L^{\frac{\beta}{\beta+1} p}(A)}=0~\mbox{ and }~\lim_{\e\downarrow0}E_\e(\omega;u_\e,A)= |A|W_0(F).
  \end{equation*}
\end{lemma}
Up to this point our argument only relied on the general moment condition of Assumption~\ref{H}. To proceed further, we require localization arguments that are based on gluing constructions for which we additionally need to suppose Assumption~\ref{A:2T}, see Lemmas~\ref{L:glue} and \ref{L:glue:convex} below. In addition to these gluing contructions, the generalization of Lemma~\ref{L:affine} to non-affine limits crucially relies on some mild regularity properties of $W_0$, which we state next:
\begin{proposition}[Continuity of $W_0$]\label{P:condwhom}
Suppose Assumptions~\ref{H} and \ref{A:2T} are satisfied. Then the function $W_0:\R^{n\times d}\to[0,\infty)$ is continuous.
\end{proposition}
The proof of the continuity of $W_0$ combines the recovery sequence constructions for affine functions, cf.\ Lemma~\ref{L:affine}, and the already mentioned gluing constructions in Lemma~\ref{L:glue} and Lemma~\ref{L:glue:convex}.
\begin{remark}\label{R:lipschitz}
If we assume in addition that the potential $V_\bb$ is $p$-Lipschitz continuous, i.e.\ that there exists a constant $L>0$ such that for all $r,s\in\R^n$ and $\bb\in\calE_0$ it holds
\begin{equation}\label{Vplipschitz}
 |V_\bb(\omega;r)-V_\bb(\omega;s)|\leq L\left(1+\lambda_\bb(\omega)|s|^{p-1}+\lambda_\bb(\omega)|r|^{p-1}\right)|r-s|,
\end{equation}
our proof simplifies. Indeed, adapting arguments of \cite{DMM86,MM94}, where the unweighted continuum case is considered, we deduce from \eqref{Vplipschitz}  directly that $W_0$ is $p$-Lipschitz continuous and that there exists a set of full measure such that \eqref{lim1} holds true for all $F\in\R^{n\times d}$. 

In the continuum case, \eqref{Vplipschitz} might be considered to be a mild assumption: Since quasiconvex integrands with $p$-growth automatically satisfy \eqref{Vplipschitz}, the $p$-Lipschitz condition already would follow from the growth condition \eqref{ass:V:1} and lower semicontinuity of $E_\e$. This observation fails in the discrete case considered here. We therefore do not assume  \eqref{Vplipschitz}.
\end{remark}

We prove the liminf inequality by appealing to the blow-up technique, introduced in \cite{FM91}, cf.\ \cite{BMS08,DG15} for similar applications to homogenization:
\begin{proposition}[Lower bound]\label{P:inf}
Suppose Assumptions~\ref{H} and \ref{A:2T} are satisfied. Fix $\omega\in \Omega_1$, cf.~Remark~\ref{R:Omega0}. Let $A\subset\R^d$ be a bounded Lipschitz domain. Let $(u_\e)$ denote a sequence in $L^1(A,\R^n)$ that weakly converges in $L^1(A)$ to a limit $u\in W^{1,p}(A,\R^n)$. Then
\begin{equation*}
  \liminf_{\e\downarrow0}E_\e(\omega;u_\e,A)\geq \int_AW_0(\nabla u(x))\,dx.
\end{equation*}
\end{proposition}
With the continuity of $W_0$ (Proposition~\ref{P:condwhom}) and the liminf inequality (Proposition~\ref{P:inf}) at hand, the existence of a recovery sequence for arbitrary functions in $W^{1,p}$ can be easily deduced from the existence of a recovery sequence for affine functions (Lemma~\ref{L:affine}) as outlined in Section~\ref{S:P}:  
\begin{proposition}[Recovery Sequence]\label{P:sup}
Suppose Assumptions~\ref{H} and \ref{A:2T} are satisfied.  Fix $\omega\in \Omega_1$, cf.~Remark~\ref{R:Omega0}. Let $A\subset \R^d$ be a Lipschitz domain and $u\in W^{1,p}(A,\R^n)$. Then there exists a sequence $(u_\e)\subset W_\loc^{1,\infty}(\R^d,\R^n)$ such that 
 $$\lim_{\e\downarrow0}\|u_\e-u\|_{L^{\frac{\beta}{\beta+1} p}(A)}=0\quad\mbox{and}\quad \lim_{\e\downarrow0}E_\e(\omega;u_\e,A)=\int_A W_0(\nabla u(x))\,dx.$$
\end{proposition}

We finally state the (somewhat technical) gluing constructions that we apply in the proofs of Proposition~\ref{P:condwhom}, \ref{P:inf} and \ref{P:sup}. In the vectorial case our construction is based on a  compact embedding in weighted spaces. It relies on the following weighted Poincar\'e inequality:
\begin{lemma}\label{L:interpolation}
Let Assumption~\ref{H} be satisfied and let $1\leq q<\infty$ satisfy
\begin{equation}\label{alphabetaqp}
 \alpha>1,\quad \left(1-\frac1\alpha\right)\frac{1}q\geq \left(1+\frac1\beta\right)\frac{1}p-\frac1d.
\end{equation}
Then there exists a constant $C<\infty$ such that for every cube $Q\in\mathcal Q_\e:=\{[a,b)\,:\,a,b\in\e\Z^d\}$, every $u\in\calA^\e$ and all $\omega\in\Omega$ we have
\begin{align}\label{ineq:interpolation}
 &\left(\fint_{Q}|u(x)-\fint_Qu(y)\,dy|^q\bigg(\sum_{\bb\in\calE_0}\lambda_\bb(\tau_{\lfloor\frac{x}{\e}\rfloor}\omega)\bigg)\,dx\right)^\frac1q\notag\\
 &\,\leq C|Q|^{\frac1d}\left(m_{\e,Q,\alpha}(\omega)\right)^\frac1q \left(\tilde m_{\e,Q,\beta}(\omega)\right)^\frac1p  \left(\frac{\e^d}{|Q|}\!\!\sum_{z\in\e\Z^d\cap (Q)_{\e R}}\sum_{\bb\in\calNN_0}\lambda_\bb(\tau_{\frac{z}{\e}}\omega)|\partial_\bb u(z)|^p\right)^{\frac{1}p},
\end{align}
where
\begin{eqnarray*}
  m_{\e,Q,\alpha}(\omega)
  &:=&
  \left\{\begin{aligned}
      &\bigg(\frac{\e^d}{|Q|}\!\sum_{z\in\e\Z^d\cap Q}\sum_{\bb\in\calE_0}\lambda_\bb(\tau_{\frac{z}{\e}}\omega)^\alpha\bigg)^{\frac1\alpha}&&\text{for }\alpha<\infty,\\
      &\sum_{\bb\in\calE_0}\sup_{\Omega}\lambda_\bb &&\text{for }\alpha=\infty,
    \end{aligned}\right.\\
  \tilde m_{\e,Q,\beta}(\omega)
  &:=&
  \left\{\begin{aligned}
      &\bigg(\frac{\e^d}{|Q|}\!\sum_{z\in\e\Z^d\cap (Q)_{\e R}}\sum_{\bb\in\calNN_0}(\tfrac{1}{\lambda_\bb(\tau_{\frac{z}{\e}}\omega)})^{\beta}\bigg)^{\frac1\beta }&&\text{for }\beta<\infty,\\
      &\sum_{\bb\in\calNN_0}\sup_\Omega\frac1{\lambda_\bb}&&\text{for }\beta=\infty.
    \end{aligned}\right.
\end{eqnarray*}
\end{lemma}
Combined with a two-scale argument we obtain:
\begin{lemma}\label{L:interpolation2}
Let Assumption~\ref{H} be satisfied and suppose \eqref{alphabetaqp} with $\frac1q>\frac1p-\frac1d$. Fix $\omega\in\Omega_0$, cf.~Remark~\ref{R:Omega0}, and a bounded Lipschitz domain $A\subset \R^d$. Let $(u_\e)$ denote a sequence with $u_\e\in\calA_\e$ satisfying 
\begin{equation}\label{ass:interpolation}
  \limsup_{\e\downarrow0}\e^d\!\!\!\!\sum_{z\in\e\Z^d\cap A}\sum_{\bb\in\calNN_0}\lambda_\bb(\tau_{\frac{z}{\e}}\omega)|\partial_\bb^\e u(z)|^p<\infty.
\end{equation}
Then, $u_\e \wto u$ weakly in $L^1(A)$ implies for all $A'\Subset A$ 
\begin{equation}\label{claim:interpolation}
 \limsup_{\e\downarrow0}\int_{A'}|u_\e(x)|^q\Big(\sum_{\bb\in\calE_0}\lambda_\bb(\tau_{\lfloor\frac{x}{\e}\rfloor}\omega)\Big)\,dx\leq \left(\sum_{\bb\in\calE_0}\mathbb E[\lambda_\bb]\right) \int_{A}|u(x)|^q\,dx.
\end{equation}
\end{lemma}

In the proof of Theorem~\ref{T:1}, we use Lemma~\ref{L:interpolation} and Lemma~\ref{L:interpolation2} in the case $p=q$, where \eqref{alphabetaqp} coincides with \eqref{intro:ass2}. If we assume $p=q$ and the strict inequality $\frac1\alpha+\frac1\beta<\frac{p}d$, we obtain \eqref{claim:interpolation} from \eqref{ass:interpolation} and $u_\e \wto u$ in $L^1(A,\R^n)$ rather directly using H\"older's inequality, Birkhoff's Theorem~\ref{T:Birk} and the Rellich compact embedding. 

Finally, we state the gluing constructions.
\begin{lemma}[Gluing I]  \label{L:glue}
Let Assumption~\ref{H} and Assumption~\ref{A:2T} (A) be satisfied.  Fix $\omega\in\Omega_0$, cf.~Remark~\ref{R:Omega0}. Then there exists a constant $C<\infty$ such that the following statement holds:\\
Let $A\subset\R^d$ be a Lipschitz domain, $\overline u\in W_\loc^{1,\infty}(\R^d,\R^n)$, $u\in W^{1,p}(A,\R^n)$, and $(u_\e)$ a sequence with $u_\e\in\calA_\e$ satisfying 
\begin{align}
    \label{L:glue:0}
    \lim_{\e\downarrow0}\|u_\e-u\|_{L^1(A)}=0.
\end{align}
Set $E:=\limsup_{\e\downarrow0} E_\e(\omega;u_{\e},A)$. Then for all $\delta>0$ sufficiently small and $m\in\N$, there exists $(v_\e)$ with $v_\e\in\calA_\e$ and $v_\e=\overline u$ in $\R^d\setminus (A)_{-\frac{\delta}{4}}$ such that
\begin{equation}\label{L:glue:st0}
 \begin{split}
 \limsup_{\e\downarrow0}E_\e(\omega;v_\e,A)\leq& \left(1+\tfrac{C}{m}\right)E+C\left|A\setminus (A)_{-\delta}\right| \left(1+\|\nabla \overline u\|_{L^\infty((A)_1)}^p\right)\\
  &+C\frac{1}{m}\left(\frac{m}{\delta}\right)^p \|u-\overline u\|_{L^p(A)}^p
  \end{split}
\end{equation}
and
\begin{align}\label{L:glue:lq}
\forall q\in[1,\infty]:~ \limsup_{\e\downarrow0}\|v_\e-\overline u\|_{L^q(A)}\leq \limsup_{\e\downarrow0}\|u_\e-\overline u\|_{L^q(A)}.
\end{align}
\end{lemma}

In the scalar case, i.e.\ $n=1$, we provide a different gluing construction that relies on a truncation argument and allows for weaker moment-conditions. 
\begin{lemma}[Gluing II]  \label{L:glue:convex}
  Suppose that Assumption~\ref{H} and Assumption~\ref{A:2T}~(B) are satisfied. Fix $\omega\in\Omega_0$, cf.~Remark~\ref{R:Omega0}. Then there exists a constant $C<\infty$ and a sequence $(O_M)_{M\in\N}\subset [0,\infty)$ with $\lim_{M\uparrow\infty}O_M=0$ such that the following statement holds:\\
  Let $A\subset\R^d$ be a Lipschitz domain, $\overline u\in W_\loc^{1,\infty}(\R^d)$, $u\in W^{1,p}(A)$, and $(u_\e)$ a sequence with $u_\e\in\calA_\e$ satisfying \eqref{L:glue:0}. Set $E:=\limsup_{\e\downarrow0} E_\e(\omega;u_{\e},A)$.
Then for all $\delta>0$ sufficiently small, $m,M\in\N$ and $s>0$, there exists $(v_\e)$ with $v_\e\in\calA_\e$, $v_\e=\overline u$ in $\R^d\setminus (A)_{-\frac{\delta}{4}}$ such that \eqref{L:glue:lq} holds, and
\begin{equation}\label{L:glue:st0:convex}
 \begin{split}
  \limsup_{\e\downarrow0}E_\e(\omega;v_\e,A)
  &\leq \left(1+\tfrac{C}{m}\right)E+C\left|A\setminus (A)_{-\delta}\right| \left(h\left(\|\nabla \overline u\|_{L^\infty((A)_1)}\right)+\left(\tfrac{ms}{\delta }\right)^p\right)\\
  &\quad+Ch\!\left(\|\nabla \overline u\|_{L^\infty((A)_1)}\right) \left(\tfrac{M}s\|u-\overline u\|_{L^1(A)}+O_M|A|\right)\\
  &\quad+C\left(E+|A|\right)^\frac{q}{p} \left(\tfrac{M}s\|u-\overline u\|_{L^1(A)}+O_M|A|\right)^\frac{p-q}{p},
  \end{split}
\end{equation}
where $h(a)=1+a^p+a^q$, and $1<q<p$ as in Assumption~\ref{A:2T}~(B).
\end{lemma}

\subsection{Moment conditions under independence}\label{sec:iid}

The moment condition imposed on the growth from below can be weakened if we replace ergodicity by the assumption of independent and identically distributed (iid) weights. Below, we discuss the argument for the lattice graph $(\Z^d,\mathbb B^d)$, where $\mathbb B^d:=\{[z,z+ e_i]\,|\,i=1,\dots,d\}$ with iid potentials. In this case the moment condition $\mathbb E[\lambda_\bb^{-\beta}]<\infty$ in the Assumptions~\ref{H} and \ref{A:2T} can be replaced by the weaker condition $\mathbb E[\lambda_\bb^{-\gamma}]<\infty$ for some $\gamma>\frac{\beta}{2d}$. Notice that the gain of a factor $\frac1{2d}$ in the integrability is related to the geometry of the graph $(\Z^d,\mathbb B^d)$.

Indeed, a careful inspection of the proof of Theorem~\ref{T:1} shows that the moment condition $\mathbb E[\lambda_\bb^{-\beta}]<\infty$ for all $\bb\in\calNN_0$ is only used to ensure existence of a random variable $f\in L^1(\Omega)$, namely $f(\omega)=\sum_{\bb\in\calNN_0}\lambda_\bb(\omega)^{-\beta}$, such that for all bounded Lipschitz domains $A\subset\R^d$
\begin{align*}
 &\left(\e^d\!\!\!\!\sum_{z\in\e\Z^d\cap A}\sum_{\bb\in\calNN_0}|\partial_\bb^\e u(z)|^{\frac{\beta}{\beta+1}p}\right)^{\frac{\beta+1}{\beta}}\\
 &\leq \left(\e^d\!\!\!\!\sum_{z\in\e\Z^d\cap A}f(\tau_{\frac{z}{\e}}\omega)\right)^{\frac{1}{\beta}} \left(\e^d\!\!\!\!\sum_{z\in\e\Z^d\cap A}\sum_{\bb\in\calNN_0}\lambda_\bb (\tau_{\frac{z}\e}\omega)|\partial_\bb^\e u(z)|^p\right),
\end{align*}
cf.\ the proofs of Lemma~\ref{L:Whomper}, Lemma~\ref{L:indwhom}~Step~4, Lemma~\ref{L:coercivity} and Lemma~\ref{L:affine}~Step~2. We argue that this estimate extends to the specific random environment that we introduce next: Set $\widetilde\Omega:=(0,\infty)$ and fix a probability measure $\widetilde{\mathbb P}$ on $(\widetilde\Omega,\widetilde{\mathcal B})$, where $\widetilde{\mathcal B}=\mathcal B((0,\infty))$. Set  $\Omega=\widetilde\Omega^{\mathbb B^d}$ and let ${\mathbb P}:=\widetilde{\mathbb P}^{\otimes\mathbb B^d}$ denote the $\mathbb B^d$-fold product measure on $\Omega$. For every $z\in\Z^d$ the shift $\tau_z:\Omega\to\Omega$ is defined, as in the examples at the end of Section~\ref{sec:2}, by $(\tau_z\omega)(\bb)=\omega(z+\bb)$ for all $\bb\in \mathbb B^d$. By construction, we have that $\{\Omega\ni\omega\mapsto \omega(\bb)\}_{\bb\in\mathbb B^d}$ are independent and identically distributed random variables. We show that a variant of the above inequality holds in this situation under weaker moment conditions:
\begin{proposition}\label{Prop:iid}
Let $(\widetilde\Omega,\widetilde{\mathbb P})$ be defined as above and suppose that
 \begin{equation*}
  \forall \bb\in\mathbb B^d\qquad \mathbb E[\omega(\bb)]<\infty\quad\mbox{and}\quad\mathbb E[\omega(\bb)^{-\gamma}]<\infty\quad\mbox{for some $\gamma>\frac1{2d(p-1)}$}.
 \end{equation*}
 Then, for all $\frac{1}{p-1}\leq\beta<2d\gamma$ there exists a random variable $f\in L^1(\Omega)$ such that for every bounded Lipschitz domain $A\subset\R^d$, $v\in\calA^\e$ and $\omega\in\Omega$ it holds
  \begin{align}\label{ineq:iid:claim}
 &\left(\e^d\!\!\!\!\sum_{z\in\e\Z^d\cap A}\sum_{i=1}^d|\partial_{e_i}^\e v(z)|^{\frac{\beta}{\beta+1}p}\right)^{\frac{\beta+1}{\beta}}\notag\\
 &\qquad\leq \left(\e^d\!\!\!\sum_{z\in\e\Z^d\cap (A)_\e}f(\tau_{\frac{z}\e}\omega)\right)^\frac1\beta\left(\e^d\!\!\!\!\sum_{z\in\e\Z^d\cap (A)_{4\e}}\sum_{i=1}^d(\tau_{\frac{z}\e}\omega)(e_i)|\partial_{e_i}^\e v(z)|^p\right).
\end{align}
\end{proposition} 

With Proposition~\ref{Prop:iid} at hand, it is straightforward to adapt the proof of Theorem~\ref{T:1} in the situation of Proposition~\ref{Prop:iid} and to obtain the following: 

Assume
$$ \mathbb E[\omega(\bb)^\alpha]+\mathbb E[\omega(\bb)^{-\gamma}]<\infty\quad\mbox{for all $\bb\in\mathbb B^d$},$$
and one of the following conditions is satisfied:
  \begin{itemize}
  \item $1<\alpha<\infty$, $\frac{1}{2d(p-1)}<\gamma$ and $\frac{1}{\alpha}+\frac{1}{2d\gamma}<\frac{p}{d}$,
  \item $1\leq\alpha<\infty$, $\frac{1}{2d(p-1)}<\gamma$ and Assumption~\ref{A:2T} (B).
  \end{itemize}
Then the conclusions of Theorem~\ref{T:1} and Corollary~\ref{C:1} (with $\frac1q>(1+\frac1{2d\gamma})\frac1p-\frac1d$) hold for energies of the form
$$E_\e(\omega;u,A)=\e^d\!\!\!\!\sum_{z\in\e\Z^d\cap A}\sum_{i=1}^d(\tau_{\frac{z}\e}\omega)(e_i)V(\partial_{e_i}^\e u(z)),$$
where $V\in C(\R^n,[0,\infty))$ satisfies standard $p$-growth.

\begin{remark}
  An interesting, open question is whether the assumptions on the exponents $\alpha,\gamma$ are critical for the above conclusion or not. In the following, we discuss the special case of scalar problems with quadratic potentials (i.e. $p=2$), $\Gamma$-convergence in $L^2$ (i.e. $q=2$), and $\alpha=1$. This case corresponds to a random conductance model on $\Z^d$ with independent and identically distributed conductances. 
  In this situation our results yield $\Gamma$-convergence (in $L^2$) provided $\gamma>\frac{1}{4}=:\gamma_c$ ($\Leftrightarrow$ $\frac1q:=\frac12>(1+\frac1{2d\gamma})\frac12-\frac1d$). For this conclusion the threshold $\gamma_c$ is critical: In a very recent work \cite{FF16}, localization of the first Dirichlet eigenvectors of the associated discrete elliptic equation is studied. It is shown that in the subcritical regime $\gamma<\gamma_c$ localization (and thus failure of homogenization in $L^2$) of the first eigenvector occures. On the other hand, for $\gamma>\gamma_c$, our results imply $\Gamma$-convergence of the corresponding quadratic form w.r.t.\ strong $L^2$-convergence (and thus homogenization of the first eigenvectors, see  \cite[Section 1.3]{FF16} for further discussions). Moreover, it is interesting to note that the exponent $\gamma_c=\frac14$ is also critical for the validity of a \emph{local central limit theorem} for the associated \emph{variable speed random walk}, see \cite[Remark 1.10]{BKM15}, and \cite{ADS16,MO16} for related results. In particular, in \cite{BKM15} it is shown that in the subcritical case $\gamma<\gamma_c$ the associated heat kernel features an anomalous decay due to trapping of the random walk, while this is not the case for $\gamma>\gamma_c$. 
\end{remark}


\section{Proofs}\label{S:P}

We tacitly drop the dependence on $\omega$ in our notation, e.g. we simply write $E_\e(u,A)$ instead of $E_\e(\omega;u,A)$ or $m_F(A)$ instead of $m_F(\omega;A)$. Furthermore, we shall use the shorthand notation
\begin{equation*}
  \lambda_{\bb}(x)=\lambda_\bb(\tau_{\lfloor x\rfloor}\omega)\qquad\text{and}\qquad V_\bb(x;\cdot)=V_\bb(\tau_{\lfloor x \rfloor}\omega;\cdot),
\end{equation*}
where $\lfloor \cdot\rfloor$ denotes the lower Gauss bracket extended component-wise to $\R^d$.

\subsection{Compactness and recovery sequence in the affine case: Corollary~\ref{C:rep_Whom} and Lemmas~\ref{L:indwhom2}, \ref{L:coercivity} and \ref{L:affine}}

\begin{proof}[Proof of Corollary~\ref{C:rep_Whom}]
  The definition of $m_F$ and a change of variables yield 
  \begin{equation*}
    \e^dm_F(\tfrac1\e Q)=\inf\left\{E_\e(\varphi,Q)\ |\ \varphi\in\calA^\e_g(Q)\,\right\},
  \end{equation*}
  where we have used that the right-hand side does not change if we add a constant to $g$.
  On the other hand, we have by Lemma~\ref{L:indwhom}  
  \begin{equation*}
    |Q|W_0(F)\stackrel{\eqref{lim1}}{=}|Q|\lim_{\e\downarrow0}\frac{m_F(\tfrac1\e Q)}{|\tfrac1\e Q|}
    =\lim_{\e\downarrow0}\e^dm_F(\tfrac1\e Q).
  \end{equation*}

\end{proof}

\begin{proof}[Proof of Lemma~\ref{L:indwhom2}]
  The upper bound is a direct consequence of the upper bound in \eqref{ass:V:1} and $\sum_{\bb\in\calE_0}\mathbb E[\lambda_\bb]<\infty$. 
  To prove the lower bound, we fix $F\in\R^{n\times d}$ and $\omega\in\Omega_0\cap \Omega_F$, cf.\ Remark~\ref{R:Omega0}. Denote by $g_F$ the linear function $g_F(x)=Fx$. Let $(\phi_k)_k$ be such that $\phi_k\in\calA_0^1(kY)$ and 
  \begin{equation}\label{lim:whom}
    \lim_{k\uparrow\infty}\frac1{k^d}E_1(g_F+\phi_k,kY)= W_0(F).
  \end{equation}
  Consider $k\gg R$. Using \eqref{ass:V:1}, $\int_{(kY)_{-R}}\nabla \phi_k\,dx=0$, \eqref{est:sumint} and H\"older's inequality, we obtain
  \begin{align*}
    |F|^p=&\left|\fint_{(kY)_{-R}}\left(F+\nabla \phi_k(y)\right)\,dy\right|^p\\
    \leq& \left(\frac{c_0}{|(kY)_{-R}|}\sum_{z\in \Z^d\cap kY}\sum_{\bb\in\calNN_0}|\partial_\bb^1(g_F+\phi_k)(z)|\right)^p\\
    \leq&{c_0}^p\frac{|k Y|^p}{|(kY)_{-R}|^p}\left(\frac{1}{k^d}\sum_{z\in \Z^d\cap kY}\sum_{\bb\in\calNN_0}\lambda_\bb(z)|\partial_\bb^1(g_F+\phi_k)(z)|^p\right)\\
    &\quad\times\left(\frac{1}{k^d}\sum_{z\in \Z^d\cap kY}\sum_{\bb\in\calNN_0}\lambda_\bb(z)^{-\frac{1}{p-1}}\right)^{p-1}\\
    \leq& {c_0}^p{c_1}\frac{|k Y|^p}{|(kY)_{-R}|^p}\left(\frac1{k^d}E_1(g_F+\phi_k,kY)+\frac{{c_1}}{k^d}\sum_{z\in \Z^d\cap kY}\sum_{\bb\in\calNN_0}\lambda_\bb(z)\right)\\
    &\quad\times\left(\frac{1}{k^d}\sum_{z\in \Z^d\cap kY}\sum_{\bb\in\calNN_0}\lambda_\bb(z)^{-\frac{1}{p-1}}\right)^{p-1}.
  \end{align*}
  Theorem~\ref{T:Birk} and \eqref{lim:whom} yield
  \begin{align*}
    |F|^p \leq&{c_0}^p{c_1}\left(W_0(F)+{c_1}\sum_{\bb\in\calNN_0}\mathbb E[\lambda_\bb]\right)\left(\sum_{\bb\in\calNN_0}\mathbb E[\lambda_\bb^{-\frac{1}{p-1}}]\right)^{p-1},
  \end{align*}
  and thus the lower bound for $W_0$ by Assumption~\ref{H}.
\end{proof}

\begin{proof}[Proof of Lemma~\ref{L:coercivity}]
Fix $\omega\in\Omega_0$.

\step{1} ($L^{\frac{\beta}{\beta+1} p}$ boundedness of $\nabla u_\e$). We claim that
\begin{align}\label{est:coer:u:0}
 \limsup_{\e\downarrow0}\int_{(A)_{-\e R}}|\nabla u_\e|^{\frac{\beta}{\beta+1} p}\,dx<\infty.
\end{align}
In the case $\beta<\infty$, H\"older's inequality with exponents $(\frac{\beta+1}{\beta},\beta+1)$ yields
\begin{equation}\label{est:coer:u:00001}
  \begin{aligned}
    & \e^d\!\!\!\!\sum_{z\in\e\Z^d\cap A}\sum_{\bb\in\calNN_0}|\partial_\bb^\e u_\e(z)|^{\frac{\beta}{\beta+1}p}\\
    &\quad\leq\left(\e^d\!\!\!\!\sum_{z\in\e\Z^d\cap
        A}\sum_{\bb\in\calNN_0}\lambda_\bb(\tfrac{z}\e)^{-\beta}\right)^{\frac1{\beta+1}}\left(\e^d\!\!\!\!\sum_{z\in\e\Z^d\cap
        A}\sum_{\bb\in\calNN_0}\lambda_\bb(\tfrac{z}\e)|\partial_\bb^\e
      u_\e(z)|^p\right)^{\frac\beta{\beta+1}}.
  \end{aligned}
\end{equation}
Two applications of Theorem~\ref{T:Birk} yield 
\begin{align*}
 &\limsup\limits_{\e\downarrow0}\e^d\!\!\!\!\sum_{z\in\e\Z^d\cap A}\sum_{\bb\in\calNN_0}\lambda_\bb(\tfrac{z}{\e})^{-\beta}=|A|\sum_{\bb\in\calNN_0}\mathbb E[\lambda_\bb^{-\beta}],\\
&\limsup\limits_{\e\downarrow0}\e^d\!\!\!\!\sum_{z\in\e\Z^d\cap A}\sum_{\bb\in\calNN_0}\lambda_\bb(\tfrac{z}{\e})|\partial_\bb^\e u_\e(z)|^p\\
 &\qquad\stackrel{\eqref{ass:V:1}}{\leq} \limsup\limits_{\e\downarrow0}c_1E_\e(\omega;u_\e,A)+c_1^2|A|\sum_{\bb\in\calNN_0}\mathbb E[\lambda_\bb]=:\overline E\stackrel{\eqref{ass:coer}}{<}\infty,
\end{align*}
and thus
\begin{align}\label{est:coer:u}
 \limsup_{\e\downarrow0} \e^d\!\!\!\!\sum_{z\in\e\Z^d\cap A}\sum_{\bb\in\calNN_0}|\partial_\bb^\e u_\e(z)|^{\frac{\beta}{\beta+1}p}\leq \overline E^{\frac\beta{\beta+1}}\left(|A|\ \sum_{\bb\in\calNN_0}\mathbb E[\lambda_\bb^{-\beta}]\right)^{\frac1{\beta+1}}<\infty.
\end{align}
Combining \eqref{est:sumint} and \eqref{est:coer:u}, we obtain the claim \eqref{est:coer:u:0}. 

If $\beta=\infty$ (and thus $\frac{\beta}{\beta+1}p=p$), the previous argument simplifies, since we can use the trivial estimate $|\partial_\bb^\e u_\e(z)|^p\leq\sup_\Omega(\frac{1}{\lambda_\bb})\lambda_\bb(\frac{z}{\e})|\partial_\bb^\e u_\e(z)|^p$ for all $\bb\in\calNN_0$.

\step{2} (Equi-integrability). We claim that
\begin{equation}\label{est:equiint}
  \limsup_{k\uparrow\infty}\limsup_{\e\downarrow0}\int_{(A)_{-\e R}\cap\{|\nabla u_\e|\geq k\}}|\nabla u_\e|\,dx=0.
\end{equation}
We only need to consider the case $\frac{\beta}{\beta+1}p=1$ (i.e.\ $\beta=\frac{1}{p-1}$), since for $\frac{\beta}{\beta+1}p>1$ estimate \eqref{est:equiint} directly follows from \eqref{est:coer:u:0}. Fix $k\in\N$. A calculation similar to the one in the previous step yields
\begin{align*}
&\limsup_{\e\downarrow0}\e^d\!\!\!\!\sum_{z\in\e\Z^d\cap A }\!\!\sum_{\bb\in\calNN_0 \atop |\partial_\bb^\e u_\e(z)|\geq k }\!\!|\partial_\bb^\e u_\e(z)|
  \leq \overline E^\frac1p\limsup_{\e\downarrow0}\left(\e^d\!\!\!\!\!\!\sum_{z\in\e\Z^d\cap A}\sum_{\bb\in\calNN_0\atop |\partial_\bb^\e u_\e(z)|\geq k}\!\!\!\lambda_\bb(\tfrac{z}{\e})^{-\frac1{p-1}}\right)^{\frac{p-1}{p}}.
\end{align*}
By appealing to the decomposition $\lambda_\bb^{-\frac1{p-1}}=(1-\chi)\lambda_\bb^{-\frac1{p-1}}+\chi\lambda_\bb^{-\frac1{p-1}}$ with $\chi$ denoting the indicator function of the peak level set $\{\lambda_\bb^{-1}>k^{\frac{p-1}2}\}$,  we deduce that
\begin{equation*}
\e^d\sum_{z\in\e\Z^d\cap A}\sum_{\bb\in\calNN_0\atop |\partial_\bb^\e u_\e(z)|\geq k}\lambda_\bb(\tfrac{z}{\e})^{-\frac1{p-1}}
\leq
\e^d\sum_{z\in\e\Z^d\cap A}\sum_{\bb\in\calNN_0 }\left(\frac{|\partial_\bb^\e u_\e(z)|}{k^\frac12}+\chi(\tfrac{z}\e)\lambda_\bb(\tfrac{z}{\e})^{-\frac{1}{p-1}}\right).
\end{equation*}
Thanks to \eqref{est:coer:u} and $\mathbb E[\lambda_\bb^{-\frac{1}{p-1}}]<\infty$ we have
\begin{equation*}
\limsup_{k\uparrow\infty}\limsup_{\e\downarrow0}\e^d\sum_{z\in\e\Z^d\cap A}\sum_{\bb\in\calNN_0 }\left(\frac{|\partial_\bb^\e u_\e(z)|}{k^\frac12}+\chi(\tfrac{z}\e)\lambda_\bb(\tfrac{z}{\e})^{-\frac{1}{p-1}}\right)
=0.
\end{equation*}
The combination of the previous three estimates yields
\begin{equation*}
  \limsup_{k\uparrow\infty}\limsup_{\e\downarrow0}\e^d\!\!\!\!\sum_{z\in\e\Z^d\cap A }\sum_{\bb\in\calNN_0 \atop |\partial_\bb^\e u(z)|\geq k}|\partial_\bb^\e u(z)|=0,
\end{equation*}
which implies the asserted estimate \eqref{est:equiint}, as can be seen by an argument similar as in the proof of Lemma~\ref{L:sumint}.

%
%
%
%

\step{3} We claim that $u\in W^{1,\frac{\beta}{\beta+1} p}(A,\R^n)$ and $u_\e\wto u$ in $W_\loc^{1,\frac{\beta}{\beta+1} p}(A,\R^n)$.

Fix $U\Subset A$. Since $U\subset (A)_{-\e R}$ for $\e>0$ sufficiently small, we have by the equi-integrability and \eqref{est:coer:u:0} that $u_\e \wto u$ in $W^{1,\frac{\beta}{\beta+1}p}(U,\R^n)$. Moreover, \eqref{est:coer:u:0} yields $\limsup_{\e\downarrow0}\|\nabla u_\e\|_{L^{\frac{\beta}{\beta+1} p}(U)}\leq C<\infty$, where $C$ is independent of $U\Subset A$. Hence,  $u_\e\wto u$ in $W_\loc^{1,\frac{\beta}{\beta+1}p}(A,\R^n)$ and $\nabla u\in L^{\frac{\beta}{\beta+1}p}(A,\R^{n\times d})$ by the weak lower semicontinuity of the norm. Thus $u\in W^{1,\frac{\beta}{\beta+1} p}(A,\R^n)$ by Poincar\'e's inequality.

\step{4} We prove $u\in W^{1,p}(A,\R^n)$ by a duality argument (similar to \cite[Theorem 5.1]{EPPW06}).

By Poincar\'e's inequality it suffices to show $\nabla u\in L^p(A,\R^{n\times d})$. Fix $j=1,\dots,d$, $k=1,\dots,n$ and a test function $\varphi\in C_c^\infty(A)$. Set $U=\operatorname{spt}\varphi\Subset A$ and $\overline \varphi(z)=\sup_{x\in z+\e Y}|\varphi(x)|$ for $z\in\e\Z^d$. We get for $\e>0$ sufficiently small  
\begin{align*}
 &\|\varphi\partial_j u_\e^k\|_{L^1(A)}=\|\varphi\partial_j u_\e^k\|_{L^1(U)}\\
 &\stackrel{\eqref{est:sumint}}{\leq} c_0\e^d\!\!\!\!\sum_{z\in\e\Z^d\cap A}\sum_{\bb\in\calNN_0}|\overline \varphi(z)||\partial_\bb^\e u(z)|\\
 &\leq c_0\left(\e^d\!\!\!\!\sum_{z\in\e\Z^d\cap A}\sum_{\bb\in\calNN_0}\lambda_\bb(\tfrac{z}{\e})|\partial_\bb^\e u(z)|^p\right)^{\frac1{p}}\left(\e^d\!\!\!\!\sum_{z\in\e\Z^d\cap A}\sum_{\bb\in\calNN_0}\!\!\!\lambda_\bb(\tfrac{z}{\e})^{\frac{-1}{p-1}}|\overline \varphi(z)|^{\frac{p}{p-1}}\right)^{\frac{p-1}{p}}.
\end{align*}
For $\eta>0$  set $I(\eta):=\{i\in\eta\Z^d:i+\eta Y\cap A\neq\emptyset\}$, $Q_i=i+\eta Y$ and $\varphi_i=\max_{x\in \bar Q_i} |\varphi(x)|^{\frac{p}{p-1}}$ for $i\in\eta\Z^d$. The lower semicontinuity of the norm and Theorem~\ref{T:Birk} yield 
\begin{align*}
 \int_A\varphi\partial_j u^k\,dx\leq&\liminf_{\e\downarrow0}\|\varphi\partial_j u_\e^k\|_{L^1(A,\R)}\\
 \leq&c_0\overline E^\frac1p\left(\limsup_{\e\downarrow0}\e^d\!\!\!\!\sum_{z\in\e\Z^d\cap A}\sum_{\bb\in\calNN_0}\lambda_\bb(\tfrac{z}{\e})^{-\frac1{p-1}}|\overline \varphi(z)|^{\frac{p}{p-1}}\right)^{\frac{p-1}{p}}\\
 \leq&c_0\overline E^\frac1p\left(\sum_{i\in I(\eta)}\varphi_i\limsup_{\e\downarrow0}\e^d\!\!\!\!\sum_{z\in\e\Z^d\cap Q_i}\sum_{\bb\in\calNN_0}\lambda_\bb(\tfrac{z}{\e})^{-\frac1{p-1}}\right)^{\frac{p-1}{p}}\\
 \leq&c_0\overline E^\frac1p\left(\sum_{\bb\in\calNN_0}\mathbb E[\lambda_\bb^{-\frac1{p-1}}]\right)^{\frac{p-1}{p}}\left(\sum_{i\in I(\eta)}\varphi_i|Q_i|\right)^{\frac{p-1}{p}},
\end{align*}
where $\overline E:=\limsup\limits_{\e\downarrow0}c_1 E_\e(\omega;u_\e,A)+c_1^2|A|\sum_{\bb\in\calNN_0}\mathbb E[\lambda_\bb]<\infty$ as in Step~1. Since $\varphi$ is smooth, we get by taking a sequence $(\eta_m)_{m\in\N}$ with $\eta_m\downarrow0$ as $m\uparrow\infty$:
$$\int_A \partial_j u^k\varphi\,dx\leq c_0\overline E^\frac1p\left(\sum_{\bb\in\calNN_0}\mathbb E[\lambda_\bb^{-\frac1{p-1}}]\right)^{\frac{p-1}{p}}\|\varphi\|_{L^\frac{p}{p-1}(A)}\quad\mbox{for all $\varphi\in C_c^\infty(A)$.}$$
Since $C_c^\infty(A)\subset L^\frac{p}{p-1}(A)$ is dense, we obtain that $\partial_j u^k\in L^p(A)$ for $k=1,\dots,n$ and $j=1,\dots,d$ by duality. 
\end{proof}

\begin{proof}[Proof of Lemma~\ref{L:affine}]
Recall that $F=\nabla g$. We fix $\omega\in\Omega_F\cap\Omega_0$, cf.\ Lemma~\ref{L:indwhom} and Remark~\ref{R:Omega0}. Our argument relies on a two-scale construction (that is similar to \cite[Lemma~4.2~Step 1]{MM94} where the non-degenerate continuum case is discussed). We introduce an additional length scale $\eta\in(0,1)$ (with $\e\ll\eta\ll 1$), and cover $A$ into disjoint cubes $Q_z$ with side length $\eta$: For $z\in\Z^d$ set $Q_z:=\eta(z+Y)$ and denote by
\begin{align*}
 I&:=\{z\in\Z^d\ :\ Q_z\subset (A)_{-\eta}\},\\
\intertext{the set of labels associated with cubes inside $A$ (with some safety distance to $\partial A$), and by}
 J&:=\{z\in \Z^d\setminus I\ :\ Q_z\cap A\neq\emptyset\},
\end{align*}
the labels of ``boundary'' cubes. Set $\partial_\eta A:=\bigcup_{z\in J}Q_z$ and note that 
\begin{equation}\label{ieta}
 \bigcup_{z\in I}Q_z\subset A\subset \bigcup_{z\in I\cup J}Q_z~\mbox{ and }~\lim_{\eta\downarrow0}|\partial_\eta A|=0,
\end{equation}
where the last identity holds, since $A$ has a Lipschitz boundary. Based on this partition we construct a doubly indexed sequence $u_{\e,\eta}$ which perturbs the affine map $g$ by functions in $\mathcal A^\e_0(Q_z)$, $z\in I$. This is done in Step~1. In Step~2 and Step~3 we estimate the perturbation and in Step~4 we conclude by extracting a suitable diagonal sequence.

\step 1 Construction of $u_{\e,\eta}$.

For each $z\in I$ choose $(\phi_{\e,\eta,z})_\e\subset\calA^1$ with $\phi_{\e,\eta,z}\in\mathcal A_0(\frac{1}{\e}Q_z)$ such that
\begin{equation}\label{L2.11:phi}
  |Q_z|W_0(F)\stackrel{\eqref{lim1}}{=}\lim_{\e\downarrow0}\e^dm_F(\tfrac{1}{\e}Q_z)= \lim_{\e\downarrow0}\e^dE_1(g+\phi_{\e,\eta,z},\tfrac{1}{\e}Q_z),
\end{equation}
and define $u_{\e,\eta}\in\mathcal A_\e$ via
\begin{equation*}
  u_{\e,\eta}(x)=g(x)+\e\phi_{\e,\eta}(\tfrac{x}{\e}),\qquad\text{where }\phi_{\e,\eta}=\sum_{z\in I}\phi_{\e,\eta,z}.
\end{equation*}
We claim that
\begin{equation}\label{L:affine:claim:2}
  (|A|-O(\eta))W_0(F)  \leq\liminf\limits_{\e\downarrow0} E_\e(u_{\e,\eta},A)
  \leq  \limsup\limits_{\e\downarrow0} E_\e(u_{\e,\eta},A)\leq |A|W_0(F)+O(\eta),
\end{equation}
where $O(\eta)$ denotes a (non-negative) function with $\limsup_{\eta\downarrow0}O(\eta)=0$.\\
This can be seen as follows: A direct consequence of \eqref{ieta} and $V_\bb\geq0$ is
$$\sum_{z\in I}E_\e(u_{\e,\eta},Q_z)\leq E_\e(u_{\e,\eta},A)\leq \sum_{z\in I\cup J}E_\e(u_{\e,\eta},Q_z).$$
Since $\phi_{\e,\eta,z}\in\mathcal A_0^1(\frac{1}{\e}Q_z)$, we obtain by the definition of $\phi_{\e,\eta}$, $u_{\e,\eta}$ and the arguments in Remark~\ref{Remark:finiterange} that
\begin{align*}
 \sum_{z\in I}E_\e(u_{\e,\eta},Q_z)=&\sum_{z\in I}\e^dE_1(g+\phi_{\e,\eta},\tfrac1\e Q_z)=\sum_{z\in I}\e^dE_1(g+\phi_{\e,\eta,z},\tfrac1\e Q_z).
\end{align*}
With \eqref{L2.11:phi} the lower bound in \eqref{L:affine:claim:2} follows:
$$\liminf_{\e\downarrow0}E_\e(u_{\e,\eta},A)\geq \sum_{z\in I}|Q_z|W_0(F)\geq (|A|-|\partial_\eta A|)W_0(F).$$
Similarly, we have
\begin{align*}
 \sum_{z\in I\cup J}E_\e(u_{\e,\eta},Q_z)=&\sum_{z\in I\cup J}\e^dE_1(g+\phi_{\e,\eta},\tfrac1\e Q_z)\\
 =&\sum_{z\in I}\e^dE_1(g+\phi_{\e,\eta,z},\tfrac1\e Q_z)+\sum_{z\in J}\e^dE_1(g,\tfrac1\e Q_z).
\end{align*}
The second term can be estimated by \eqref{ass:V:1} and Theorem~\ref{T:Birk} as
\begin{align*}
 \limsup_{\e\downarrow0}\sum_{j\in J}\e^dE_1(g,\tfrac1\e Q_j)\leq& \sum_{j\in J}\lim\limits_{\e\downarrow0}\e^d\!\!\!\!\sum_{z\in\Z^d\cap\frac1\e Q_j}\sum_{\bb\in\calE_0}{c_1}(1+\lambda_\bb(\tau_z\omega)(|F|^p+1))\\
=&|\partial_\eta A|\sum_{\bb\in\calE_0}{c_1}(1+\mathbb E[\lambda_\bb](|F|^p+1)).
\end{align*}
Hence, the upper bound of \eqref{L:affine:claim:2} follows:
\begin{align*}
 \limsup_{\e\downarrow0}E_\e(u_{\e,\eta},A)\leq |A|W_0(F)+|\partial_\eta A|\sum_{\bb\in\calE_0}{c_1}(1+\mathbb E[\lambda_\bb](|F|^p+1)).
\end{align*}

\step 2 Estimate on $\phi_{\e,\eta,z}$.

We claim that 
\begin{equation}\label{L2.11:step2}
  \max_{z\in I}\left(\limsup\limits_{\e\downarrow0}\fint_{\frac{1}{\e}Q_z} |\nabla\phi_{\e,\eta,z}|^{\frac{\beta}{\beta+1}p}\right)\lesssim 1,
\end{equation}
where $\lesssim$ stands for $\leq$ up to a constant that is independent of $\eta$.
Next we provide the argument: Since $\phi_{\e,\eta,z}\in\mathcal A_0^1(\tfrac{1}{\e} Q_z)$, the function $\phi_{\e,\eta,z}$ vanishes outside $(\frac{1}{\e}Q_z)_{-R}$, and thus
\begin{equation*}
  \begin{aligned}
    \e^d\int_{\tfrac{1}{\e}Q_z}|\nabla\phi_{\e,\eta,z}|^{\frac{\beta}{\beta+1}p}\,dx&\lesssim |Q_z||F|^{\frac{\beta}{\beta+1}p}+\e^d\int_{(\tfrac{1}{\e}Q_z)_{-R}}|F+\nabla\phi_{\e,\eta,z}|^{\frac{\beta}{\beta+1}p}\,dx.
    \end{aligned}
\end{equation*}
For $\beta<\infty$, the second term on the right-hand side above can be estimated as follows:
\begin{align*}
 &\e^d\int_{(\tfrac{1}{\e}Q_z)_{-R}}|F+\nabla\phi_{\e,\eta,z}|^{\frac{\beta}{\beta+1}p}\,dx\stackrel{\eqref{est:sumint}}{\leq} c_0\e^d\!\!\!\!\sum_{z'\in\Z^d\cap \frac1\e Q_z}\sum_{\bb\in\calNN_0}|\partial_\bb^1(g+\phi_{\e,\eta,z})(z'))|^{\frac{\beta}{\beta+1}p}\\
 &\leq c_0\left(\e^d\!\!\!\!\sum_{z'\in\Z^d\cap \frac1\e Q_z}\sum_{\bb\in\calNN_0}\lambda_\bb(z')^{-\beta}\right)^{\frac{1}{\beta+1}}\\
 &\qquad \times\left(\e^d\!\!\!\!\sum_{z'\in\Z^d\cap \frac1\e Q_z}\sum_{\bb\in\calNN_0}\lambda_\bb(z')|\partial_\bb^1(g+\phi_{\e,\eta,z})(z'))|^p\right)^{\frac{\beta}{\beta+1}}.
\end{align*}
The combination of the previous two estimates, \eqref{ass:V:1}, \eqref{L2.11:phi} and Birkhoff's ergodic theorem, cf. Theorem~\ref{T:Birk}, yields
\begin{multline*}
    \limsup\limits_{\e\downarrow 0}\e^d\int_{\tfrac{1}{\e}Q_z}|\nabla\phi_{\e,\eta,z}|^{\frac{\beta}{\beta+1}p}\,dx
    \lesssim|Q_z||F|^{\frac{\beta}{\beta+1}p}\\+|Q_z|\left(\sum_{\bb\in\calNN_0}\mathbb E[\lambda_\bb^{-\beta}]\right)^{\frac{1}{\beta+1}}
    \left(W_0(F) + 1 + \sum_{\bb\in\calNN_0}\mathbb E[\lambda_\bb]\right)^\frac{\beta}{\beta+1}.
\end{multline*}
Thanks to \eqref{ass:X0} this implies \eqref{L2.11:step2}. 
If $\beta=\infty$ (and thus $\frac{\beta}{\beta+1}p=p$), the previous argument simplifies, since we can smuggle in the weight with the trivial estimate $|\cdot|^p\leq\sup_\Omega(\frac{1}{\lambda_\bb})\lambda_\bb|\cdot|^p$ for all $\bb\in\calNN_0$.

\step 3 Estimate on $u_{\e,\eta}$.

We claim that
\begin{equation*}
  \limsup\limits_{\e\downarrow0}\int_A |g-u_{\e,\eta}|^{\frac{\beta}{\beta+1}p}\,dx \leq O(\eta).
\end{equation*}
%
Indeed, due to the definition of $u_{\e,\eta}$ and by the Poincar\'e-Friedrich inequality we get
\begin{eqnarray*}
  \int_A |g-u_{\e,\eta}|^{\frac{\beta}{\beta+1}p}\,dx &=&  \int_A |\e\phi_{\e,\eta}(\tfrac{\cdot}{\e})|^{\frac{\beta}{\beta+1}p}\,dx=\sum_{z\in I}\int_{Q_z}|\e\phi_{\e,\eta,z}(\tfrac{\cdot}{\e})|^{\frac{\beta}{\beta+1}p}\,dx\\
&\leq& C\eta^{\frac{\beta}{\beta+1}p} \sum_{z\in I}\int_{Q_z}|\nabla \phi_{\e,\eta,z}(\tfrac{\cdot}{\e})|^{\frac{\beta}{\beta+1}p}\,dx\\
&=& C\eta^{\frac{\beta}{\beta+1}p} \sum_{z\in I}\e^d\int_{\frac{1}{\e}Q_z}|\nabla \phi_{\e,\eta,z}|^{\frac{\beta}{\beta+1}p}\,dx
\end{eqnarray*}
which combined with \eqref{ieta} and \eqref{L2.11:step2} yields
\begin{eqnarray*}
  \limsup\limits_{\e\downarrow0}\int_A |g-u_{\e,\eta}|^{\frac{\beta}{\beta+1}p}\,dx
&\leq& C \eta^{\frac{\beta}{\beta+1}p} |A|,
\end{eqnarray*}
and thus the asserted estimate.

\step 4 Conclusion.

Set
\begin{equation*}
  f(\e,\eta):=\big|E_\e(u_{\e,\eta},A)-|A|W_0(F)\big|+\|g-u_{\e,\eta}\|_{L^{\frac{\beta}{\beta+1}p}(A)}.
\end{equation*}
Then by Step~1 and Step~3, we have
\begin{equation*}
  \limsup\limits_{\eta\downarrow0}\limsup\limits_{\e\downarrow0}f(\e,\eta)=0.
\end{equation*}
Hence, by Attouch's diagonalization argument there exists a diagonal sequence $\eta=\eta(\e)$ such that $f(\e,\eta(\e))\to 0$ for $\e\downarrow0$ and we deduce that $u_\e:=u_{\e,\eta(\e)}$ defines the sought after recovery sequence.
\end{proof}

\subsection{Compactness and embeddings in weighted spaces: Lemmas~\ref{L:compactness}, \ref{L:interpolation} and ~\ref{L:interpolation2}}\label{S:comp}
We first prove Lemma~\ref{L:interpolation} and Lemma~\ref{L:interpolation2}, and then Lemma~\ref{L:compactness}, which is a variant of Lemma~\ref{L:interpolation2}.

\begin{proof}[Proof of Lemma~\ref{L:interpolation}]
 Since $\beta\geq\frac1{p-1}$, we have $\frac{\beta}{\beta+1}p\geq1$. Hence, by the Gagliardo-Nirenberg-Sobolev inequality and $(1-\frac1\alpha)\frac1q\geq(1+\frac1\beta)\frac1p-\frac1d$ there exists $C<\infty$ such that for all cubes $Q\subset \R^d$ and $v\in W^{1,\infty}(Q,\R^n)$ we have
\begin{align}\label{loc:poinc}
 \left(\fint_Q |v-(\overline v)_Q|^{\frac{{\alpha}}{{\alpha}-1}q}\,dx\right)^{\frac{\alpha-1}{\alpha} \frac1q}\leq C|Q|^{\frac1d}\left(\fint_Q |\nabla v|^{\frac{{\beta}}{{ \beta}+1}p}\,dx\right)^{\frac{ \beta+1}{\beta} \frac1p},
\end{align}
where $(\overline v)_Q:=\fint_{Q}v(x)\,dx$ and the convention $\frac{\alpha-1}{\alpha}=1$ for $\alpha=\infty$ and $\frac{\beta+1}\beta=1$ for $\beta=\infty$. Fix $Q\in\mathcal Q_\e$ and $v\in \calA^\e$. By H\"older's inequality with exponent $\alpha\in(1,\infty)$, we have
\begin{align*}
 &\fint_Q(\sum_{\bb\in\calE_0}\lambda_\bb(\tfrac{x}{\e}))|v(x)-(\overline v)_Q|^q\,dx\leq (\#\calE_0)^{\frac{\alpha-1}{\alpha}}m_{\e,Q,\alpha}\left(\fint_{Q}|v(x)-(\overline v)_Q|^{\frac{\alpha}{\alpha-1}q}\,dx\right)^{\frac{\alpha-1}{\alpha}}.
\end{align*}
From \eqref{est:sumint}, \eqref{loc:poinc}, and a further application of H\"older's inequality with exponent $\beta+1\in(1,\infty)$, we deduce
\begin{align*}
 &\left(\fint_{Q}|v(x)-(\overline v)_Q|^{\frac{\alpha}{\alpha-1}q}\,dx\right)^{\frac{\alpha-1}{\alpha}\frac1q}\leq C|Q|^\frac1d\left(\fint_{Q}|\nabla v(x)|^{\frac{{\beta}}{{\beta}+1}p}\,dx\right)^{\frac{{\beta}+1}{{ \beta }}\frac{1}{p}}\\
 &\leq C|Q|^\frac1d\left(c_0\frac{\e^d}{|Q|}\!\!\sum_{z\in\e\Z^d\cap (Q)_{\e R}}\sum_{\bb\in\calNN_0}|\partial_\bb^\e v(z)|^{\frac{{\beta}}{{\beta}+1}p}\right)^{\frac{{\beta}+1}{{ \beta}}\frac1p}\\
&\leq C{c_0}^{\frac{\beta+1}{\beta}\frac1p}|Q|^\frac1d\left(\tilde m_{\e,Q,\beta}\right)^\frac1p\left(\frac{\e^d}{|Q|}\!\!\sum_{z\in\e\Z^d\cap (Q)_{\e R}}\sum_{\bb\in\calNN_0}\lambda_\bb(\tfrac{z}{\e})|\partial_\bb^\e v(z)|^p\right)^\frac1{p}.
\end{align*}
Inequality \eqref{ineq:interpolation} follows from the previous two estimates.\\
If $\alpha=\infty$ or $\beta=\infty$, we may put in the weights not by appealing to H\"older's inequality, but with help of the elementary pointwise estimates $|\cdot|^q\lambda_\bb\leq |\cdot|^q\sup_\Omega \lambda_\bb$ for all $\bb\in\calE_0$ and $|\cdot|^q\leq \sup_\Omega (\frac{1}{\lambda_\bb})\lambda_\bb|\cdot|^q$ for all $\bb\in\calNN_0$.
\end{proof}

\begin{proof}[Proof of Lemma~\ref{L:interpolation2}]
Fix $\omega\in\Omega_0$ and $A'\Subset A$.  We show that \eqref{ineq:interpolation} implies \eqref{claim:interpolation} by a two-scale argument similar to the one in the proof of Lemma~\ref{L:affine}: We introduce an additional length scale $\eta\in (0,1)$ (with $\e\ll \eta\ll \delta$), and cover $A'$ with cubes of side-length $\eta$. For $i\in\Z^d$ set $Q_i:=\eta(i+Y)$ and define
 $$I:=\{i\in\Z^d\,:\,Q_i\cap A'\}\neq\emptyset.$$
Moreover, we denote by $Q^\e_i$ the smallest cube in $\mathcal Q_\e:=\{\,[a,b)\,:\,a,b\in\e\Z^d\,\}$ which contains $Q_i$. Choosing $\eta$ sufficiently small we have $\cup_{i\in I}(Q_i^\e)_\e \Subset A$, and thus 
\begin{equation*}
 \# I\leq |A|\eta^{-d}.
\end{equation*}
We claim that for every $f\in L^1(\Omega)$ we have
\begin{align}\label{birk:variant}
 \forall i\in I\,:\,\limsup_{\e\downarrow0}\e^d\!\!\!\!\sum_{z\in\e\Z^d\cap (Q_i^\e)_{\e R}} f(\tau_{\frac{z}{\e}}\omega)\leq \eta^d\mathbb E[f].
\end{align}
Indeed, for given $\rho>0$, we have $(Q_i^\e)_{\e R}\subset (Q_i)_\rho$ for $\e>0$ sufficiently small and thus by Theorem~\ref{T:Birk} that
\begin{align*}
 \limsup_{\e\downarrow0}\e^d\!\!\!\!\sum_{z\in\e\Z^d\cap (Q_i^\e)_{\e R}} f(\tau_{\frac{z}{\e}}\omega)\leq \lim_{\e\downarrow0}\e^d\!\!\!\!\sum_{z\in\e\Z^d\cap (Q_i)_\rho} f(\tau_{\frac{z}{\e}}\omega)=(\eta+2\rho)^d\mathbb E[f].
\end{align*}
Since $\rho>0$ can be chosen arbitrarily small, we get \eqref{birk:variant}. Since $\#I<\infty$, we have $\limsup_{\e\downarrow0}\sup_{i\in I} c_{\e,i}\leq \sup_{i\in I}\limsup_{\e\downarrow0}c_{\e,i}$ for $c_{\e,i}:=\e^d\sum_{z\in\e\Z^d\cap (Q_i^\e)_{\e R}} f(\tau_{\frac{z}{\e}}\omega)$ and thus \eqref{birk:variant} yields
\begin{align}\label{birk:variant2}
  \limsup_{\e\downarrow0}\sup_{i\in I}\e^d\!\!\!\!\sum_{z\in\e\Z^d\cap (Q_i^\e)_{\e R}} f(\tau_{\frac{z}{\e}}\omega)\leq \eta^d\mathbb E[f].
\end{align}

Since $A'\subset\bigcup_{i\in I}Q_i^\e$, we have with $\overline u_{\e,i}:=\fint_{Q_i^\e} u_\e(x)\,dx$ the estimate
\begin{align*}
   &\left(\int_{A'}\Big(\sum_{\bb\in\calE_0}\lambda_\bb(\tfrac{x}{\e})\Big)|u_\e(x)|^q\,dx\right)^{\frac{1}{q}}
    \leq \left(\sum_{i\in I}|\overline u_{\e,i}|^q\int_{Q^\e_i}\Big(\sum_{\bb\in\calE_0}\lambda_\bb(\tfrac{x}{\e})\Big)\right)^\frac1q \\ &\qquad\qquad\qquad\qquad\qquad\qquad+\left(\sum_{i\in I}\int_{Q_i^\e}\Big(\sum_{\bb\in\calE_0}\lambda_\bb(\tfrac{x}{\e})\Big)|u_\e(x)-\overline u_{\e,i}|^q\,dx\right)^{\frac1q}  ,
\end{align*}
%

Since $u_\e\wto u$ weakly in $L^1$ we have $\overline u_{\e,i}\to \fint_{Q_i}u\,dx$. Moreover, \eqref{birk:variant} implies 
  $\int_{Q^\e_i}\Big(\sum_{\bb\in\calE_0}\lambda_\bb(\tfrac{x}{\e})\Big)\to \eta^d\sum_{\bb\in\calE_0}\mathbb E[\lambda_\bb]$. The combination of both yields
  \begin{equation*}
   \limsup\limits_{\e\downarrow 0} \sum_{i\in I}|\overline u_{\e,i}|^q\int_{Q^\e_i}\Big(\sum_{\bb\in\calE_0}\lambda_\bb(\tfrac{x}{\e})\Big)
    \leq  \sum_{\bb\in\calE_0}\mathbb E[\lambda_\bb]\,\int_A |u|^q\,dx.
  \end{equation*}
  We claim that
  \begin{equation}\label{ineq:00001}
     \limsup\limits_{\e\downarrow 0} \sum_{i\in I}\int_{Q_i^\e}\Big(\sum_{\bb\in\calE_0}\lambda_\bb(\tfrac{x}{\e})\Big)|u_\e(x)-\overline u_{\e,i}|^q\,dx\leq C\eta^{q+d\min\{ 1-\frac{q}{p},0\}},
  \end{equation}
  where here and below $C$ denotes a finite constant that might change from line to line but can be chosen independent of $\e$ and $\eta$. Note that the right-hand side vanishes as $\eta\downarrow 0$, since the exponent is always positive thanks to our assumption $\frac{1}{q}> \frac{1}{p}-\frac{1}{d}$. Hence, the proof of the lemma follows from the combination of \eqref{ineq:00001} and the previous estimates by taking the limit $\eta\downarrow 0$. We start our argument for \eqref{ineq:00001} with an application of Lemma~\ref{L:interpolation}, which combined with \eqref{birk:variant2} and the moment condition \eqref{ass:X0}, yields
  \begin{equation}\label{ineq:00002}
    \begin{aligned}
      & \limsup\limits_{\e\downarrow 0} \sum_{i\in I}\int_{Q_i^\e}\Big(\sum_{\bb\in\calE_0}\lambda_\bb(\tfrac{x}{\e})\Big)|u_\e(x)-\overline u_{\e,i}|^q\,dx\\
      &\leq C\eta^{q+d(1-\frac{q}{p})}\,\mathbb E[(\sum_{\bb\in\calE_0}\lambda_\bb)^\alpha]^{\frac{q}{\alpha}}\,\mathbb E[\sum_{\bb\in\calNN_0}(\frac{1}{\lambda_\bb})^\beta]^{\frac{q}{p\beta}}\\
      &\qquad\qquad\times  \limsup\limits_{\e\downarrow 0} \sum_{i\in I}\left(\e^d\!\!\!\!\sum_{z\in\e\Z^d\cap (Q_i^\e)_{\e
              R}}\sum_{\bb\in\calNN_0}\lambda_\bb(\tfrac{z}{\e})|\partial_\bb^\e
          u_\e(z)|^p\right)^{\frac{q}{p}}\\
      &\leq
      C\eta^{q+d(1-\frac{q}{p})}\limsup_{\e\downarrow0}\underbrace{\sum_{i\in
          I}\left(\e^d\!\!\!\!\sum_{z\in\e\Z^d\cap (Q_i^\e)_{\e
              R}}\sum_{\bb\in\calNN_0}\lambda_\bb(\tfrac{z}{\e})|\partial_\bb^\e
          u_\e(z)|^p\right)^{\frac{q}{p}}}_{=:(\star)},
    \end{aligned} \end{equation}
  with the convention that $\mathbb E[(\mu)^\gamma]^{\frac{1}{\gamma}}=\sup_\Omega \mu$ for $\gamma=\infty$.
  Note that
  \begin{eqnarray*}
    (\star)&\leq& (\# I)^{\max\{1-\frac{q}{p},0\}} \left(\e^d\sum_{i\in I}\sum_{z\in\e\Z^d\cap (Q_i^\e)_{\e R}}\sum_{\bb\in\calNN_0}\lambda_\bb(\tfrac{z}{\e})|\partial_\bb^\e
                  u_\e(z)|^p\right)^{\frac{q}{p}}\\
           &\leq& C\eta^{-d\max\{1-\frac{q}{p},0\}} \left(\e^d\sum_{z\in\e\Z^d\cap A}\sum_{\bb\in\calNN_0}\lambda_\bb(\tfrac{z}{\e})|\partial_\bb^\e
                  u_\e(z)|^p\right)^{\frac{q}{p}}.
  \end{eqnarray*}
  Indeed, the first estimate follows from H\"older's inequality if $q<p$ and from the discrete $\ell^{\frac{q}{p}}-\ell^1$ estimate if $q\geq p$. The second estimate holds due to the fact that (for sufficiently small $\e\ll 1$) any point $z\in \e\Z^d$ is contained in at most $2^d$ cubes $(Q^\e_i)_{\e R}$ with $z\in (Q^\e_i)_{\e R}$, and thus the double sum $\sum_{i\in I}\sum_{z\in\e\Z^d\cap (Q_i^\e)_{\e R}}$ on the right-hand side can be reduced to $\sum_{z\in\e \Z^d\cap A}$. Combined with \eqref{ineq:00002} and \eqref{ass:interpolation} the claimed inequality \eqref{ineq:00001} follows.
\end{proof}

\begin{proof}[Proof of Lemma~\ref{L:compactness}]
Since $(u_\e)$ has finite energy, cf.\ \eqref{ass:coer}, we obtain by Lemma~\ref{L:coercivity} that $u\in W^{1,p}(A,\R^n)$ and $u_\e\wto u$ weakly in $W^{1,\frac{\beta}{\beta+1}p}_{\loc}(A,\R^n)$. Hence, Rellich's Theorem yields strong convergence in $L^q_{\loc}(A,\R^n)$, provided the strict inequality $\frac{1}{q}>(1+\frac{1}{\beta})\frac1p-\frac1d$ holds. We argue that the conclusion remains valid in the critical case by appealing to Lemma~\ref{L:interpolation2}.  We first notice that \eqref{ass:on:q} still implies $u\in L^q(A,\R^n)$ by the Gagliardo-Nirenberg-Sobolev inequality. Hence, for any $\delta>0$ we find $v\in C^1(\R^d,\R^n)$ with
\begin{equation*}
  \int_A|u-v|^q\,dx\leq \delta.
\end{equation*}
We would like to apply Lemma~\ref{L:interpolation2} with $\alpha=\infty$. In order to do so we need to modify the weight functions: Set $\overline \lambda_\bb:=\min\{\lambda_\bb,1\}$ for $\bb\in\calNN_0$ and $\overline \lambda_\bb:=1+ \min\{\lambda_\bb,1\}$ for $\bb\in\calE_0\setminus\calNN_0$. Without loss of generality we might assume that $\calE_0\setminus \calNN_0\neq\emptyset$, otherwise we simply add an additional edge.
Note that:
\begin{itemize}
\item The weights $\overline \lambda_\bb$ satisfy the moment conditions \eqref{ass:X0} with $\overline \alpha=\infty$, $\overline \beta= \beta$ and we have $\sum_{\bb\in\calE_0}\mathbb E[\overline \lambda_\bb]\leq 2\#\calE_0$.
\item From $\bar\lambda_\bb\leq\lambda_\bb$ for all $\bb\in\calNN_0$ and \eqref{ass:coer} we deduce that 
  \begin{equation*}
    \limsup_{\e\downarrow0}\e^d\!\!\!\!\sum_{z\in\e\Z^d\cap A}\sum_{\bb\in\calNN_0}\overline \lambda_\bb(\tfrac{z}\e)|\partial_\bb^\e u_\e (z)-\partial_\bb^\e v(z)|^p<\infty.
  \end{equation*}
\item We have
  \begin{equation*}
    1\leq \sum_{\bb\in\calE_0\setminus\calNN_0}\overline\lambda_\bb\qquad\text{and}\qquad \sum_{\bb\in\calE_0}\mathbb E[\overline\lambda_\bb]\leq 2\#\calE_0.
  \end{equation*}
\end{itemize}
Hence, by Lemma~\ref{L:interpolation2} applied with exponents $p,\bar\alpha,\bar\beta$, we get
\begin{eqnarray*}
 \limsup_{\e\downarrow0}\int_{A'}|u_\e(x)-u(x)|^q\,dx&\leq& \limsup_{\e\downarrow0}\int_{A'}|u_\e(x)-v(x)|^q\,dx+\delta\\
 &\leq& \limsup_{\e\downarrow0}\int_{A'}|u_\e(x)-v(x)|^q (\sum_{\bb\in\calE_0}\overline \lambda_\bb(\tfrac{x}{\e}))\,dx + \delta\\
 &\stackrel{\eqref{claim:interpolation}}{\leq}&(2\#\calE_0) \int_A|u(x)-v(x)|^q\,dx+\delta\leq (2\#\calE_0+1)\delta.
\end{eqnarray*}
Since $\delta>0$ is arbitrary, the conclusion follows.
\end{proof}

\subsection{Gluing construction:  Lemmas~\ref{L:glue} and \ref{L:glue:convex}}

\begin{proof}[Proof of Lemma~\ref{L:glue}]
We adapt the classical averaging argument of De Giorgi \cite{DG75}, see also \cite[Proof of Lemma~2.1(b)]{Mueller87}, to the present discrete and degenerate setting. The sought after functions $v_\e$ differ from $u_\e$ only close to the boundary of $A$ in a thin layer of thickness proportional to $\delta$. In this layer $v_\e$ is defined (with help of suitable cut-off functions $\phi^k$) as a convex combination of $\overline u$ and $u_\e$, see Step~1 below. In the main part of the proof (see Step~2 and Step~3) we estimate the (asymptotic) energy increment $\limsup_{\e\downarrow 0}E_\e(v_\e,A)-\limsup_{\e\downarrow 0}E_\e(u_\e,A)$, e.g. by appealing to Lemma~\ref{L:interpolation2}. 
\smallskip

Without loss of generality we may assume that
\begin{equation*}
  E=\limsup_{\e\downarrow0} E_\e(u_{\e},A)<\infty.
\end{equation*}
Fix parameters $\delta>0$ and $m\in\N$. In the following  $\lesssim$ stands for $\leq$ up to a multiplicative constant which is independent of $\delta,\e,k,m$ and $\overline u$.

\step 1  Construction of the modified sequence.

For $k=0,\ldots,m$ we introduce the sets
\begin{eqnarray*}
    A^k&:=&\Big\{\,x\in A\,:\,\dist(x,\partial
    A)>\delta\tfrac{2m-k}{2m}\,\Big\},\\
    A_+^k&:=&\Big\{\,x\in A\,:\,\dist(x,\partial
    A)>\delta \tfrac{2m-k-\tfrac{1}{4}}{2m}\,\Big\},\\
    A^k_-&:=&\Big\{\,x\in A\,:\,\dist(x,\partial
    A)>\delta \tfrac{2m-k+\tfrac{1}{4}}{2m}\,\Big\}.
  \end{eqnarray*}
Note that $(A)_{-\delta }=A^0\Subset A_+^0\Subset A_-^1\Subset A^1\Subset A_+^1\Subset\cdots\Subset A_+^m\Subset (A)_{-\frac{\delta}{4}}$. Moreover, the minimal distance between the boundary of consecutive sets, say $A_-^1$ and $A^1$, is at least $\frac{\delta }{8m}$. For $k=0,\dots,m-1$, we consider scalar functions $\phi^k\in C_0^\infty(A)$ such that
$$0\leq \phi^k\leq1,\quad\phi^k=\begin{cases}1&\mbox{in $A_+^k$}\\0&\mbox{in $\R^d\setminus A_-^{k+1}$}                                                  
                                                      \end{cases},\quad \|\nabla \phi^k\|_{L^\infty(A)}\lesssim \frac{m}{\delta}.$$
We define $u_\e^k\in\calA^\e$ by 
$$u_\e^k(x):=\overline u(x)+\phi^k(x)(u_{\e}(x)-\overline u(x)),\quad\mbox{$x\in \e\calL$}.$$
Note that for all $\bb\in\calE_0$ and for all $\e>0$ sufficiently small compared to $\frac{\delta}m$ we have
\begin{equation}\label{uek:decomp}
   \partial_\bb^\e u_\e^k(z)=\begin{cases}
                              \partial_\bb^\e u_\e(z)&\mbox{for $z\in\e\Z^d\cap A^k$}\\
                              \partial_\bb^\e \overline u(z)&\mbox{for $z\in\e\Z^d\setminus A^{k+1}$}.
                             \end{cases}
\end{equation}
Furthermore, thanks to  $\|\nabla \phi^k\|_{L^\infty(A)}\lesssim \frac{m}{\delta}$ and the discrete product rule (recall that $\bb=[x_\bb,y_\bb]$):
$$\partial_\bb^\e(uv)(z)=v(z+\e x_\bb)\partial_\bb^\e u(z)+u(z+\e y_\bb)\partial_\bb^\e v(z),$$
we have
\begin{align}\label{uek:est}
  |\partial_\bb^\e u_{\e}^k(z)|\lesssim&|\partial_\bb^\e \overline u(z)|+\frac{m}{\delta}|(u_{\e}-\overline u)(z+\e x_\bb)|+|\partial_\bb^\e u_{\e}(z)|.
\end{align}
We claim that it suffices to prove
\begin{align}\label{L:glue:claim1}
 \limsup_{\e\downarrow0}\frac{1}{m}\sum_{k=0}^{m-1}\left(E_\e(u_\e^k,A)-E\right)\lesssim& \frac{E}{m}+|A\setminus A^0|(1+\|\nabla \overline u\|_{L^\infty((A)_1)}^p)\notag\\
 &+\frac{1}{m}\|\tfrac{m}{\delta}(u-\overline u)\|_{L^p(A)}^p.
\end{align}
 Indeed, from \eqref{L:glue:claim1} and $A^0=(A)_{-\delta}$ we conclude that $\limsup_{\e\downarrow0}\frac{1}{m}\sum_{k=0}^{m-1}E_\e(u_\e^k,A)$ is bounded by the right-hand side of \eqref{L:glue:st0}. Since $\frac{1}{m}\sum_{k=0}^{m-1}E_\e(u_\e^k,A)$ is an arithmetric mean of non-negative numbers, we find a sequence $(\hat k_\e)$ satisfying $\hat k_\e\in \{1,\dots,m-1\}$ and $E_\e(u_\e^{\hat k_\e},A)\leq \frac1m\sum_{k=0}^{m-1}E_\e(u_\e^k,A)$ for every $\e>0$. Thus, the claim follows by $u_\e^k=\overline u$ on $\R^d\setminus (A)_{-\frac{\delta}{4}}$ and $\|u_\e^k-\overline u\|_{L^q(A)}\leq \|u_\e-\overline u\|_{L^q(A)}$ for every $\e>0$, $q\in[1,\infty]$ and $k\in\{0,\dots,m-1\}$.


\step{2} Estimates on $A^k$ and $A\setminus A^{k+1}$.\\
The starting point for the proof of \eqref{L:glue:claim1} is a decomposition of an energy difference: For $S^k:=A^{k+1}\setminus A^k$ we have:
\begin{align*}
 &\sum_{k=0}^{m-1}\left(E_\e(u_\e^k,A)-E_\e(u_\e,A)\right)\\
 &\stackrel{\eqref{uek:decomp}}{=}\sum_{k=0}^{m-1}\left(E_\e(u_\e,A^k)+E_\e(u_\e^k,S^k)+E_\e(\overline u,A\setminus A^{k+1})-E_\e(u_\e,A)\right).
\end{align*}
Thanks to $V_\bb\geq0$ and $A^k\subset A$ we have $E_\e(u_\e,A^k)-E_\e(u_\e,A)\leq 0$. Combined with the definition of $E$ we get 
\begin{equation*}
 \limsup_{\e\downarrow0}\sum_{k=0}^{m-1}(E_\e(u_\e^k,A)-E)\leq \limsup_{\e\downarrow0}\sum_{k=0}^{m-1}\left(E_\e(u_\e^k,S^k)+E_\e(\overline u,A\setminus A^{k+1})\right).
\end{equation*}
Since $\overline u\in W_\loc^{1,\infty}(\R^d,\R^n)$, we have for $\e>0$ sufficiently small
\begin{equation}
\sup_{z\in\e\Z^d\cap A}\sup_{\bb\in\calE_0}|\partial_\bb^\e \overline u(z)|\leq \|\nabla \overline u\|_{L^\infty((A)_1)}.\label{eq:L:glue:est0001}
\end{equation}

Hence, we obtain by \eqref{ass:V:1}, $A^0\subset A^{k+1}$ and Theorem~\ref{T:Birk} 
\begin{align}\label{L:glue:est1}
\limsup\limits_{\e\downarrow0}E_\e(\overline u,A\setminus A^{k+1})\leq&\lim\limits_{\e\downarrow0}\e^d\!\!\!\!\!\!\sum_{z\in\e\Z^d\cap (A\setminus A^0)}\sum_{\bb\in\calE_0}{c_1}\Big(1+\lambda_\bb(\tfrac{z}{\e})(\|\nabla \overline u\|_{L^\infty((A)_1)}^p+1)\Big)\notag\\
=& |A\setminus A^0|\sum_{\bb\in\calE_0}{c_1}\Big(1+\mathbb E[\lambda_\bb](\|\nabla \overline u\|_{L^\infty((A)_1)}^p+1)\Big).
\end{align}

\step{3} Estimates on $S^k$. 

We claim that
\begin{equation}\label{L:glue:claim2}
  \limsup\limits_{\e\downarrow0}\sum_{k=0}^{m-1}E_\e(u_{\e}^k,S^k)\lesssim E+|A\setminus A^0|(1+\|\nabla \overline u\|_{L^\infty((A)_1)}^p) +\|\tfrac{m}{\delta}(u-\overline u)\|_{L^p(A)}^p. 
\end{equation}
By \eqref{ass:V:1} and \eqref{uek:est}, we have
\begin{align}\label{L:glue:est2}
 E_\e(u_{\e}^k,S^k)\lesssim&\e^d\!\!\!\!\sum_{z\in\e\Z^d\cap S^k}\sum_{\bb\in\calE_0}\bigg(1+\lambda_\bb(\tfrac{z}{\e})\big(1+|\partial_\bb^\e \overline u(z)|^p+|\partial_\bb^\e u_{\e}(z)|^p\notag\\
 &\qquad\qquad\qquad+\frac{m^p}{\delta^p}|(u_{\e}-\overline u)(z+\e x_\bb)|^p\big)\bigg).
\end{align}
Notice that
\begin{equation}\label{slices}
 S^i\cap S^j=\emptyset\mbox{ if $i\neq j$ and }\bigcup_{k=0}^{m-1}S^k\subset A^m\setminus A^0.
\end{equation}
Hence, similar calculations as for \eqref{L:glue:est1} yield
\begin{equation}\label{L:glue:est3}
 \limsup_{\e\downarrow0}\sum_{k=0}^{m-1}\e^d\!\!\!\!\sum_{z\in\e\Z^d\cap S^k}\sum_{\bb\in\calE_0}\left(1+\lambda_\bb(\tfrac{z}\e\right) \left(1+|\partial_\bb^\e \overline u(z)|^p)\right)
\lesssim |A\setminus A^0| (1+\|\nabla \overline u\|_{L^\infty((A)_1)}^p).
\end{equation}
Next, we estimate the term involving $|\partial_\bb^\e u_\e|^p$ in \eqref{L:glue:est2}. Using \eqref{ass:V:1}, \eqref{slices} and Theorem~\ref{T:Birk}, we obtain
\begin{equation}\label{L:glue:est5}
\begin{split}
&\limsup\limits_{\e\downarrow0}\sum_{k=0}^{m-1}\e^d\!\!\!\!\sum_{z\in\e\Z^d\cap S^k}\sum_{\bb\in\calE_0}\lambda_\bb(\tfrac{z}\e)|\partial_\bb^\e u_{\e}(z)|^p\\
&\qquad\lesssim\limsup\limits_{\e\downarrow0}\sum_{k=0}^{m-1}\e^d\!\!\!\!\sum_{z\in\e\Z^d\cap S^k}\sum_{\bb\in\calE_0}\left(V_\bb(\tfrac{z}\e;\partial_\bb^\e u_{\e}(z))+\lambda_\bb(\tfrac{z}\e)\right)\\
&\qquad\lesssim \limsup\limits_{\e\downarrow0}E_\e(u_{\e},A)+|A\setminus A^0|=E+|A\setminus A^0|.
\end{split}
\end{equation}
Finally, we consider the term involving $u_\e-\overline u$. We claim that
\begin{equation}\label{L:glue:claim3}
 \limsup_{\e\downarrow0}\sum_{k=0}^{m-1}\e^d\!\!\!\!\!\sum_{z\in\e\Z^d\cap S^k}\sum_{\bb\in\calE_0}\lambda_\bb(\tfrac{z}{\e})|(u_{\e}-\overline u)(z+\e x_\bb)|^p\lesssim \|u-\overline u\|_{L^p(A)}^p.
\end{equation}
Indeed, by the definition of $\calA^\e$ there exists a constant $c_4>0$ depending only on $\calT$ and $1\leq q<\infty$ (and not on $\e>0$) such that for all $u\in\calA^\e$ and $z\in\e\Z^d$ 
\begin{equation}\label{lpequiv}
      \e^d\max_{x\in z+\e  Y}|u(x)|^p\leq c_4\int_{z+\e Y}|u(x)|^p\,dx.
\end{equation}
Hence, from \eqref{lpequiv}, \eqref{slices}, $A^m\Subset (A)_{-\frac{\delta}4}$, $x_\bb\in Y$ (see~\eqref{def:E0}), and Lemma~\ref{L:interpolation2} (with $q=p$) we obtain 
\begin{align*}
 &\limsup_{\e\downarrow0}\sum_{k=0}^{m-1}\e^d\!\!\!\!\!\sum_{z\in\e\Z^d\cap S^k}\sum_{\bb\in\calE_0}\lambda_\bb(\tfrac{z}{\e})|(u_{\e}-\overline u)(z+\e x_\bb)|^p\\
 &\qquad\lesssim \limsup_{\e\downarrow0}\int_{(A)_{-\frac{\delta}{4}}}\bigg(\sum_{\bb\in\calE_0}\lambda_\bb(\tfrac{x}{\e})\bigg)|u_\e(x)-\overline u(x)|^p\,dx\\
 &\qquad\lesssim  \int_A |u(x)-\overline u(x)|^p\,dx.
\end{align*}
Combining \eqref{L:glue:est2}, \eqref{L:glue:est3}, \eqref{L:glue:est5} and \eqref{L:glue:claim3}, we obtain \eqref{L:glue:claim2}.
\end{proof}

\begin{proof}[Proof of Lemma~\ref{L:glue:convex}] 
  Without loss of generality we may assume that 
  \begin{equation*}
    E=\limsup_{\e\downarrow0} E_\e(u_{\e},A)<\infty.
  \end{equation*}
  Fix parameters $\delta>0$, $m,M\in\N$ and $s>0$. In the following we write $\lesssim$ if $\leq$ holds up to a multiplicative constant which is independent of $\delta,\e,m,M,\overline u,s$ and the additional index $k$ that will show up in the construction below.
  \smallskip

  The argument is a modification of the proof of Lemma~\ref{L:glue}. It additionally invokes a truncation argument that truncates peaks at level $s>0$ of the scalar function $\overline u-u_\e$: We define  $w_\e\in\calA^\e$ via
  \begin{equation*}
    w_\e(x)=\max\{\min\{s,u_\e(x)-\overline u(x)\},-s\},\quad x\in\e\calL.
  \end{equation*}
 In contrast to Lemma~\ref{L:glue} the estimate of $\limsup_{\e\downarrow 0}E_\e(v_\e,A)-\limsup_{\e\downarrow 0}E_\e(u_\e,A)$ is not based on Lemma~\ref{L:interpolation2}, but exploits the good interplay between the truncation and Assumption~\ref{A:2T} (B), see Step~3 below.
  \smallskip

  \step{1} Construction of the modified sequences.

  For $k=0,\dots,m$ let $A_-^k,A,A_+^k$ and $\phi^k$ be defined as in the proof of Lemma~\ref{L:glue}. Let $u_\e^k\in\calA^\e$ be defined by $u_\e^k(x)=\overline u(x)+\phi^k(x)w_\e(x)$ for $x\in\e\calL$, and set $O_M:=\sum_{\bb\in\calE_0}\mathbb E[\lambda_\bb\chi_{\{\lambda_\bb>M\}}]$. We claim that it suffices to show that
  \begin{align}\label{L:glue:convex:final}
    &\limsup\limits_{\e\downarrow0}\sum_{k=0}^{m-1}\left(E_\e(u_\e^k,A)-E\right)\lesssim E+ m|A\setminus A^0|(1+\|\nabla \overline{u}\|_{L^\infty((A)_1)}^p+(\tfrac{ms}{\delta })^p)\notag\\
    &\quad+m(\|\nabla \overline{u}\|_{L^\infty((A)_1)}^p+\|\nabla \overline{u}\|_{L^\infty((A)_1)}^q)(\tfrac{M}s\|u_\e-\overline u\|_{L^1(A)}+O_M|A|)\notag\\
    &\quad+m(E+|A|)^\frac{p}{q}\left(\tfrac{M}s\|u_\e-\overline u\|_{L^1(A)}+O_M|A|\right)^\frac{p-q}{q}.
  \end{align}
  Indeed, combining the argument below \eqref{L:glue:claim1} in the proof of Lemma~\ref{L:glue} and the definition of $h$ in the statement of Lemma~\ref{L:glue:convex}, we arrive at \eqref{L:glue:st0:convex}. 
  \smallskip

  As in the proof of Lemma~\ref{L:glue}, we prove \eqref{L:glue:convex:final} by splitting the energy into its contributions associated with $A\setminus A^k$ and $A^k$.

\step{2} Estimate on $A\setminus A^k$. 

We claim that
\begin{align}\label{L:glue:convex:aohneak}
\limsup\limits_{\e\downarrow0}\sum_{k=0}^{m-1}E_\e(u_\e^k,A\setminus A^k)\lesssim E+m|A\setminus A^0|(1+\|\nabla \overline u\|_{L^\infty((A)_1)}^p+(\tfrac{ms}{\delta })^p).
\end{align}
The argument is similar to the proof of Lemma~\ref{L:glue}. In particular, on the subdomain $A\setminus A^{k+1}$ the argument remains unchanged, since the sequences coincide on this set -- as can be seen by comparing \eqref{uek:decomp} with the identity
  \begin{equation*}
    \partial_\bb^\e u_\e^k(z)=\begin{cases}
      \partial_\bb^\e \overline u(z)+\partial_\bb^\e w_\e(z)&\mbox{for $z\in\e\Z^d\cap A^k$},\\
      \partial_\bb^\e \overline u(z)&\mbox{for $z\in\e\Z^d\setminus A^{k+1}$},
    \end{cases}
  \end{equation*}
  which holds for all $\e>0$ sufficiently small. On $S^k:=A^{k+1}\setminus A^k$ we replace \eqref{L:glue:est2} with the estimate
  \begin{equation}\label{L:glue:est2:scalar}
    E_\e(u_{\e}^k,S^k) \lesssim \e^d\!\!\!\!\sum_{z\in\e\Z^d\cap S^k}\sum_{\bb\in\calE_0}\left(1+\lambda_\bb(\tfrac{z}{\e})\big(1+|\partial_\bb^\e \overline u(z)|^p+|\partial_\bb^\e u_{\e}(z)|^p+\tfrac{m^ps^p}{\delta^p}\big)\right),
  \end{equation}
  the proof of which we postpone to the end of this step. Then, as in the proof of Lemma~\ref{L:glue} the asserted inequality \eqref{L:glue:convex:aohneak} follows from \eqref{L:glue:est1}, \eqref{L:glue:est3}, \eqref{L:glue:est5} and
  \begin{align*}
    \limsup\limits_{\e\downarrow0}\sum_{k=0}^{m-1}\e^d\!\!\!\!\!\sum_{z\in\e\Z^d\cap S^k}\sum_{\bb\in\calE_0}\frac{m^ps^p}{\delta^p}\lambda_\bb(\tfrac{z}{\e})\leq \left(\frac{ms}{\delta}\right)^p|A\setminus A^0|\sum_{\bb\in\calE_0}\mathbb{E}[\lambda_\bb],
  \end{align*}
  which follows from Birkhoff's ergodic theorem. Finally, we prove \eqref{L:glue:est2:scalar}. Note that by construction we have
  \begin{equation}\label{eq:nablatrunc}
    \forall z\in\e\Z^d\,\forall \bb\in\calE_0\,\exists t\in[0,1]\,:\,\partial_\bb^\e w_\e(z)=t(\partial_\bb^\e u_\e(z)-\partial_\bb^\e\overline u(z)).
  \end{equation}
  Combined with  $\|\nabla \phi^k\|_{L^\infty(\R^d)}\lesssim\frac{m}{\delta}$ we get
  \begin{equation*}
    |\partial_\bb^\e u_{\e}^k(z)|\lesssim|\partial_\bb^\e \overline u(z)|+\frac{m}{\delta}|w_\e(z+\e x_\bb)|+|\partial_\bb^\e w_\e(z)|
  \lesssim |\partial_\bb^\e \overline u(z)|+|\partial_\bb^\e u_{\e}(z)|+\frac{ms}{\delta},
  \end{equation*}
  which together with the growth condition \eqref{ass:V:1} yields \eqref{L:glue:est2:scalar}.

\step{3} Estimate on $A^k$.

We claim that
\begin{align*}
  &\limsup\limits_{\e\downarrow0}E_\e(u_\e^{k},A^k)-E\notag\\
  &\lesssim (1+\|\nabla \overline u\|_{L^\infty((A)_1)}^p+\|\nabla \overline u\|_{L^\infty((A)_1)}^q)(\tfrac{M}s\|u-\overline u\|_{L^1(A)}+O_M|A|)\notag\\
  &\quad+\left(E+|A|\right)^\frac{q}{p}\left(\tfrac{M}s\|u-\overline u\|_{L^1(A)}+O_M|A|\right)^{\frac{p-q}{p}}.
\end{align*}
For the proof define the indicator function
\begin{equation*}
 \chi_\e(z,\bb):=\begin{cases}
                  0&\mbox{if $\partial_\bb^\e w_\e(z)=\partial_\bb^\e(u_\e(z)-\overline u(z))$},\\
                  1&\mbox{if $\partial_\bb^\e w_\e(z)\neq\partial_\bb^\e(u_\e(z)-\overline u(z))$},
                 \end{cases}
\end{equation*}
and note that the claim follows from the following three estimates: 
\begin{subequations}
  \begin{align}
    \label{L:glue:scalar:est001}
    &E_\e(u^k_\e,A^k) - E_\e(u_\e,A^k)\\
    \notag &\qquad\quad\lesssim \e^d\!\!\!\!\!\sum_{z\in\e\Z^d\cap
      A^k}\sum_{\bb\in\calE_0}\chi_\e(z,\bb)\Big(1+\lambda_\bb(\tfrac{z}{\e})\big(1+|\partial_\bb^\e\overline u(z)|^p
    +|\partial_\bb^\e\overline u(z)|^q+|\partial_\bb^\e u_\e(z)|^q\big)\Big),\\
    \label{L:glue:scalar:est003}
    &\limsup_{\e\downarrow0}\e^d\!\!\!\!\!\sum_{z\in\e\Z^d\cap A^k}\sum_{\bb\in\calE_0}\chi_\e(z,\bb)\Big(1+\lambda_\bb(\tfrac{z}{\e})\big(|\partial_\bb^\e\overline u(z)|^p+|\partial_\bb^\e\overline u(z)|^q\big)\Big)\\
    \notag&\qquad\quad  \lesssim (1+\|\nabla \overline u\|_{L^\infty((A)_1)}^p+\|\nabla \overline u\|_{L^\infty((A)_1)}^q)(\tfrac{M}s\|u-\overline u\|_{L^1(A)}+O_M|A|),\\
    \label{L:glue:scalar:est004}
    &\limsup_{\e\downarrow0}\e^d\!\!\!\!\!\sum_{z\in\e\Z^d\cap
      A^k}\sum_{\bb\in\calE_0}\chi_\e(z,\bb)\lambda_\bb(\tfrac{z}{\e})|\partial_\bb^\e
    u_\e(z)|^q \\
    \notag&\qquad\quad\lesssim
    \left(E+|A|\right)^\frac{q}{p}\left(\tfrac{M}s\|u-\overline
      u\|_{L^1(A)}+|A|C_m\right)^\frac{p-q}{q}.
  \end{align}
\end{subequations}
{\it Argument for \eqref{L:glue:scalar:est001}.} 
Since $\partial_\bb^\e\phi^k=0$ on $A_k$ we have $\partial_\bb^\e u^k_\e=\partial_\bb^\e\overline u+\partial_\bb^\e
w_\e$ on $A_k$. Hence, thanks to the definition of $\chi_\e$ we have
\begin{equation*}
  \partial_\bb^\e u^k_\e= (1-\chi_\e) \partial_\bb^\e
  u_\e+\chi_\e \partial_\bb^\e u_\e^k
  \qquad\text{on }\e\Z^d\cap A_k\text{ for all }\bb\in\calE_0,
\end{equation*}
and thus
\begin{equation}\label{eq:glue:scalar:est001:1}
  \Big(V_\bb(\tfrac{z}\e;\partial_\bb^\e u_\e^k(z)) -V_\bb(\tfrac{z}\e;\partial_\bb^\e u_\e(z))\Big)=\chi_\e\Big(V_\bb(\tfrac{z}\e;\partial_\bb^\e u_\e^k(z)) - V_\bb(\tfrac{z}\e;\partial_\bb^\e u_\e(z))\Big).
\end{equation}
 Furthermore, \eqref{eq:nablatrunc} yields the representation 
\begin{eqnarray*}
  \partial_\bb^\e u_\e^k &=& \Big(t\partial_\bb^\e u_\e+(1-t)\partial_\bb^\e\overline u\Big),
\end{eqnarray*}
for some $t=t(z,\bb)\in[0,1]$. Hence, by convexity and non-negativity of $W_\bb:=V_\bb+f_\bb$  we obtain (for all $z\in\e\Z^d\cap A_k$):
\begin{eqnarray*}
  V_\bb(\tfrac{z}\e;\partial_\bb^\e u_\e^k(z)) &\leq&  W_\bb(\tfrac{z}\e;\partial_\bb^\e u_\e^k(z)) \leq \Big(\,t W_\bb(\tfrac{z}\e;\partial_\bb^\e u_\e(z))+(1-t)W_\bb(\tfrac{z}\e;\partial_\bb^\e\overline u(z))\Big)\\
  &\leq& \Big(\,V_\bb(\tfrac{z}\e;\partial_\bb^\e u_\e(z))+f_\bb(\tfrac{z}{\e};\partial_\bb^\e u_\e(z))+W_\bb(\tfrac{z}\e;\partial_\bb^\e\overline u(z)\Big).
\end{eqnarray*}
We combine this estimate with \eqref{eq:glue:scalar:est001:1} and appeal to \eqref{ass:V:1} and Assumption~\ref{A:2T}~(B). We get
\begin{eqnarray*}
  &&V_\bb(\tfrac{z}\e;\partial_\bb^\e u_\e^k(z)) -V_\bb(\tfrac{z}\e;\partial_\bb^\e u_\e(z))\\
  &\leq& \chi_\e\Big(\,f_\bb(\tfrac{z}{\e};\partial_\bb^\e u_\e(z))+W_\bb(\tfrac{z}\e;\partial_\bb^\e\overline u(z)\Big)\\
  &\lesssim&\chi_\e\Big(\,
  1+\lambda_\bb(\tfrac{z}{\e})\big(1+|\partial_\bb^\e\overline u(z)|^p+|\partial_\bb^\e\overline u(z)|^q+|\partial_\bb^\e u_\e(z)|^q\big)\Big),
\end{eqnarray*}
and \eqref{L:glue:scalar:est001} follows by summation.

{\it Argument for \eqref{L:glue:scalar:est003}.} By definition of $w_\e$ we have 
\begin{equation*}
  \chi_\e(z,\bb)\leq \frac{|(u_\e-\overline u)(z+\e x_\bb)|+|(u_\e-\overline u)(z+\e y_\bb)|}{s}.
\end{equation*}
For $\e>0$ sufficiently small, the endpoints of any edge $z+\bb$ with $z\in\e\Z^d\cap A^k$ and $\bb\in\calE_0$ are contained in $A$. Hence,
\begin{align}\label{L:glue:scalar:2}
  \e^d\!\!\!\!\!\sum_{z\in\e\Z^d\cap A^k}\sum_{\bb\in\calE_0}\chi_\e(z,\bb)\stackrel{\eqref{lpequiv}}{\lesssim}\tfrac{1}{s}\|u_\e-\overline u\|_{L^1(A)}.
\end{align}
We combine this estimate with another truncation argument: Let $\chi_{\bb,M}:\Omega\to\{0,1\}$ denote the indicator function of $\{\,\lambda_\bb\leq M\}$ and set $\chi_{\bb,M}(x):=\chi_{\bb,M}(\tau_{\lfloor x\rfloor}\omega)$. Then \eqref{L:glue:scalar:2} and Birkhoff's Theorem~\ref{T:Birk} applied to the random variable $\sum_{\bb\in\calE_0}\lambda_\bb\chi_{\bb,M}$ yield
\begin{align}\label{L:glue:scalar:1}
  &\limsup_{\e\downarrow0}\e^d\!\!\!\!\!\sum_{z\in\e\Z^d\cap A^k}\sum_{\bb\in\calE_0}\lambda_\bb(\tfrac{z}{\e})\chi_\e(z,\bb)\\
  \notag&\quad\lesssim \limsup_{\e\downarrow0}\left(\tfrac{M}s\|u_\e-\overline u\|_{L^1(A)}+\e^d\!\!\!\!\!\sum_{z\in\e\Z^d\cap A^k}\sum_{\bb\in\calE_0}\lambda_\bb(\tfrac{z}{\e})\chi_{\bb,M}(\tfrac{z}{\e})\right)\\
  \notag&\quad=\tfrac{M}s\|u-\overline u\|_{L^1(A)}+O_M|A|.
\end{align}
Now, \eqref{L:glue:scalar:est003} follows by combining the previous two estimates with the general bound \eqref{eq:L:glue:est0001}.

{\it Argument for \eqref{L:glue:scalar:est004}.} Since $1<q<p$, we have by H\"older's inequality:
\begin{align*}
 &\e^d\!\!\!\!\!\sum_{z\in\e\Z^d\cap A^k}\sum_{\bb\in\calE_0}\chi_\e(z,\bb)\lambda_\bb(\tfrac{z}{\e})|\partial_\bb^\e u_\e|^q\\
 &\leq \left(\e^d\!\!\!\!\sum_{z\in\e\Z^d\cap A^k}\sum_{\bb\in\calE_0}\lambda_\bb(\tfrac{z}{\e})|\partial_\bb^\e u_\e(z)|^p\right)^\frac{q}{p}\left(\e^d\!\!\!\!\!\sum_{z\in\e\Z^d\cap A^k}\sum_{\bb\in\calE_0} \lambda_\bb(\tfrac{z}{\e})\chi_\e(z,\bb)\right)^\frac{p-q}{p}\\
 &\stackrel{\eqref{ass:V:1}}{\lesssim} \left(E_\e(u_\e,A)+|A|\right)^\frac{q}{p}\left(\e^d\!\!\!\!\!\sum_{z\in\e\Z^d\cap A^k}\sum_{\bb\in\calE_0}\lambda_\bb(\tfrac{z}{\e})\chi_\e(z,\bb)\right)^\frac{p-q}{q}.
\end{align*}
Hence \eqref{L:glue:scalar:1} yields
\begin{align*}
  &\limsup_{\e\downarrow0}\e^d\!\!\!\!\!\sum_{z\in\e\Z^d\cap A^k}\sum_{\bb\in\calE_0}\lambda_\bb(\tfrac{z}{\e})|\partial_\bb^\e u_\e(z)|^q\chi_\e(z,\bb) \\
  &\qquad\qquad\lesssim \left(E+|A|\right)^\frac{q}{p}\left(\tfrac{M}s\|u-\overline u\|_{L^1(A)}+O_M|A|\right)^\frac{p-q}{q}.
\end{align*}
\end{proof}

\subsection{Continuity of $W_0$: Proposition~\ref{P:condwhom}}

\begin{proof}[Proof of Proposition~\ref{P:condwhom}]
  Let $F\in\R^{n\times d}$ and let $(F^N)\subset\R^{n\times d}$ denote an arbitrary sequence converging to $F$. We denote by $g$ and $g^N$ the linear functions defined by $g(x):=Fx$ and $g^N(x):=F^Nx$. We claim that
  \begin{equation}\label{w0cont}
    W_0(F)\leq \liminf_{N\uparrow\infty}W_0(F^N)\leq \limsup_{N\uparrow\infty}W_0(F^N)\leq W_0(F).
  \end{equation}
  Since $F$ and $(F^N)$ are arbitrary, this implies continuity of $W_0$. We split the proof of \eqref{w0cont} into three steps.  

  \step 1 Proof of the lower bound.

  We claim that $W_0(F)\leq \liminf_{N\uparrow\infty}W_0(F^N)$. For the proof fix $\omega\in\Omega_0\cap\Omega_F\cap\bigcap_{N\in\N}\Omega_{F^N}$, cf.~Lemma~\ref{L:indwhom} and Remark~\ref{R:Omega0}. For given $N\in\N$, we find by Lemma~\ref{L:affine} a sequence $(u_\e^N)$ such that 
  \begin{equation}\label{whomcontlb}
    \lim_{\e\downarrow0}\|u_\e^N-g^N\|_{L^{\frac{\beta}{\beta+1} p}(Y)}=0,\quad\lim_{\e\downarrow0}E_\e(\omega;u_\e^N,Y)=W_0(F^N).
  \end{equation}
  We compare $W_0(F^N)$ and $W_0(F)$ on the level of the energy functional $E^\e(\omega;\cdot,Y)$. For this purpose we represent $W_0(F)$ via \eqref{lim1e} and we need to fit the boundary values of $u^N_\e$ to the condition in the minimization problem \eqref{lim1e}. This is done  by appealing to the gluing constructions of Lemma~\ref{L:glue} and Lemma~\ref{L:glue:convex}. We first discuss the vectorial case: assume Assumption~\ref{A:2T}~(A). An application of Lemma~\ref{L:glue} to $u_\e=u^N_\e$, $u=g^N$ and $\overline u=g$ (with $N$ fixed), shows that there exists a constant  $C<\infty$ such that for $\delta>0$ sufficiently small and $m\in\N$ we have a sequence $v_\e\in\calA^\e_g(Y)$ such that
\begin{multline*}
  \limsup_{\e\downarrow0}E_\e(\omega;v_\e,Y)  -(1+\tfrac{C}m)\limsup_{\e\downarrow0}E_\e(\omega;u_\e^N,Y)\\
  \leq C|Y\setminus (Y)_{-\delta}|(1+|F|^p)+C\|\tfrac{m}{\delta}(g^N-g)\|_{L^{p}(Y)}^p.
\end{multline*}
Thanks to \eqref{lim1e} and \eqref{whomcontlb} this implies
\begin{align*}
  W_0(F)-(1+\tfrac{C}m)W_0(F^N)\leq\, C|Y\setminus (Y)_{-\delta}|(1+|F|^p)+C\|\tfrac{m}{\delta}(g^N-g)\|_{L^{p}(Y)}^p.
\end{align*}
Taking successively the limits superior $N\uparrow\infty$, $m\uparrow\infty$ and $\delta\downarrow0$, the asserted lower bound follows. In the scalar case, i.e.\ under Assumption~\ref{A:2T}~(B), we proceed similarly: From Lemma~\ref{L:glue:convex} and \eqref{lim1e} we deduce that there exists $C<\infty$ such that for $\delta>0$ sufficiently small, $m,M\in\N$ and $s>0$ it holds:
\begin{align*}
   &W_0(F)-(1+\tfrac{C}m)W_0(F^N)\\
   &\stackrel{\eqref{lim1e}}{=}\lim_{\e\downarrow0}\inf\left\{E_\e(\omega;\varphi,Y)\ |\ \varphi\in\calA^\e_g(Y)\right\}-(1+\tfrac{C}m)\lim_{\e\downarrow0}E_\e(\omega;u_\e^N,Y)\\
   &\leq C|Y\setminus (Y)_{-\delta}|(h(|F|)+(\tfrac{ms}{\delta})^p)+Ch(|F|)(\tfrac{M}s\|(g^N-g)\|_{L^{1}(Y)}+O_M)\\
   &\qquad+C(W_0(F^N)+1)^\frac{q}{p}(\tfrac{M}s\|(g^N-g)\|_{L^{1}(Y)}+O_M)^\frac{p-q}{p}.
\end{align*}
Taking successively the limits superior $N\uparrow\infty$, $s\downarrow0$, $m\uparrow\infty$, $M\uparrow\infty$ (recall $\lim\limits_{M\uparrow\infty}O_M=0$) and $\delta\downarrow0$, we obtain $W_\ho(F)\leq \liminf\limits_{N\uparrow\infty}W_\ho(F^N)$.

\step 2 Proof of the upper bound.

The upper bound estimate $\limsup_{N\uparrow\infty}W_0(F^N)\leq W_0(F)$ follows by interchanging the roles of $F$ and $F^N$ in the argument of Step~2: Indeed, by Lemma~\ref{L:affine}, we find a sequence $(u_\e)$ such that
\begin{equation*}
\lim_{\e\downarrow0}\|u_\e-g\|_{L^{\frac{\beta}{\beta+1} p}(Y)}=0,\quad\lim_{\e\downarrow0}E_\e(\omega;u_\e,Y)=W_0(F).
\end{equation*}
Using Lemma~\ref{L:glue} and \eqref{lim1e}, we find
\begin{align*}
   &W_0(F^N)-(1+\tfrac{C}m)W_0(F)\\
   &=\lim_{\e\downarrow0}\left\{E_\e(\omega;\varphi,Y)\ |\ \varphi\in\calA^\e_g(Y)\right\}-(1+\tfrac{C}m)\lim_{\e\downarrow0}E_\e(\omega;u_\e,Y)\\
   &\leq C|Y\setminus (Y)_{-\delta}|(1+|F^N|^p)+\|\tfrac{m}{\delta}(g^N-g)\|_{L^{p}(Y)}^p.
\end{align*}
Thus, taking successively the limit superior $N\uparrow\infty$, $m\uparrow\infty$ and $\delta\downarrow0$, we obtain the asserted upper bound estimate. Obviously, we can show the same inequality in the scalar case by appealing to Lemma~\ref{L:glue:convex}.
\end{proof}

\subsection{Lower bound: Proposition~\ref{P:inf}}

\begin{proof}[Proof of Proposition~\ref{P:inf}]
  Let $(u_\e)$ and $u\in W^{1,p}(A,\R^n)$ be such that $u_\e\wto u$ in $L^1(A,\R^n)$. We need to show  that
  \begin{equation}\label{L:inf:claim}
    \liminf_{\e\downarrow0}E_\e(u_\e,A)\geq \int_AW_0(\nabla u(x))\,dx.
  \end{equation}
  By passing to a subsequence, we may assume that 
  \begin{equation}\label{L:inf:sub}
    \liminf_{\e\downarrow0}E_\e(u_\e,A)=\lim_{\e\downarrow0}E_\e(u_\e,A)<\infty.
  \end{equation}
  From \eqref{L:inf:sub} and Lemma~\ref{L:coercivity}, we deduce that $u_\e\wto u$ in $W_\loc^{1,\frac{\beta}{\beta+1} p}(A,\R^n)$.

\step{1} Localization. 

We define a sequence of positive Radon measures $(\mu_\e)$ via 
\begin{equation*}
  \mu_\e:=\mu_\e(\omega;dz):=\e^d\sum_{z\in\e\Z^d}\sum_{\bb\in\calE_0}V_\bb(\tau_{\frac{z}\e}\omega;\partial_\bb^\e u_\e(z))\delta_{z},
\end{equation*}
and note that
\begin{equation}\label{inf:0002}
  \mu_\e (Q)=E_\e(\omega;u_\e,Q)\qquad\text{for any cube }Q\subset\R^d.
\end{equation}
\eqref{L:inf:sub} implies that $\limsup_{\e\downarrow0}\mu_\e(A)<\infty$ and thus there exists a positive Radon measure $\mu$ such that $\mu_\e\wto\mu$ in the sense of measures (up to a subsequence). In particular, it holds
\begin{equation}\label{measureweak}
  \begin{split}
  &\liminf\limits_{\e\downarrow0}\mu_\e (U)\geq \mu(U)\,\mbox{for all open sets $U\subset A$,}\\
 &\lim\limits_{\e\downarrow0}\mu_\e(B)=\mu(B)\,\mbox{for all bounded Borel sets $B\subset A$ with $\mu (\partial B)=0$,}
\end{split}
 \end{equation}
cf.~\cite[Section~1.9]{EG92}. By the Lebesgue decomposition theorem, we can decompose the measure $\mu$ with respect to the Lebesgue measure into its absolutely continuous part and its singular part, that is $\mu=\mu_a+\mu_s$ with $\mu_a\ll dx$ and $\mu_s\perp dx$. Moreover, $\mu_a$ and $\mu_s$ are positive Radon measures and there exists a non-negative $f\in L^1(A)$ with $\mu_a=f\,dx$ and
\begin{equation}\label{inf:f(x)}
 f(x)=\lim_{\rho\downarrow0}\frac{\mu_a(Q_\rho(x))}{\rho^d}=\lim_{\rho\downarrow0}\frac{\mu(Q_\rho(x))}{\rho^d}\quad\mbox{for a.e. }x\in A,
\end{equation}
with $Q_\rho(x):=x+\rho (-\tfrac12,\tfrac12)^d$.

For \eqref{L:inf:claim} it suffices to show that
\begin{equation}\label{fgeqWhom}
 f(x)\geq W_0(\nabla u(x))\quad\mbox{for a.e.\ $x\in A$}.
\end{equation}
Indeed, we have
\begin{equation*}
\liminf_{\e\downarrow0}\mu_\e(A)\stackrel{\eqref{measureweak}}{\geq} \mu(A)\geq\mu_a(A)=\int_Af(x)\,dx\geq \int_A W_0(\nabla u(x))\,dx.
\end{equation*}

\step{2} Approximate differentiability.

We claim that for a.e. $x_0\in A$ we can find a sequence $\rho_j\downarrow 0$ and a sequence of affine functions $(g_j)$ with  $\nabla g_j\equiv F_j\in\mathbb Q^{n\times d}$ and $g_j(x_0)=u(x_0)$ such that
\begin{align}
  \label{inf:diff2}
  &F_j\to \nabla u(x_0)\qquad\text{and}\qquad
  \lim_{j}\frac1\rho_j\left(\fint_{Q_{\rho_j}(x_0)}|u-g_j|^{p}\right)^\frac1{p}=0,\\
  \label{inf:limfxo}
  &f(x_0)=\lim_{j}\lim_{\e\downarrow0}\frac{\mu_\e(Q_{\rho_j}(x_0))}{\rho_j^d}.
\end{align}
This can be seen as follows: Since $u\in W^{1,p}(A,\R^n)$, by approximate differentiability (cf.~\cite[Section~6.1]{EG92})  there exists a set $S\subset A$ of measure zero such that for all $x_0\notin S$ we have
\begin{equation}\label{inf:diff}
 \lim_{\rho\downarrow0}\frac1\rho\left(\fint_{Q_\rho(x_0)}|u(y)-u(x_0)-\nabla u(x_0)(y-x_0)|^{p}\,dy\right)^\frac1{p}=0.
\end{equation}
Evidently, in the line above we might replace $\nabla u(x_0)$ by a suitable sequence $(F_\rho)\subset\mathbb Q^{n\times d}$ converging to $\nabla u(x_0)$. Hence, the affine function
\begin{equation*}
  g_\rho(y):=u(x_0)+F_\rho(y-x_0).
\end{equation*}
satisfies \eqref{inf:diff2} (for any sequence $\rho_j\downarrow 0$).

Next, we prove \eqref{inf:limfxo}. 
Without loss of generality we may assume that for $x\in A\setminus S$ we have \eqref{inf:f(x)}, \eqref{inf:diff} and $f(x)<\infty$. Fix $x_0\in A\setminus S$. Since $A$ is open there exists $\rho_0$ such that $Q_\rho(x_0)\Subset A$ for all $\rho\in(0,\rho_0)$. Moreover, $\mu(A)<\infty$ implies $\mu(\partial Q_\rho(x_0))=0$ for almost every $\rho\in(0,\rho_0)$. In particular, we can find a sequence $\rho_j\downarrow 0$ with $\mu(\partial Q_{\rho_j}(x_0))=0$. Now,  \eqref{inf:limfxo} follows from \eqref{measureweak} and \eqref{inf:f(x)}.

\step{3} Proof of \eqref{fgeqWhom} in the vectorial case.

Suppose Assumption~\ref{A:2T} (A) is satisfied and let $x_0$, $(\rho_j)$, $(g_j)$ and $(F_j)$ be as in Step~2. For convenience, set $Q_j:=Q_{\rho_j}(x_0)$. Since $(F_j)\subset\mathbb Q^{n\times d}$ and $\omega\in\Omega_1$, Corollary~\ref{C:rep_Whom} yields
\begin{equation}\label{inf:0001}
  W_0(F_j)=\lim_{\e\downarrow0}\left(\frac{1}{|Q_j|}\inf\{E_\e(\omega;\varphi,Q_j)\ :\ \varphi\in\calA^\e_{g_j}(Q_j)\,\}\right)\qquad\text{for all }j,
\end{equation}
We apply the gluing Lemma~\ref{L:glue} with $A=Q_j$, $\overline u=g_j$. Hence, for all $\delta=\hat\delta \rho_j$ with $0<\hat \delta \ll1$ and $m\in\N$ there exists a sequence $v_\e\in\calA^\e_{g_j}(Q_j)$ such that
\begin{align*}
  &\limsup_{\e\downarrow0}\frac{E_\e(\omega;v_\e,Q_j)}{|Q_j|}  -(1+\tfrac{C}m)\limsup_{\e\downarrow0}\frac{E_\e(\omega;u_\e,Q_j)}{|Q_j|}\\
  &\qquad \leq C\left(\frac{|Q_j\setminus (Q_j)_{-\hat\delta\rho_j}|}{|Q_j|}(1+|F_j|^p)+\frac{1}{|Q_j|}\|\tfrac{m}{\hat\delta\rho_j}(u-g_j)\|_{L^{p}(Q_j)}^p\right)\\
  &\qquad \leq C\left(2\hat\delta d(1+|F_j|^p)+(\tfrac{m}{\hat\delta})^p\left((\tfrac{1}{\rho_j})^p\fint_{Q_j}|u-g_j|^p\right)\right).
\end{align*}
Above, the constant $C$ is independent of $j$, $m$ and $\hat\delta$.
Hence, since $v_\e\in\calA^\e_{g_j}(Q_j)$ and thanks to \eqref{inf:0001}, \eqref{inf:0002} and \eqref{measureweak} we get
\begin{align*}
  W_0(F_j)-(1+\frac{C}{m})\frac{\mu (Q_j)}{|Q_j|}\leq C\left(2\hat\delta d(1+|F_j|^p)+(\tfrac{m}{\hat\delta})^p\left((\tfrac{1}{\rho_j})^p\fint_{Q_j}|u-g_j|^p\right)\right).
\end{align*}
We successively take the limits $j\to\infty$, $m\uparrow\infty$ and $\hat\delta\downarrow 0$ by appealing to the continuity of $W_0$, \eqref{inf:diff2}, \eqref{inf:limfxo}, \eqref{inf:f(x)} and $f(x_0)<\infty$. This yields
\begin{align*}
 W_0(\nabla u(x_0))-f(x_0)\leq 0.
\end{align*}
Since this is true for a.e. $x_0\in A$, the proof of \eqref{fgeqWhom} is complete.

\step{4} Proof of \eqref{fgeqWhom} in the scalar case.

Suppose that Assumption~\ref{A:2T}~(B) is satisfied. We proceed analogously to Step~3 with the only difference that instead of Lemma~\ref{L:glue} we apply Lemma~\ref{L:glue:convex}. The latter shows that there exists a constant $C$ such that for all $m,M\in\N$ and $s>0$ and $j$ we have
\begin{align*}
 W_0(F_j)-(1+\tfrac{C}{m})\frac{\mu(Q_j)}{|Q_j|} 
 \leq &C\Bigg(\hat \delta d\,\Big(h(|F_j|)+(\tfrac{ms}{\hat\delta \rho_j })^p\Big)
 +h(|F_j|)\big(\tfrac{M}s\fint_{Q_j}|u-g_j|+O_M\big)\\
 &\qquad+\big(\tfrac{\mu(Q_j)}{|Q_j|}+1\big)^\frac{q}{p}\big(\tfrac{M}s\fint_{Q_j}|u-g_j|+O_M\big)^{\frac{p-q}{p}}\,\Bigg).
\end{align*}
We substitute $s=\frac{\rho_j}{m^2}$. Now, the conclusion follows by taking successively the limits superior $j\to \infty$, $m\uparrow\infty$, $M\uparrow\infty$ and $\hat\delta\downarrow0$.
\end{proof}

\subsection{Recovery sequence: Proposition~\ref{P:sup}}

\begin{proof}[Proof of Proposition~\ref{P:sup}] Throughout the proof, we fix $\omega\in\Omega_1$.

\step{1} Recovery sequence with prescribed boundary values for affine functions.

Let $g$ be an affine function and set $F:=\nabla g$. Let $A\subset\R^d$ be a bounded Lipschitz domain. We claim that there exists a sequence $(u_\e)$ with
\begin{equation}\label{eq:sup:0001}
  u_\e\in\calA^\e_g(A),\qquad\lim_{\e\downarrow 0}E_\e(u_\e,A)=|A|W_0(F)\qquad\text{and}\qquad \|u_\e-u\|_{L^{\frac{\beta}{\beta+1}p}(A)}\to 0.
\end{equation}
We first show that we can find a sequence $v_\e\in\calA^\e$ that satisfies \eqref{eq:sup:0001} (but not necessarily coincides with $g$ on the boundary). To that end let $(F_j)\subset\Q^{n\times d}$ be a sequence converging to $F$ and set $g_j(x):=g(0)+F_jx$. By the definition of $\Omega_1$, cf. Remark~\ref{R:Omega0}, and Lemma~\ref{L:affine} there exists for every $j\in\N$ a sequence $u_{\e,j}$ such that $u_{\e,j}\to g_j$ in $L^{\frac{\beta}{\beta+1} p}(A,\R^n)$ and $\lim_{\e\downarrow0}E_\e(u_{\e,j},A)=|A|W_0(F_j)$. Using the continuity of $W_0$ and the uniform convergence of $g_j\to g$, we obtain
\begin{align*}
  \limsup_{j\uparrow\infty}\lim_{\e\downarrow0}\left(|E_\e(u_{\e,j},A)-|A|W_0(F)|+\|u_{\e,j}-g\|_{L^{\frac{\beta}{\beta+1} p}(A)}\right)=0.
\end{align*}
Hence, we obtain the sought after sequence $v_\e$ by passing to a diagonal sequence. 

Next, we use the gluing construction of Lemma~\ref{L:glue} and Lemma~\ref{L:glue:convex} to fit the boundary values. We only discuss the vectorial case, i.e.\ we suppose that Assumption~\ref{A:2T} (A) is satisfied. (The scalar case is similar and left to the reader). Lemma~\ref{L:glue} applied with $u_\e=v_\e$, $\overline u=g$ shows that for any $\delta>0$ sufficiently small, $m\in\N$, there exists $v_\e^{\delta,m}$ with $v_\e^{\delta,m}=g$ in $\R^d\setminus (A)_{-\frac{\delta}4}$ such that the quantity
\begin{equation*}
 f(\e,m,\delta):=\max\left\{0,\left(E_\e(v_\e^{\delta,m},A)-|A|W_0(F)\right)\right\}+\|v_\e^{\delta,m}-g\|_{L^{\frac{\beta}{\beta+1} p}(A)}+\theta_{\delta,\e},
\end{equation*}
with 
\begin{equation*}
  \theta_{\delta,\e}=\begin{cases}
    0&\mbox{if $(A)_{-\frac{\delta}{4}}\subset (A)_{-\e R}$,}\\
    1&\mbox{otherwise.}
  \end{cases} 
\end{equation*}
satisfies
\begin{equation*}
  \limsup_{\delta\downarrow0}\limsup_{m\uparrow\infty}\limsup_{\e\downarrow0}f(\e,m,\delta)\leq 0.
\end{equation*}
Thus there exists a diagonal sequence such that $f(\e,m(\e),\delta(\e))\to0$ as $\e\downarrow0$. We conclude that $u_\e:=v_\e^{\delta(\e),m(\e)}$ satisfies 
\begin{equation*}
  u_\e\in\calA^\e_g(A),\qquad\limsup_{\e\downarrow 0}E_\e(u_\e,A)\leq |A|W_0(F)\qquad\text{and}\qquad \|u_\e-u\|_{L^{\frac{\beta}{\beta+1}p}(A)}\to 0.
\end{equation*}
Combined with the lower bound Proposition~\ref{P:inf}, the claim follows.

\step{2} The general case.

It is sufficient to show the $\limsup$ inequality for piecewise affine functions of the form: $u\in W^{1,\infty}(\R^d,\R^n)$
\begin{equation*}
    u=\sum_{j=1}^M\chi_{A^j}g^j + \chi_{\R^d\setminus A}u,
\end{equation*}
where $A^1,\ldots,A^M$  partition $A$ into disjoint Lipschitz sets and the $g^j$'s are affine functions. The general case follows from the density of piecewise affine functions in $W^{1,p}(A,\R^n)$, the lower semicontinuity of the $\Gamma$-$\limsup$ and the continuity of $W_0$ and \eqref{cond:Whom1}.
\smallskip

To treat the small non-locality of our discrete energy, we introduce an intermediate length scale $0<\rho\ll1$ (which finally will be sent to zero). In the following we assume that $\e>0$ is sufficiently small (relative to $\rho$) such that we have
\begin{equation}\label{limsup:0002}
  j\neq k\quad\Rightarrow\quad ((A^j)_{-\rho})_{\e R}\cap   ((A^k)_{-\rho})_{\e R} =\emptyset.
\end{equation}
For $\e>0$ let $v_\e$ denote the unique (piecewise affine) function in $\calA^\e$ that satisfies $v_\e=u$ on $\e\calL$. Furthermore, we denote by $u_{\e,\rho}^j\in\calA^\e_{g_j}((A_j)_{-\rho})$ the recovery sequence for $g^j$ on the set $(A^j)_{-\rho}$ obtained via Step~1, i.e.\ we have
\begin{equation}\label{limsup:pwaff0}
  \lim_{\e\downarrow0}E_\e(u_{\e,\rho}^{i},(A^i)_{-\rho})=\int_{(A^j)_{-\rho}}W_0(\nabla g^j(x))\,dx\quad\text{and}\quad \lim\limits_{\e\downarrow0}\|u_{\e,\rho}^j-g^j\|_{L^{\frac{\beta}{\beta+1}p}(A)}=0.
\end{equation}
We define $u_{\e,\rho}\in\calA^\e$ via
\begin{equation*}
  u_{\e,\rho}=v_\e +\sum_{j=1}^M\left(u_{\e,\rho}^{j}- g^j\right)\qquad\text{on }\e\calL,
\end{equation*}
and note that $\lim\limits_{\e\to 0}\|u_{\e,\rho}-u\|_{L^{\frac{\beta}{\beta+1}p}(A)}=0$. Furthermore, it can be checked that for $0<\e\ll \rho$ we have (cf. \eqref{limsup:0002}):
\begin{enumerate}[(i)]
\item $v_\e=g_j$ in $((A^j)_{-\rho})_{\e R}$,
\item $u_{\e,\rho}=u_{\e,\rho}^{j}$ in $((A^j)_{-\rho})_{\e R}$,
\item $u_{\e,\rho}=v_\e$ in $(A^j\setminus (A^j)_{-\rho})_{\e R}$.
\end{enumerate}
Thus, as in Remark~\ref{Remark:finiterange} we get the decomposition
\begin{equation*}
E_\e(u_{\e,\rho},A)=\sum_{i=1}^M\bigg(E_\e(u_{\e,\rho}^{i},(A^i)_{-\rho})+E_\e(v_\e,A^i\setminus (A^i)_{-\rho})\bigg).
\end{equation*}
By combining \eqref{limsup:pwaff0} with the estimate
\begin{equation}\label{limsup:pwaff2}
 \limsup_{\rho\downarrow0}\limsup_{\e\downarrow0}E_\e(v_\e,A^i\setminus (A^i)_{-\rho})=0\quad\mbox{for $i=1,\dots,M$},
\end{equation}
(whose proof we postpone to the end of this step), we deduce that
\begin{equation*}
  \limsup_{\rho\downarrow0}\limsup_{\e\downarrow0}\left(\Big|E_\e(u_{\e,\rho},A)-\int_A W_0(\nabla u(x))\,dx\Big|+\|u_{\e,\rho}-u\|_{L^{\frac{\beta}{\beta+1}p}(A)}\right)\leq 0.
\end{equation*}
Hence, the claim follows by passing to  a suitable diagonal sequence.

It remains to prove \eqref{limsup:pwaff2}: For $\e>0$ sufficiently small, we have
$$\sup_{z\in\e\Z^d\cap A}\sup_{\bb\in\calE_0}|\partial_\bb^\e v_\e(z)|\leq \|\nabla v_\e\|_{L^\infty((A)_1)}\leq \|\nabla u\|_{L^\infty((A)_1)}.$$ 
Using \eqref{ass:V:1} and the Ergodic Theorem~\ref{T:Birk}, we obtain
\begin{align*}
 &\limsup_{\e\downarrow0}E_\e(v_\e,A^i\setminus (A^i)_{-\rho})\\
 &\leq \limsup_{\e\downarrow0}\e^d\!\!\!\!\!\!\!\!\sum_{z\in\e\Z^d\cap (A^i\setminus (A^i)_{-\rho})}\sum_{\bb\in\calE_0}{c_1} \left(1+\lambda_\bb(\tfrac{z}{\e})(|\partial_\bb^\e v_\e(z)|^p+1)\right)\\
 &=  |A^i\setminus (A^i)_{-\rho}|\sum_{\bb\in\calE_0}{c_1}\left(1+\mathbb E[\lambda_\bb](\|\nabla u\|_{L^\infty((A)_1)}^p+1)\right).
 \end{align*}
Taking the limit $\rho\downarrow 0$ yields \eqref{limsup:pwaff2}.
\end{proof}

\subsection{Asymptotic formula: Lemma~\ref{L:Whomper}}\label{S:Whomper}
\begin{proof}[Proof of Lemma~\ref{L:Whomper}]
Fix $F\in\R^{n\times d}$. It suffices to prove
\begin{equation}\label{whomper:lim}
 \lim_{k\uparrow\infty}W_\ho^{(k)}(\omega;F)=W_0(F)\qquad\mbox{for all $\omega\in\Omega_F\cap\Omega_0$}.
\end{equation}
Indeed, since any $\phi\in\calA_0^1(kY)$ can be extended to a periodic function in $\calA_\#(kY)$, we have
\begin{equation}\label{whomper:est1}
 W_\ho^{(k)}(\omega;F)\leq \frac{1}{k^d}m_F(\omega;kY).
\end{equation}
The right-hand side converges almost everywhere and in $L^1(\Omega)$ thanks to the subadditive ergodic theorem, cf.~Lemma~\ref{L:indwhom}. Hence, by dominated convergence \eqref{whomper:est1} yields
$$\lim_{k\uparrow\infty}\mathbb E[W_\ho^{(k)}(\cdot;F)-W_0(F)]=0$$
To see \eqref{whomper:lim} note that \eqref{whomper:est1} implies
$$W_0(F)=\lim_{k\uparrow\infty}\frac1{k^d}m_F(\omega;kY)\geq\limsup_{k\uparrow\infty} W_\ho^{(k)}(\omega;F).$$
It remains to show
\begin{equation*}
 \liminf_{k\uparrow\infty}W_\ho^{(k)}(\omega;F)\geq W_0(F)\quad\mbox{for all $\omega\in\Omega_F\cap\Omega_0$}.
\end{equation*}
To that end, we fix $\omega\in\Omega_F\cap\Omega_0$ and choose $\phi_k\in\calA_\#(kY)$ with $\fint_{kY}\phi_k\,dx=0$ such that
\begin{equation}\label{def:phik}
 W_\ho^{(k)}(\omega;F)=\frac1{k^d}\sum_{z\in\Z^d\cap kY}\sum_{\bb\in\calE_0}V_\bb(z;\partial_\bb(g_F+\phi_k)(z)),
\end{equation}
where $g_F$ denotes the linear function $g(x):=Fx$. Note that $\phi_k$ exists, since $\calA_\#(k Y)$ is a finite dimensional space. We claim that
\begin{equation}\label{whomper:compactness}
 \limsup_{k\uparrow\infty}\fint_{kY}|\nabla \phi_k|^{\frac{\beta}{\beta+1}p}\,dx<\infty.
\end{equation}
In the following, we write $\lesssim$ if $\leq$ holds up to a multiplicative constant which is independent of $k$. The $kY$-periodicity of $\phi_k$ combined with similar calculations as in Step~2 of the proof of Lemma~\ref{L:affine} yield for $k\gg R$
\begin{align*}
 &\frac1{k^d}\sum_{z\in\Z^d\cap (kY)_R}\sum_{\bb\in\calNN_0}|\partial_\bb^1 \phi_k(z)|^{\frac{\beta}{\beta+1}p}\lesssim\frac1{k^d}\sum_{z\in\Z^d\cap kY}\sum_{\bb\in\calNN_0}|\partial_\bb^1 \phi_k(z)|^{\frac{\beta}{\beta+1}p}\\
 &\lesssim |F|^{\frac{\beta}{\beta+1}p}+ \frac1{k^d}\sum_{z\in\Z^d\cap kY}\sum_{\bb\in\calNN_0}|\partial_\bb^1 (g_F+\phi_k)(z)|^{\frac{\beta}{\beta+1}p}\\
 &\lesssim|F|^{\frac{\beta }{\beta+1}p}\\
 &\quad+ \left(\frac1{k^d}\!\!\!\sum_{z\in\Z^d\cap kY}\sum_{\bb\in\calNN_0}\!\!\!\!\lambda_\bb(z)^{-\beta}\right)^{\frac{1}{\beta+1}}\!\!\!\left(\frac1{k^d}\!\!\!\!\sum_{z\in\Z^d\cap kY}\sum_{\bb\in\calNN_0}\lambda_\bb(z)|\partial_\bb^1 (g_F+\phi_k)(z)|^p\right)^{\frac{\beta}{\beta+1}}\!\!\!.
\end{align*}
The above inequality, \eqref{def:phik}, \eqref{ass:V:1}  and Theorem~\ref{T:Birk} yield
\begin{align*}
 \limsup_{k\uparrow\infty}\frac1{k^d}\sum_{z\in\Z^d\cap (kY)_R}\sum_{\bb\in\calNN_0}|\partial_\bb \phi_k(z)|^{\frac{\beta}{\beta+1}p}<\infty,
\end{align*}
and thus \eqref{whomper:compactness} by \eqref{est:sumint}.\\
Set $\varphi_k(\cdot)=k^{-1}\phi_k(k\cdot)$, so that
$$W_\ho^{(k)}(\omega;F)=E_\frac1k(\omega;g_F+\varphi_k,Y).$$
The definition of $\phi_k$ implies $\fint_Y\varphi_k\,dx=0$ and $\varphi_k\in W_{\#}^{1,p}(Y,\R^n)$ for all $k\in\N$, where $W_{\#}^{1,p}(Y,\R^n):=\{u\in W_\loc^{1,p}(\R^d,\R^n):\mbox{ $u$ is $Y$-periodic}\}$. Hence, Poincar\'e's inequality and \eqref{whomper:compactness} yield: 
$$\limsup_{k\uparrow\infty}\|\varphi_k\|_{W^{1,\frac{\beta}{\beta+1}p}(Y)}<\infty.$$
Hence, there exists $\varphi\in W_{\#}^{1,p}(Y,\R^n)$ with $\fint_Y \varphi\,dx=0$ such that, up to a subsequence, $\varphi_k\wto \varphi$ in $W^{1,\frac{\beta}{\beta+1}p}(Y,\R^n)$ and thus Proposition~\ref{P:inf} yields 
\begin{equation*}
\liminf_{k\uparrow\infty} W_\ho^{(k)}(\omega;F)=\liminf_{k\uparrow\infty} E_{\frac1k}(\omega;g_F+\varphi_k,Y)\geq \int_Y W_0(F+\nabla \varphi(x))\,dx.
\end{equation*}
The quasiconvexity of $W_0$, the growth condition \eqref{cond:Whom1} and the periodicity of $\varphi$ imply (see \cite[Proposition 5.13]{Dac07})
$$\int_Y W_0(F+\nabla \varphi(x))\,dx\geq W_0(F).$$
Altogether, we have shown  \eqref{whomper:lim}.
\end{proof}

\subsection{Moments under independence: Proposition~\ref{Prop:iid}}

\begin{proof}[Proof of Proposition~\ref{Prop:iid}]

The proof follows very closely the arguments of the proof of \cite[Proposition 3.7]{MO16} where the case $p=2$ is considered, see also \cite[Theorem 1.12]{ADS16}. In the following we write $\ee\in\e\mathbb B^d\cap A$ if $[x,y]=\ee\in \e \mathbb B^d$ and $x,y\in A$.

For every given edge $\ee=[z,z+e_i]\in\mathbb B^d$, it is easy to see that there exist $2d$ disjoint paths $\ell_1(\ee),\dots,\ell_{2d}(\ee)$ in $\mathbb B^d\cap B_4(z)$ connecting $z$ and $z+e_i$ and the length of each path does not exceed $9$. We define random variables $\{\mu(\omega;\ee)\}_{\ee\in\mathbb B^d}$ by
\begin{equation*}
 \mu(\omega;\ee)^{-\frac{p}{p-1}}:=\inf_{1\leq i\leq 2d}\sum_{\bb\in \ell_i(\ee)}\omega(\bb)^{-\frac{1}{p-1}}.
\end{equation*}
By definition, we have $\mu(\omega;\ee)=\mu(\tau_{z_\ee}\omega;\bb_\ee)$, where $z_\ee \in\Z^d$ and $\bb_\ee\in\{e_1,\dots,e_d\}$ are uniquely given by $\ee=z_\ee+\bb_\ee$.
We show that:

\begin{enumerate}
 \item[(i)]  There exists a constant $C=C(d)<\infty$ such that for all $v\in\calA^\e$:
\begin{align}\label{ineq:iid:1}
 &\left(\e^d\!\!\!\!\sum_{z\in\e\Z^d\cap A}\sum_{i=1}^d|\partial_{e_i}^\e v(z)|^{\frac{\beta}{\beta+1}p}\right)^{\frac{\beta+1}{\beta}}\notag\\
 &\leq C\left(\e^d\!\!\!\sum_{z\in\e\Z^d\cap (A)_\e}\sum_{i=1}^d\mu(\tau_{\frac{z}\e}\omega;e_i)^{-\beta p}\right)^\frac1\beta\left(\e^d\sum_{z\in\e\Z^d\cap (A)_{4\e}}\sum_{i=1}^d(\tau_{\frac{z}\e}\omega)(e_i)|\partial_{e_i}^\e v(z)|^p\right).
\end{align}
\item[(ii)] It holds $\mathbb E[\mu(\cdot;\ee)^{-\beta p}]<\infty$ for all $\ee\in\mathbb B^d$.
\end{enumerate}
Obviously, inequality \eqref{ineq:iid:claim} follows from (i) and (ii) with $f(\omega)=C\sum_{i=1}^d\mu(\omega;e_i)^{-\beta p}$.\\

\step{1} Proof of \eqref{ineq:iid:1}.

Fix $\ee=[z,z+e_i]\in\mathbb B^d$. Denote by $\ell(\ee)$ the minimizing path in the definition of $\mu(\omega;\ee)$. The triangular and H\"older inequality yield
\begin{align*}
 |\nabla v(\ee)|\leq& \sum_{\bb\in\ell(\ee)}|\nabla v(\bb)|\leq\left(\sum_{\bb\in\ell(\ee)}{\omega(\bb)}^{-\frac{1}{p-1}}\right)^{\frac{p}{p-1}}\left(\sum_{\bb\in\ell(\ee)}\omega(\bb)|\nabla v(\bb)|^p\right)^\frac1p\\
 =&\mu(\omega;\ee)^{-1}\left(\sum_{\bb\in\ell(\ee)}\omega(\bb)|\nabla v(\bb)|^p\right)^\frac1p.
\end{align*}
Multiply both sides by $\mu(\omega;\ee)$ and take the $p$th power. This shows that there exists $C=C(d)<\infty$ such that for all $v\in\calA^\e$:
\begin{align*}
 \sum_{\ee\in\e\mathbb B^d\cap A}\mu(\omega;\tfrac{\ee}\e)^p|\nabla v(\ee)|^p\leq&\sum_{\ee\in\e\mathbb B^d\cap A}\sum_{\bb\in \ell(\tfrac{\ee}\e)}\omega(\bb)|\nabla v(\e\bb)|^p\leq C\!\!\!\sum_{\ee\in\e\mathbb B^d\cap (A)_{4\e}}\omega(\tfrac{\ee}\e)|\nabla v(\ee)|^p,
\end{align*}
and thus (by H\"older's inequality)
\begin{align*}
 &\left(\e^d\!\!\!\!\sum_{z\in\e\Z^d\cap A}\sum_{i=1}^d|\partial_{e_i}^\e v(z)|^{\frac{\beta}{\beta+1}p}\right)^{\frac{\beta+1}{\beta}}\leq\left(\e^d\!\!\!\!\sum_{\ee\in\e\mathbb B^d\cap (A)_{\e}}|\nabla v(\ee)|^{\frac{{\beta}}{{\beta}+1}p}\right)^\frac{{\beta}+1}{{\beta} }\\
 &\leq C\left(\e^d\!\!\!\sum_{\ee\in\e\mathbb B^d\cap (A)_\e}\mu(\omega;\tfrac{\ee}\e)^{-\beta p}\right)^\frac1\beta\left(\e^d\!\!\!\!\sum_{\ee\in\e\mathbb B^d\cap (A)_{4\e}}\omega(\tfrac{\ee}\e)|\nabla v(\ee)|^p\right)
\end{align*}
which implies \eqref{ineq:iid:1}.

\step{2} Proof of $\mathbb E[\mu(\cdot;\ee)^{-\beta p}]<\infty$ for all $\ee\in\mathbb B^d$.

It suffices to show that 
\begin{equation}\label{ineq:iid:moment}
 \mathbb P[\mu(\cdot;\ee)^{-1}>t]\lesssim t^{-2d\gamma p},
\end{equation}
where here and for the rest of the proof $\lesssim$ means $\leq$ up to a multiplicative constant which depends on $\beta,\gamma,d$ and $p$. Indeed, since $\beta p>0$ and $2d\gamma>\beta$, we have
\begin{align*}
\mathbb E[\mu(\cdot;\ee)^{-\beta p}]=&\frac1{\beta p}\int_0^\infty t^{\beta p-1}\mathbb P[\mu(\ee)^{-1}>t]\,dt
\lesssim1+\int_1^\infty  t^{(\beta-2d\gamma) p-1}\,dt
\lesssim1, 
\end{align*}
since $(\beta-2d\gamma)p<0$. Finally, we prove \eqref{ineq:iid:moment}: By independence of $\{\omega(\ee)\}_{\ee\in\mathbb B^d}$ and disjointness of the paths, we have for $t>0$ that
\begin{equation*}
 \mathbb P[\mu(\cdot;e)^{-1}>t]=\mathbb P\left[\mu(\cdot;e)^{-\frac{p}{p-1}}>t^\frac{p}{p-1}\right]=\prod_{i=1}^{2d}\mathbb P\left[\sum_{\bb\in\ell_i(\ee)}\omega(\bb)^{-\frac1{p-1}}>t^\frac{p}{p-1}\right].
\end{equation*}
Since the length of $\ell_i(\ee)$ is bounded by $9$ and due to stationarity and that $\{\omega(\ee)\}_{\ee\in\mathbb B^d}$ are identically distributed, we obtain
\begin{align*}
\prod_{i=1}^{2d}\mathbb P\left[\sum_{\bb\in\ell_i(\ee)}\omega(\bb)^{-\frac1{p-1}}>t^\frac{p}{p-1}\right]\leq&9\mathbb P\left[\omega(\bb)^{-\frac1{p-1}}>\frac{t^\frac{p}{p-1}}{9}\right]^{2d}\\
 \leq&9\mathbb P\left[\omega(\bb)^{-\gamma}> \frac{t^{\gamma p}}{9^{\gamma(p-1)}}\right]^{2d}\\
 \leq& 9^{1+2d\gamma(p-1)}\mathbb E[\omega(\bb)^{-\gamma}]^{2d}t^{-2d\gamma p}\\
 \lesssim& t^{-2d\gamma p},
\end{align*}
since $\mathbb E[\omega(\bb)^{-\gamma}]<\infty$ for all $\bb\in\mathbb B^d$ by assumption. 
\end{proof}

\section{Acknowledgments}
This work was supported by a grant of the Deutsche Forschungsgemeinschaft (DFG) SCHL 1706/2-1. SN and AS acknowledge the hospitality of the Max-Planck-Institute for Mathematics in the Sciences, Leipzig. Moreover, SN and MS acknowledge the hospitality of the  Weierstrass Institute for Applied Analysis and Stochastics, Berlin. They are also grateful for the support by the DFG in the context of TU Dresden's Institutional Strategy \textit{``The Synergetic University''}. MS is grateful for the support by the University of W\"urzburg. We would like to thank F. Flegel, M. Heida and M. Slowik for stimulating discussions.

\appendix
\section{Proof of Lemma~\ref{L:sumint}, Corollary~\ref{C:1} and Lemma~\ref{L:indwhom}}
\label{appendix}

\begin{proof}[Proof of Lemma~\ref{L:sumint}]
Fix $1\leq q<\infty$. In the following we write $\lesssim$ if $\leq$ holds up to a multiplicative constant that only depends on $(\calL,\calNN)$ and $q$.
By the definition of $\calA^1$ we have
\begin{equation}\label{eq:L:sumint:1}
\forall u\in\calA^1:\quad \int_Y|\nabla u|^q\,dx\lesssim \sum_{x,y\in\calL\cap [0,1]^d \atop x\neq y}|u(y)-u(x)|^q.
\end{equation}
By the definition of $\calNN$ and of $R$ (cf. \eqref{def_of_R}) for any pair $(x,y)\in\calL\cap[0,1]^d$ with $x\neq y$ there exists a path $\ell(x,y)$ in $\calNN\cap B_{\frac{R}{4}}(0)$ that connects $x$ and $y$. Hence,
\begin{align*}
  |u(y)-u(x)|^q&\lesssim \sum_{\ee \in\ell (x,y)}|\nabla u(\ee)|^q\leq\sum_{\ee \in \calNN\cap B_{\frac{R}4}(0)}|\nabla u(\ee)|^q,
\end{align*}
which together with \eqref{eq:L:sumint:1} and the definition of $\calNN$ yields \eqref{est:AB}.
\smallskip

Next, we show \eqref{est:sumint}. By scaling it suffices to consider the case $\e=1$. Since $(A)_{-R}\subset\bigcup_{z\in \Z^d\cap (A)_{-\frac{3}{4}R}}(z+Y)$ (thanks to $R> 4\sqrt d$), estimate \eqref{est:AB} yields
\begin{eqnarray*}
 \int_{(A)_{-R}}|\nabla u|^q\,dx&\leq&
 \sum_{z\in \Z^d\cap(A)_{-\frac{3}{4}R}}\int_{z+Y}|\nabla u|^q\,dx\\
 &\lesssim&
 \sum_{z\in \Z^d\cap(A)_{-\frac{3}{4}R}}\ \sum_{z'\in \Z^d\cap B_{\frac{1}{2}R}(z)}\sum_{\bb\in\calNN_0}|\partial_\bb^1 u(z')|^q\\
 &\lesssim& \sum_{z''\in \Z^d\cap A}\sum_{\bb\in\calNN_0}|\partial_\bb^1 u(z'')|^q,
\end{eqnarray*}
where for the last estimate we used that 
\begin{enumerate}[(i)]
\item $\bigcup_{z\in(A)_{-\frac{3}{4}R}}\Big(B_{\frac{1}{2}R}(z)\Big)\subset A$,
\item $z''\in \Z^d\cap A$  is contained in at most $C\lesssim 1$ sets of the form $\Z^d\cap B_{\frac{1}{2}R}(z)$ with $z\in\Z^d\cap (A)_{-\frac{3}{4}R}$.
\end{enumerate}
\end{proof}

\begin{proof}[Proof of Corollary~\ref{C:1}]

\step{1} Proof of part (a) -- coercivity.\\
Let $(u_\e)$ be a sequence satisfying \eqref{ene:bounded:bc}. We claim that there exists a $u\in g+W^{1,p}(A,\R^n)$ such that, up to a subsequence, $u_\e \wto u$ weakly in $W^{1,\frac{\beta}{\beta+1}p}(A,\R^n)$ and strongly in $L^q(A,\R^n)$. 

Thanks to the boundary condition $u_\e=g$ in $\R^d\setminus (A)_{-\e R}$ and in view of Lemma~\ref{L:compactness}, it suffices to show that
\begin{equation}\label{C:1:eq:000}
  \begin{aligned}
    &\limsup\limits_{\e\downarrow 0}I_\e<\infty\qquad \text{and}\qquad \limsup\limits_{\e\downarrow 0}\tilde E_\e(u_\e)<\infty\\
    &\text{where }I_\e:=\left(\int_{(A)_1}|\nabla
      u_\e|^{\frac{\beta}{\beta+1}p}\,dx\right)^{\frac{\beta+1}{\beta}\frac1p}\qquad\text{and}\qquad
    \tilde E_\e(u_\e):=E_\e(\omega;u_\e,(A)_2).
  \end{aligned}
\end{equation}
The difference of $\tilde E_\e$ and $H_\e$ consists of the energy contribution of edges associated with the boundary layer $(A)_2\setminus (A)_{-\e R}$. Since $u_\e=g$ on $\R^d\setminus (A)_{-\e R}$, we deduce from \eqref{ene:bounded:bc}, that
  \begin{equation}\label{C:1:eq:001}
    \limsup\limits_{\e\downarrow 0}\Big(\tilde E_\e(u_\e)-F_\e(u_\e)\Big)<\infty.
  \end{equation}
 From \eqref{est:coer:u:00001} in the proof of Lemma~\ref{L:coercivity}, we learn that for all sufficiently small $\e>0$ we have
  \begin{equation}\label{C:1:eq:002}
    I_\e^p\leq C\left(\tilde E_\e(u_\e)+1\right),
  \end{equation}
  where here and below $C$ denotes a finite constant that might change from line to line, but can be chosen independent of $\e$.
Furthermore, thanks to \eqref{def:w-conv}, $g\in W^{1,\infty}(\R^d,\R^n)$, and the Poincar\'e-Sobolev inequality applied to $u_\e-g$, we have 
\begin{eqnarray}\label{C:1:eq:003}
  |F_\e(u_\e)| &\leq&   \left(\e^d\!\!\!\sum_{x\in \e\calL\cap A}|f_\e(x)|^{\frac{q}{q-1}}\right)^{\frac{q-1}{q}}\left(\e^d\!\!\!\sum_{x\in \e\calL\cap A}|u_\e(x)|^q\right)^{\frac{1}{q}}\notag\\
  &\leq& C\left(\int_{(A)_1}|u_\e-g|^q\,dx\right)^{\frac{1}{q}}+\left(\int_{(A)_1}|g|^q\,dx\right)^{\frac{1}{q}}\notag\\
  &\leq& C I_\e +1.
\end{eqnarray}
We combine the previous three estimates as follows:
\begin{eqnarray*}
  I_\e^p&\stackrel{\eqref{C:1:eq:002}}{\leq}& C\left(\tilde E_\e(u_\e)+1\right)\\
  &\leq& C\left(\big(\tilde E_\e(u_\e)-F_\e(u_\e)\big)+|F_\e(u_\e)|+1\right)\\
  & \stackrel{\eqref{C:1:eq:003}}{\leq}&C\left(\big(\tilde E_\e(u_\e)-F_\e(u_\e)\big)+I_\e +1\right).
\end{eqnarray*}
Since $p>1$ and thanks to \eqref{C:1:eq:001}, the assertion \eqref{C:1:eq:000} follows.

\step{2} Proof of part (b) -- $\Gamma$-convergence.

By \eqref{def:w-conv} we have
\begin{equation*}
  \lim\limits_{\e\downarrow 0}F_\e(u_\e)=\int_A f\cdot u\,dx
\end{equation*}
for any sequence $u_\e\to u$ in $L^q(A,\R^n)$ with $u_\e\in\calA^\e_g(A)$. Hence, it suffices to prove $\Gamma$-convergence for $H_\e$. We start with the lower bound. Let $(u_\e)$ and $u$ be such that $u_\e\to u$ in $L^{q}(A,\R^n)$. Without loss of generality, we may assume that \eqref{ene:bounded:bc} holds. Then by Step~1, we get $u\in g+W_0^{1,p}(A,\R^n)$. For any $U\Subset A$ we have $H_{\e}(u_\e)\geq E_\e(u_\e,U)$, and thus by Proposition~\ref{P:inf}:
$$\liminf_{\e\downarrow0}H_{\e}(u_\e)\geq\liminf_{\e\downarrow0}E_\e(u_\e,U)\geq \int_UW_\ho(\nabla u(x))\,dx.$$
Taking the supremum over $U\Subset A$ and using $W_\ho\geq0$, the lower bound follows. Next, we prove existence of a recovery sequence. By a standard approximation argument, it is sufficient to consider $u\in g+C_c^\infty(A,\R^n)$. Combining Proposition~\ref{P:sup} and the gluing construction Lemma~\ref{L:glue} and Lemma~\ref{L:glue:convex}, respectively, we obtain a sequence $(u_\e)$ with $u_\e\to u$ in $L^{\frac{\beta}{\beta+1}p}(A,\R^n)$ and $u_\e=u=g$ in $\R^d\setminus (A)_{-\e R}$ such that
$$\lim_{\e\downarrow0}E_\e(u_\e,A)=\lim_{\e\downarrow0}\e^d\!\!\!\!\sum_{z\in\e\Z^d\cap A}\sum_{\bb\in\calE_0}V_\bb(\tau_{\frac{z}\e}\omega;\partial_\bb^\e u_\e(z))=\int_AW_\ho(\nabla u(x))\,dx.$$
By the definition of $R$, cf. Remark~\ref{Remark:finiterange0}, we have $\partial_\bb^\e u(z)=\partial_\bb^\e g(z)$ for all $z\in\e\Z^d\setminus A$ (see \eqref{rem:finiterange}). Combined with the fact that for all $0<\rho<1$ and $\e\ll 1$ we have $H_{\e}(u_\e)\leq E_\e(u_\e,(A)_\rho)$, we deduce that
\begin{align*}
  &\limsup_{\e\downarrow0}H_\e(u_\e)\leq \limsup\limits_{\rho\downarrow 0}\limsup_{\e\downarrow0}E_\e(u_\e,(A)_\rho)=E_\ho(u,A).
\end{align*}
In view of Lemma~\ref{L:compactness} we even have $u_\e\to u$ in $L^q(A,\R^n)$, which concludes the proof of the recovery sequence.
\medskip

The convergence of minima and minimizers is a consequence of the compactness (Step~1) and of $\Gamma$-convergence (Step~2). Note that minimizers of $J_\e$ exist thanks to the coercivity and the continuity of the potentials, and the finite dimensionality of the problem.
\end{proof}

\begin{proof}[Proof of Lemma~\ref{L:indwhom}]
First notice that 
\begin{equation*}
  \inf_{k\in\N}\frac{\mathbb E \left[m_F(\cdot;kY)\right]}{k^d}=\lim_{k\uparrow\infty\atop k\in\N}\frac{\mathbb E \left[m_F(\cdot;kY)\right]}{k^d}
\end{equation*}
holds by monotonicity.
Set $\mathcal I:=\{[a,b)~|~a,b\in\Z^d\}$ and fix $F\in\R^{n\times d}$. We denote by $\mathcal L^1$ the class of integrable functions on $(\Omega,\mathcal F,\mathbb P)$.

\step{1} We claim that 
\begin{equation*}
  m_F:\calI\to\mathcal L^1,\qquad A\mapsto m_F(\cdot,A),
\end{equation*}
defines a subadditive and stationary process (see e.g.~\cite{Krengel}). To that end we need to check the following three properties:
%
\begin{enumerate}[(i)]
\item $m_F$ defines a set function from $\mathcal I$ to $\mathcal L^1$. Indeed, since $V_\bb\geq 0$ and thanks to \eqref{ass:V:1} we have
\begin{align*}
  0\leq \mathbb E[m_F(A)]\leq&|A|\sum_{\bb\in\calE_0}{c_1}(1+\mathbb E[\lambda_\bb](|F|^p+1))<\infty.
\end{align*}
\item $m_F$ is stationary, i.e.\ for every $z\in\Z^d$, $\omega\in\Omega$ and $A\in \mathcal I$ it holds $m_F(\tau_z\omega;A)=m_F(\omega;z+A)$. This is a direct consequence of a change of variable in the infimum problem in the definition of $m_F$ and the stationarity of the interaction potentials $V_\bb$.
\item $m_F(\omega;\cdot)$ is subadditiv. Let $A_1,\dots,A_M\subset \mathcal I$ be disjoint sets and $\bigcup_{i=1}^MA_i=A\in\mathcal I$. For every $A_i$, we find $\phi_{i}\in\calA_{0}^1(A_i)$ such that
$$E_1(\omega;g_F+\phi_{i},A_i) = m_F(\omega;A_i),$$ where $g_F$ denotes the linear function $g_F(x):=Fx$.
We define $\phi:=\sum_{i=1}^M\phi_{i}\in\calA_0^1(A)$. Since $\phi_{i}\in\calA_0^1(A_i)$ for all $i=1,\dots,M$, we have $\phi=\phi_{i}$ on $(A_i)_R$ for $i=1,\dots,M$. Moreover, we have that for every $z\in\Z^d\cap A$ there exists a unique $i\in\{1,\dots,M\}$ such that $z\in A_i$. Hence, subadditivity follows: 
\begin{align*}
  m_F(\omega;A)\leq&E_1(\omega;g_F+\phi,A)=\sum_{i=1}^M\sum_{z\in\Z^d\cap A_i}\sum_{\bb\in\calE_0}V_\bb(\tau_z\omega;\partial_\bb^1g_F(z)+\partial_\bb^1\phi(z))\\
  =&\sum_{i=1}^ME_1(\omega;g_F+\phi_{i},A_i)=\sum_{i=1}^Mm_F(\omega;A_i).
 \end{align*}
\end{enumerate}
  
\step{2} Ergodic limit for cubes in $\calI$.\\
By the previous step, we are in position to apply the Ackoglu-Krengel subadditive ergodic theorem in the version \cite[Theorem 1]{ACG11}, cf. \cite[Theorem 2.9]{AK81}: There exists a set $\Omega_F\subset\Omega$ of full measure such that for all $\omega\in\Omega_F$ and all cubes of the form $Q=[a,b)$ with $a,b\in\Z^d$, we have 
\begin{align*}
 \lim_{k\uparrow\infty\atop k\in\N}\frac{1}{|kQ|}m_F(\omega;kQ)=\lim_{k\uparrow\infty\atop k\in\N}\frac1{k^d}\mathbb E\left[m_F(\cdot;kY)\right]=W_0(F).
\end{align*}

\step{3} General cubes.\\
To pass to the limit along $\R$ and cubes $Q$ satisfying $\bar Q=[a,b]$ with $a,b\in\R^d$, we adapt the arguments of \cite[Proposition 1]{DMM86} (and \cite[Corollary~3.3]{MM94}) to our degenerate and discrete situation. For every $\delta>0$ there exists a dilation $T\geq 1$ and cubes $Q_\delta^-$, $Q_\delta^+\in\calI$ such that 
$$Q_\delta^-\subset TQ\subset Q_\delta^+,\quad \frac{|Q_\delta^-|}{|TQ|}\geq 1-\delta,\quad \frac{|TQ|}{|Q_\delta^+|}\geq 1-\delta.$$
The growth condition \eqref{ass:V:1} on $V_\bb$ yields for all open bounded Lipschitz sets with $B\subset A$ the following inequality
\begin{align*}
 m_F(\omega;A)\leq m_F(\omega;B)+\sum_{z\in\Z^d\cap A\setminus B}\sum_{\bb\in\calE_0}{c_1}(1+\lambda_\bb(\tau_z\omega)(|F|^p+1)).
\end{align*}
For every $t>0$, we set $t^+=\lfloor t\rfloor+1$ and $t^-=\lfloor t \rfloor$. Thanks to the ergodic theorem in Step~2, we have
\begin{align*}
&\lim_{t\uparrow\infty}\frac{1}{|t^+Q_\delta^+|}\sum_{z\in\Z^d\cap (t^+Q_\delta^+\setminus tTQ)}\lambda_\bb(\tau_z\omega)\leq \delta\mathbb{E}[\lambda_\bb],\\
&\lim_{t\uparrow\infty}\frac{1}{|tTQ|}\sum_{z\in\Z^d\cap(tTQ\setminus t^-Q_\delta^-)}\lambda_\bb(\tau_z\omega)\leq \delta\mathbb{E}[\lambda_\bb].
\end{align*}
Thus,
\begin{align*}
 W_0(F)=&\lim_{t^+\uparrow\infty}\frac{m_F(\omega;t^+Q_\delta^+)}{|t^+Q_\delta^+|}\\
 \leq& \liminf_{t\uparrow\infty}\frac{m_F(\omega;tTQ)}{|tTQ|}+\delta\sum_{\bb\in\calE_0}{c_1}(\mathbb E[\lambda_\bb](|F|^p+1)+1)\\
 \leq& \limsup_{t\uparrow\infty}\frac{m_F(\omega;tTQ)}{|tTQ|}+\delta\sum_{\bb\in\calE_0}{c_1}(\mathbb E[\lambda_\bb](|F|^p+1)+1)\\
 \leq&\lim_{t^-\uparrow\infty}\frac{m_F(\omega;t^-Q_\delta^-)}{|t^-Q_\delta^-|}+2\delta\sum_{\bb\in\calE_0}{c_1}(\mathbb E[\lambda_\bb](|F|^p+1)+1)\\
 =&W_0(F)+2\delta\sum_{\bb\in\calE_0}{c_1}(\mathbb E[\lambda_\bb](|F|^p+1)+1).
\end{align*}
The assertion follows by the arbitrariness of $\delta>0$.
\end{proof}

\bibliographystyle{siamplain}

\end{document}